\documentclass[reqno,11pt,a4paper]{amsart}

\usepackage[margin=2cm]{geometry}

\usepackage{amsmath,amsthm,amssymb,amsfonts,dsfont,ifpdf,mathtools,verbatim}
\usepackage{enumitem}
\usepackage{graphicx}
\usepackage[usenames,dvipsnames]{color}
\usepackage[all]{xy}
\usepackage{mathrsfs}

\usepackage{csquotes}
\MakeOuterQuote{"}

\usepackage[colorlinks=true, linkcolor=blue, urlcolor=blue, citecolor=ForestGreen]{hyperref}
\numberwithin{equation}{section}

\newtheorem {thm}    {Theorem}[section]
\newtheorem {lem}      [thm]    {Lemma}

\newtheorem {cor}  [thm]    {Corollary}
\newtheorem {prop}[thm]    {Proposition}
\newtheorem* {prop*} {Proposition}

\newtheorem*{claim*}   {Claim}

\newtheorem*{conj*} {Conjecture}

\theoremstyle{definition}

\newtheorem {rmk}    [thm]    {Remark}
\newtheorem*{rmk*}  {Remark}

\newtheorem*{qst*} {Question}

\newtheorem* {problem*}{Problem}

\newtheorem*{term} {Notation}

\newcounter{AbcT}

\numberwithin{equation}{section}

\newcommand {\C} {{\mathbb C}}

\newcommand {\Hyp} {{\mathbb H}}

\newcommand {\N} {{\mathbb N}}

\newcommand {\R} {{\mathbb R}}
\newcommand {\T} {{\mathbb T}}

\newcommand {\Z} {{\mathbb Z}}

\newcommand {\cC} {{\mathscr C}}

\newcommand {\cF} {{\mathscr F}}

\newcommand {\cM} {{\mathscr M}}
\newcommand {\cN} {{\mathscr N}}

\DeclareMathOperator{\SL}{SL}
\DeclareMathOperator{\PSL}{PSL}
\DeclareMathOperator{\GL}{GL}

\DeclareMathOperator{\SO}{SO}

\DeclareMathOperator{\Spec}{Spec}

\newcommand{\eps}{\varepsilon}

\newcommand {\IGNORE}[1]  {}

\newcommand {\norm}[1] {\left\| {#1} \right\|}

\newcommand {\bsl} {\backslash}

\newcommand {\La} {{\Lambda}}
\newcommand{\Ga}{\Gamma}

\renewcommand{\sl}{\mathfrak{sl}}

\newcommand\vol{\operatorname{vol}}

\newcommand\Ad{\operatorname{Ad}}
\newcommand{\Lie}{\operatorname{Lie}}

\begin{document}
	
	\title[Large hyperbolic circles]{Large hyperbolic circles}
	\author[E.~Corso]{Emilio Corso}
	\author[D.~Ravotti]{Davide Ravotti}
	\address[E.C.]{ETH Z\"urich, R\"amistrasse 101
		CH-8092 Z\"urich
		Switzerland}
	\email{corsoemilio2@gmail.com}
	\address[D.R.]{Universit\"{a}t Wien, Department of Mathematics, Oskar-Morgenstern-Platz 1, 1090 Wien, Austria}
	\email{davide.ravotti@gmail.com}
	\date{\today}
	
	
	\begin{abstract}
		We consider circles of common centre and increasing radius on a compact hyperbolic surface and, more generally, on its unit tangent bundle. We establish a precise asymptotics for their rate of equidistribution. Our result holds for translates of any circle arc by arbitrary elements of $\SL_2(\R)$. Our proof relies on a spectral method pioneered by Ratner and subsequently developed by Burger in the study of geodesic and horocycle flows. We further derive statistical limit theorems, with compactly supported limiting distribution, for appropriately rescaled circle averages of sufficient regular observables. Finally, we discuss applications to the classical circle problem in the hyperbolic plane, following the approach of Duke-Rudnick-Sarnak and Eskin-McMullen.
	\end{abstract}
	\maketitle
	

	\section{Introduction and main results}
	
	\subsection{Dilating sets in diverse geometric contexts}
	\label{sec:intro}
	It is an intriguing geometric problem to understand the long-term distribution properties of a diversified range of progressively dilating shapes, when the ambient space in which they live is folded according to a prescribed rule. Formally, the framework underlying such a question can be laid down as follows. Consider a compact connected Riemannian manifold $S$ and a Riemannian covering space $N$ with covering projection $\pi\colon N\to S$. The manifold $N$ plays the role of the ambient space, while $S$ is to be interpreted as the manifold resulting after a certain folding procedure operated on $N$. A homothety on $N$ is a diffeomorphism $h\colon N\to N$ transforming the Riemannian metric on $N$ into a scalar multiple thereof; if the rescaling ratio is equal to $1$, it simply reduces to a Riemannian isometry. Let now $(h_t)_{t\in \R_{>0}}$ be a collection of homotheties on $N$ whose scaling factor tends to infinity as $t$ does, and fix a Borel probability measure $\mu$ on $N$. The latter should be thought of as describing quantitatively the original shape, whose dilations we wish to examine. For instance, $\mu$ might be the renormalized volume measure on a finite-volume Riemannian submanifold of $N$, e.g.~a rectifiable curve. We then let $\mu_t$ be the push-forward of $\mu$  under the homothety $h_t$ and denote by $m_t$ the projection of $\mu_t$ onto $S$, for any $t>0$.
	
	A mathematical formulation of our problem consists in asking for the limit points, as $t$ goes to infinity, of the family of measures $(m_t)_{t>0}$ in the topology of weak$^*$ convergence of probability measures on $S$. Somewhat less pretentiously, it is already interesting to determine sufficient geometric conditions on the initial measure $\mu$ ensuring that the $m_t$ equidistribute as $t$ grows larger, that is, that they converge in the aforementioned topology to the renormalized volume measure $\vol_S$ on $S$; this circumstance amounts pictorially to the fact that the measures $m_t$ (and hence, in an intuitive sense, their supports) fill up the folded space $S$ in the most uniform way.
	
	\subsubsection*{The Euclidean case}
	
	To the best of the authors' knowledge, the question formulated in the previous paragraph was first asked by Dennis Sullivan (cf.~\cite{Randol}) in the zero-curvature setting of Euclidean spaces and tori, that is, when $S=\T^{d}=\R^{d}/\Z^d$ and $N=\R^{d}$ for some integer $d\geq 1$, and for $\mu$ being the volume measure on a compact lower-dimensional submanifold of $\R^{n}$. In this case, the transformations $h_t$ are the standard linear homotheties $x\mapsto tx,\;x\in \R^{d}$. The problem was originally addressed by Randol in~\cite{Randol}, both in an Euclidean and in a hyperbolic setup. In the former case, the following equidistribution result is established.
	
	\begin{thm}[{cf.~\cite[Thm.~1]{Randol}}]
		\label{thm:Randol}
		Suppose $C$ is the smooth boundary of a compact convex subset of $\R^{d}$ with non-empty interior, and assume its Gaussian curvature is everywhere positive. Let $\mu$ be the volume measure on $C$, renormalized to be a probability measure. Then the probability measures $m_t$ on $\T^{d}$ defined by
		\begin{equation*}
			m_t(A)=\mu(\{x\in \R^{n}:tx+\Z^{d}\in A \})\;, \quad A\subset \T^{d} \emph{ Borel }
		\end{equation*} equidistribute towards the Haar measure on $\T^{d}$ as $t\to\infty$.
	\end{thm} 
	
	\begin{rmk}
		\label{rmk:quantitative}
		\begin{itemize}
			\item [(a)] The result in Theorem~\ref{thm:Randol} is actually quantitative: for every sufficiently regular function $f$ on $\T^{d}$, it holds that
			\begin{equation*}
				\biggl|\int_{\T^{d}}f\;\text{d}m_t-\int_{\T^{d}}f\;\text{d}m_{\T^{d}}\biggr|\ll_{f,d}t^{-(d-1)/2}\;,
			\end{equation*}
			where $m_{\T^d}$ denotes the Haar probability measure on $\T^{d}$.
			\item [(b)] As a special case of Theorem~\ref{thm:Randol}, expanding spheres in $\R^{d}$ equidistribute on the standard torus with a polynomial rate depending on the dimension $d$. The study of expanding spheres in other geometric settings shall be a driving theme of this manuscript. 
			\item [(c)] In~\cite[Thm.~2]{Randol}, equidistribution is generalized to uniform measures supported on lower-dimensional rectilinear simplices in $\R^{d}$.
		\end{itemize}
	\end{rmk}
	
	The proof of Theorem~\ref{thm:Randol} (more generally, of its quantitative version stated in Remark~\ref{rmk:quantitative}(a)) relies on classical Fourier analysis on the $d$-dimensional torus; it is remarkably straightforward, once the decay properties at infinity of the Fourier transform of the measure $\mu$ 
	are known. As such, it has been elaborated upon by Strichartz in~\cite{Strichartz} to prove the following generalization of Theorem~\ref{thm:Randol}. Let us say that the Fourier transform $\hat{\mu}\colon \R^{d}\to \C$ of $\mu$ decays on rays if 
	\begin{equation}
		\label{eq:raydecay}
		\lim\limits_{t\to\infty}\hat{\mu}(tx)=0
	\end{equation}
	for every nonzero vector $x\in \R^{d}$.
	\begin{thm}[{cf.~\cite[Lem.~1]{Strichartz}}]
		\label{thm:Strichartz}
		Let $\mu$ be a Borel probability measure on $\R^{d}$ whose Fourier transform decays on rays. Then the conclusion of Theorem~\ref{thm:Randol} holds true. 
	\end{thm}
	
	It actually suffices that $\hat{\mu}$ decays on integral rays, that is, that~\eqref{eq:raydecay} is verified, less restrictively, for any non-zero $x\in \Z^{d}$. On the other hand, it turns out (cf.~\cite[Lem.~1]{Strichartz}) that decay on arbitrary rays is equivalent to equidistribution of the projections of the $\mu_t$ onto any torus $\R^{d}/\La$, where $\La$ is a lattice in $\R^{d}$.

	To conclude this brief account of the state of the art on the problem in flat geometry,  we mention that the question of the limiting distribution of expanding circles has been recently examined in the setting of translation surfaces by Colognese and Pollicott~\cite{Colognese-Pollicott}, who prove (non-effective) equidistribution towards a probability measure which is equivalent, though in general not proportional, to the area measure on the given surface.

	\subsubsection*{The hyperbolic case}
	As already mentioned, Randol's investigations in~\cite{Randol} were not confined to the zero-curvature case. In the hyperbolic framework, namely when the sectional curvature is constantly equal to $-1$, $S$ is a compact connected hyperbolic $d$-manifold ($d\geq 2$) and $N$ can be taken as its Riemannian universal covering manifold, that is, the $d$-dimensional hyperbolic space $\Hyp^{d}$. A choice of the homotheties $(h_t)_{t>0}$ is determined by fixing a base point $x_0\in \Hyp^{d}$ and letting $h_t$ be the map which transforms\footnote{This produces a well-defined assignment, as $\Hyp^{d}$ is a uniquely geodesic metric space (cf.~\cite[Part I, Chap.~1]{Bridson-Haefliger}) for the distance induced by the hyperbolic Riemannian metric.} each Riemannian geodesic $\gamma(s)$ ($s\in \R$) passing through $x_0$ at time $0$ into the geodesic $\gamma(ts)$. In this context, the result that can be elicited from the discussion in~\cite{Randol} reads as follows.
	\begin{thm}[{cf.~\cite{Randol}}]
		\label{thm:Randolhyp}
		Let $S$ be a compact connected hyperbolic $d$-manifold, $\pi\colon \Hyp^{d}\to S$ the universal covering map, $C$ a $(d-1)$-dimensional hyperbolic  sphere of unit radius centered at a point $x_0\in \Hyp^{d}$, $\mu$ the unique isometrically-invariant\footnote{Here we obviously intend invariance under isometries of $C$ equipped with the induced hyperbolic metric.} Borel probability measure on $C$, $(h_t)_{t>0}$ the family of homotheties $\Hyp^{d}\to \Hyp^{d}$ with center $x_0$ defined as above. Then the probability measures $m_t$ on $S$ defined by
		\begin{equation}
			\label{eq:measureshyperbolic}
			m_t(A)=\mu(\{x\in \Hyp^{d}:\pi\circ h_t(x)\in A \})\;, \quad A\subset S \emph{ Borel}
		\end{equation}
		equidistribute towards the renormalized volume measure on $S$ as $t\to\infty$.
	\end{thm}
	
	\begin{rmk}
		Here again the result takes on a quantitative form: the rate of equidistribution of the measures $m_t$ defined in~\eqref{eq:measureshyperbolic} is exponential, as opposed to the Euclidean case, with the exponent depending on the spectral gap of the hyperbolic manifold $S$ (cf.~\cite{Randol}).
	\end{rmk}
	
	Akin in spirit to the proof of Theorem~\ref{thm:Randol}, the argument leading to Theorem~\ref{thm:Randolhyp} is based upon the harmonic analysis of locally symmetric spaces via techniques related to the Selberg trace formula; for those, the reader is referred to Selberg's original article~\cite{Selberg-trace}.

	\subsubsection*{Lifting the question to unit tangent bundles} 
	
	Let us now consider the following upgraded version of the problem explored in the foregoing paragraphs. Suppose that $S$, $N$ and $(h_t)_{t>0}$ are as in the beginning of the present section, with $y_0\in N$ being the common center of the homotheties $h_t$, and let $C$ be a compact Riemannian hypersurface in $N$.  Assume further that the Riemannian distance function on $N$ turns it into a uniquely geodesic metric space\footnote{This is certainly the case for $N=\R^d$ and $N=\Hyp^{d}$, equipped respectively with the standard Euclidean metric and with the hyperbolic metric.}. If $T^{1}N$ denotes the unit tangent bundle of $N$, then $C$ identifies uniquely the subset $\tilde{C}$ of $T^{1}N$ consisting of all pairs $(x,v)$, where $x$ is a point in $C$ and $v$ is the unit-length tangent vector to the unique geodesic connecting $y_0$ to $x$. Similarly, if $C_t$ indicates the image of $C$ under the homothety $h_t$ for any $t>0$, we denote by $\tilde{C_t}$ its lift to $T^{1}N$ obtained in the previous fashion. A natural question thus arises as to the asymptotic distribution of the $\tilde{C_t}$ when projected to the unit tangent bundle of $S$; in this respect, the natural choice of a measure on $\tilde{C_t}$ is given by the push-forward of the renormalized volume measure on $C_t$ under the canonical identification of the latter submanifold with its lift $\tilde{C_t}$. If the projections to $S$ of the hypersurfaces $C_t$ equidistribute towards the normalized volume measure on $S$, it may be expected that the projections of the lifts $\tilde{C_t}$ equidistribute towards the corresponding Liouville measure on the unit tangent bundle $T^{1}S$ (cf.~\cite[Part 1, Sec.~5.4]{Hasselblatt-Katok}). 
	
	The question lends itself to a description in the language of smooth dynamical systems. If $(g_t^{(N)})_{t\in \R}$ is the geodesic flow on $T^{1}N$ (cf.~\cite[Part 1, Chap.~5]{Hasselblatt-Katok}), it takes a few moments to realize that, for any $t>1$, the set $\tilde{C_t}$ is the image of the original lift $\tilde{C}$ under the transformation $g_{t-1}^{(N)}$; the same relation holds for the natural measures carried by the latter sets, and carries over to their projections to $T^1S$, using the geodesic flow $(g_t^{(S)})_{t\in \R}$ defined on it in place of $(g_t^{(N)})_{t\in \R}$. 
	
	The equidistribution problem in this formulation is treated in Margulis' thesis~\cite{Margulis-thesis}, which contains several striking developments and applications of the theory of Anosov systems to the large-scale geometry of negatively curved manifolds; among those, a proof is provided of equidistribution, towards the Liouville measure, of lifts of expanding circles on finite-volume hyperbolic surfaces. 
	For further comments thereupon, as well as for the connection to the hyperbolic circle problem, the reader is referred to Section~\ref{sec:circleproblemhyperbolic}.
	
	It is the chief purpose of the present work to provide a precise asymptotic expansion for the equidistribution rate of lifts of dilating hyperbolic circles, as well as of arbitrary sub-arcs thereof, on unit tangent bundles of compact hyperbolic surfaces; in the vein of the works of Randol~\cite{Randol} and Strichartz~\cite{Strichartz}, we resort to a spectral approach originating in the work of Ratner~\cite{Ratner} on quantitative mixing of geodesic and horocycle flows on Riemann surfaces of finite volume. Section~\ref{sec:quantitativeequidistribution} describes such results and expands on their connection to previous developments,
		 in particular to the closely related work of Bufetov and Forni~\cite{BuFo},
	 whereas Sections~\ref{sec:introductionCLT} and~\ref{sec:circleproblemhyperbolic} discuss a number of applications to statistical limit theorems and the hyperbolic lattice point counting problem. 
	
	We conclude this historical overview by pointing out that    
	Margulis' groundbreaking contributions in~\cite{Margulis-thesis}, together with the gradual emergence of conspicuous applications to counting and Diophantine problems, spawned intensive research aimed at understanding the asymptotic distribution properties of translates of finite-volume subgroup orbits, as well as of more general subsets, on homogeneous spaces\footnote{It is worth noticing at this point that lifts of expanding hyperbolic circles represent a particular instance, as they are geodesic translates of orbits of the maximal compact subgroup $\SO_2(\R)$ on quotients of $\SL_2(\R)$: see Section~\ref{sec:quantitativeequidistribution}.}; without purporting to provide an exhaustive list, we mention in this direction the works~\cite{Benoist-Oh,Einsiedler-Margulis-Venkatesh,Eskin-McMullen,Eskin-Mozes-Shah,Kra-Shah-Sun,Shah,Shah-second,Shah-third,Shah-fourth,PYang}. 
	
	\subsection{The setup: circles in hyperbolic surfaces and in their unit tangent bundles}
	
	We now set the stage for the main questions we address in the present manuscript, referring to Section~\ref{sec:preliminaries} for the required background.
	Let $\Ga<\SL_2(\R)$ be a cocompact lattice, that is, a discrete subgroup of $\SL_2(\R)$ such that the quotient space $\Ga\bsl \SL_2(\R)$ is compact; we indicate the latter homogeneous space with $M$. The group $\Ga$ acts properly discontinuously and isometrically on the Poincar\'{e} upper half-plane $\Hyp=\{z=x+iy \in \C:y>0 \}$, endowed with the standard hyperbolic Riemannian metric, by M\"{o}bius transformations; when the projection of $\Ga$ to $\PSL_2(\R)=\SL_2(\R)/\{\pm I_2\}$ is torsion-free, the quotient $S=\Ga\bsl \Hyp$ is a compact connected orientable smooth surface, inheriting a hyperbolic metric from $\Hyp$. With respect to such a metric, there is a canonical identification of $M$ with (possibly, a double cover of) the unit tangent bundle\footnote{More precisely, if $\Ga$ is the preimage under the canonical projection map $\SL_2(\R)\to \PSL_2(\R)$ of a cocompact lattice in $\PSL_2(\R)$, then $M$ identifies with $T^{1}S$; else, it is a double cover thereof. 
		
		In the case where the image of $\Ga$ in $\PSL_2(\R)$ has non-trivial torsion elements, then $S$ has the structure of an orbifold (cf.~\cite[Chap.~13]{Ratcliffe}). For the purposes of the paper, we shall never be concerned with this distinction.} $T^{1}S$.
	
	Let $(r_{s})_{s\in \R}$ be the one-parameter flow on $M$ defined by 
	\begin{equation}
		\label{eq:rotationflow}
		r_s(\Ga g)=\Ga g \begin{pmatrix}
			\cos{s/2}&\sin{s/2}\\
			-\sin{s/2}& \cos{s/2}
		\end{pmatrix}
		=\Ga g \exp{s\Theta}, \quad \Theta=\begin{pmatrix}
			0&1/2\\
			-1/2&0
		\end{pmatrix}
		\in \sl_2(\R)=\Lie(\SL_2(\R)),
	\end{equation}
	and denote by $(\phi_t^{X})_{t\in \R}$ the geodesic flow on $M$, which is given algebraically by 
	\begin{equation}
		\label{eq:geodesicflow}
		\phi^{X}_{s}(\Ga g)=\Ga g \begin{pmatrix} e^{t/2}&0\\ 0&e^{-t/2} \end{pmatrix}=\Ga g \exp{tX},\quad X=\begin{pmatrix}
			1/2 &0\\
			0&-1/2 \end{pmatrix}
		\in \sl_2(\R).
	\end{equation}
	
	For any point $p=\Ga g\in M$, the orbit of $p$ under the flow $(r_s)_{s\in \R}$ is the preimage of $z=\pi(p)$ under the fibration $\pi\colon M\to S$. Therefore, if $M$ identifies with $T^{1}S$, then this set consists of all unit tangent vectors to $z\in S$. For any real number  $t>0$, the time-$t$ geodesic evolution  $\phi_t^{X}(\{r_s(p):s\in \R \})$ of the previous set coincides with the projection to $M$ of the subset of $T^{1}\Hyp$ given by all outward-pointing normal vectors to the hyperbolic circle in $\Hyp$ of radius $t$ centered at (a fixed representative in $\Hyp$ of) $z$.
	
	We indicate with $\vol$ the Haar probability measure on $M$, that is, the unique $\SL_2(\R)$-invariant Borel probability measure on $M$; under the identification of $M$ with $T^1S$, it coincides with the Liouville measure projecting to the normalized hyperbolic area measure on $S$. For any $r\in \N\cup \{\infty \}$, we denote by $\mathscr{C}^{r}(M)$ the set of complex-valued functions of class $\cC^{r}$ defined on the smooth manifold $M$. The supremum norm of a continuous function $f\colon M\to \C$ is denoted by $\norm{f}_{\infty}$. For any $m\in \N_{\geq 1}$, $f\in \cC^{m}(M)$ and $j\in \{0,\dots,m\}$ let $\nabla^{j}f$ be the $j^{\text{th}}$ covariant derivative of $f$ and $|(\nabla^{j}f)(p)|$ its norm at a point $p \in M$. Define then the $\cC^{m}$-norm of $f$ (cf.~\cite[Chap.~1]{Aubin}) as
	\begin{equation}
		\label{eq:Cknorm}
		\norm{f}_{\cC^{m}}=\sum_{j=0}^{m}\sup\limits_{p \in M}|(\nabla^{j}f)(p)|\;.
	\end{equation}
	
	Let $L^{2}(M)$ be the Hilbert space of complex-valued functions on $M$ whose modulus is square-integrable with respect to the measure $\vol$, and denote by
	\begin{equation*}
		\langle \phi,\psi\rangle=\int_M \phi\;\bar{\psi}\;\text{d}\vol
	\end{equation*}
	the inner product of two elements $\phi,\psi\in L^{2}(M)$.
	
	Define two additional elements
	\begin{equation*}
		U=
		\begin{pmatrix}
			0&1\\
			0&0
		\end{pmatrix}
		\;,\quad 
		V=
		\begin{pmatrix}
			0&0\\
			1&0
		\end{pmatrix}
	\end{equation*}
	in the Lie algebra $\sl_2(\R)$.
	Identifying elements of the universal enveloping algebra of $\sl_2(\R)$ with left-invariant differential operators on the space $\cC^{\infty}(M)$, we define the Casimir operator as the second-order differential operator $\square=-X^{2}+X-UV\colon \cC^{2}(M)\to \cC^0(M)$. It admits a unique maximal extension to an unbounded self-adjoint operator on $L^{2}(M)$; in particular, its spectrum $\text{Spec}(\square)$ consists of real numbers. As $M$ is compact, it is well-known that $\text{Spec}(\square)$ is pure point, and is a discrete subset of $\R$. The elements of $\Spec(\square)$ classify the irreducible representations strongly contained in the Koopman representation arising from the measure-preserving action of $\SL_2(\R)$ on the measure space $(M,\vol)$, as belonging to the principal, complementary or discrete series representations if the corresponding eigenvalue $\mu$ satisfies, respectively, $\mu\geq 1/4$, $0<\mu<1/4$, $\mu\leq 0$. With the normalization we have chosen\footnote{Since $\sl_2(\R)$ is a simple Lie algebra, Casimir elements in its universal enveloping algebra are uniquely determined up to real scalars.}, the action of $\square$ on $\cC^{2}$-functions defined on the surface $S$ is given by the Laplace-Beltrami operator $\Delta_S$ associated to the hyperbolic structure on $S$.
	
	We are interested in quantitative equidistribution properties of the uniform probability measures supported on the circle arcs
	\begin{equation*}
		\phi_t^{X}(\{r_s(p):0\leq s\leq \theta \})
	\end{equation*}
	as $t$ goes to infinity, for every fixed $p \in M$ and $\theta\in (0,4\pi]$ (cf.~our normalization of $\Theta$ in~\eqref{eq:rotationflow}, a full circle corresponds to $\theta=4\pi$). For any parameter $\theta\in (0,4\pi]$ and any continuous function $f\colon M\to \C$, we thus define the function $k_{f,\theta}\colon M\times \R\to \C$ as
	\begin{equation}
		\label{eq:ktheta}
		k_{f,\theta}(p,t)\coloneqq \frac{1}{\theta} \int_0^{\theta}f\circ \phi_t^{X}\circ r_s(p)\;\text{d}s\; ,\quad p \in M,\;t\in \R.
	\end{equation}
	In the forthcoming subsection, we provide a precise asymptotic expansion of $k_{f,\theta}(p,t)$ as $t$ tends to infinity, first for joint eigenfunctions\footnote{As we shall explain in Section~\ref{sec:unitaryrepresentations}, there exists an orthonormal basis of $L^{2}(M)$ consisting of such joint eigenfunctions; Theorem~\ref{thm:case_Theta_eigenfn} is thus to be regarded as a building block for the more general Theorem~\ref{thm:mainexpandingtranslates}.} of the operators $\square$ and $\Theta$ (Theorem~\ref{thm:case_Theta_eigenfn}), and then for arbitrary functions fulfilling a suitable regularity condition (Theorem~\ref{thm:mainexpandingtranslates}).

	\subsection{Quantitative equidistribution of expanding translates of circle arcs}
	\label{sec:quantitativeequidistribution}
	
	We begin with the case of joint eigenfunctions of the operators $\square$ and $\Theta$.
	Observe that the left-invariant vector field $\Theta\in \sl_2(\R)$ acts as an unbounded skew-symmetric operator on $L^{2}(M)$; if $\Theta f=\lambda f$ for some $f\in \mathscr{C}^{1}(M)$ and $\lambda \in \C$, then $\lambda=\frac{i}{2}n$ for some $n\in \Z$. 
	
	In the following statement and throughout, we associate to each $\mu\in \text{Spec}(\square)$
	the unique complex number $\nu\in \R_{\geq 0}\cup i\R_{>0}$ satisfying $1-\nu^2=4\mu$.

	\begin{thm}\label{thm:case_Theta_eigenfn}
		There exist real constants $\kappa_0$ and $\kappa(\mu)$ for any positive Casimir eigenvalue $\mu$ such that the following assertions hold. Let $\theta\in (0,4\pi]$, $\mu\in \emph{Spec}(\square)$ and $n\in \Z$, and suppose $f\in \mathscr{C}^{2}(M)$ satisfies $\square f=\mu f$, $\Theta f=\frac{i}{2}nf$. Define $k_{f,\theta}(p,t)$ as in~\eqref{eq:ktheta}.  
		\begin{enumerate}
			\item If $\mu>1/4$, there exist H\"{o}lder-continuous functions $D^{+}_{\theta,\mu,n}f,D^{-}_{\theta,\mu,n}f\colon M\to \C$ with H\"{o}lder exponent $1/2$ and
			\begin{equation*}
				\norm{D^{\pm}_{\theta,\mu,n}f}_{\infty}\leq \frac{\kappa(\mu)}{\theta}(n^2+1)\norm{f}_{\mathscr{C}^{1}} 
			\end{equation*}
			such that, for every $p \in M$ and $t\geq 1$, 
			\begin{equation}
				\label{eq:estabovequarter}
				k_{f,\theta}(p,t)=e^{-\frac{t}{2}}\cos{\biggl(\frac{\Im{\nu}}{2}t\biggr)}D^{+}_{\theta,\mu,n}f(p)+e^{-\frac{t}{2}}\sin{\biggl(\frac{\Im{\nu}}{2}t\biggr)}D^{-}_{\theta,\mu,n}f(p)+\mathcal{R}_{\theta,\mu,n}f(p,t)\;,
			\end{equation}
			where 
			\begin{equation}
				\label{eq:estremainderabovequarter}
				|\mathcal{R}_{\theta,\mu,n}f(p,t)|\leq \frac{8\kappa_0(n^{2}+1)}{\theta\; \Im{\nu}}\norm{f}_{\mathscr{C}^{1}}e^{-t}\;.
			\end{equation}
			\item If $\mu=1/4$, there exist H\"{o}lder-continuous functions $D^{+}_{\theta,1/4,n}f\colon M\to \C$, with H\"{o}lder exponent $1/2-\eps$ for every $\eps>0$, and $D^{-}_{\theta,1/4,n}f\colon M\to \C$, with H\"{o}lder exponent $1/2$, and satisfying
			\begin{equation*}
				\norm{D^{\pm}_{\theta,1/4,n}f}_{\infty}\leq \frac{\kappa(1/4)}{\theta}(n^2+1)\norm{f}_{\mathscr{C}^{1}}\;,
			\end{equation*}
			such that, for every $p \in M$ and $t\geq 1$,
			\begin{equation}
				\label{eq:estquarter}
				k_{f,\theta}(p,t)=e^{-\frac{t}{2}}D^{+}_{\theta,1/4,n}f(p)+te^{-\frac{t}{2}}D^{-}_{\theta,1/4,n}f(p)+\mathcal{R}_{\theta,1/4,n}f(p,t)\;,
			\end{equation}
			where 
			\begin{equation*}
				|\mathcal{R}_{\theta,1/4,n}f(p,t)|\leq \frac{4\kappa_0}{\theta}(n^{2}+1)\norm{f}_{\mathscr{C}^{1}}(t+1)e^{-t}\;.
			\end{equation*}
			\item If $0<\mu<1/4$, there exist functions $D^{+}_{\theta,\mu,n}f, D^{-}_{\theta,\mu,n}f \colon M\to \C$, respectively H\"{older}-continuous with H\"{o}lder exponent $\frac{1-\nu}{2}$ and  of class $\cC^{1}$, and satisfying
			\begin{equation*}
				\norm{D^{\pm}_{\theta,\mu,n}f}_{\infty}\leq \frac{\kappa(\mu)}{\theta}(n^{2}+1)\norm{f}_{\mathscr{C}^{1}}\;,
			\end{equation*}
			such that, for every $p \in M$ and $t\geq 1$,
			\begin{equation}
				\label{eq:estpositivebelowquarter}
				k_{f,\theta}(p,t)=e^{-\frac{1+\nu}{2}t}D^{+}_{\theta,\mu,n}f(p)+e^{-\frac{1-\nu}{2}t}D^{-}_{\theta,\mu,n}f(p)+\mathcal{R}_{\theta,\mu,n}f(p,t)\;,
			\end{equation}
			where 
			\begin{equation}
				\label{eq:estremainderbelowquarter}
				|\mathcal{R}_{\theta,\mu,n}f(p,t)|\leq \frac{4\kappa_0}{\theta\nu(1-\nu)(1+\nu)}(n^{2}+1)\norm{f}_{\mathscr{C}^{1}}e^{-t}\;.
			\end{equation}
			\item If $\mu=0$, there exists a function $G_{\theta,n}f\colon M\times \R_{>0}\to \C$, with $G_{\theta,n}f(\cdot,t)$ of class $\cC^{1}$ for any $t>0$, $G_{\theta,n}f(p,\cdot)$ continuous for every $p \in M$ and
			\begin{equation*}
				\sup\limits_{p\in M,\;t>0}|G_{\theta,n}f(p,t)|\leq \frac{\kappa_0}{\theta}(n^{2}+1)\norm{f}_{\cC^{1}}
			\end{equation*}
			such that, for every $p \in M$ and $t\geq 1$,  
			\begin{equation}
				\label{eq:estzero}
				k_{f,\theta}(p,t)=\int_{M}f\;\emph{d}\vol+e^{-t}\int_1^{t}-G_{\theta,n}f(p,\xi)\;\emph{d}\xi+\mathcal{R}_{\theta,0,n}f(p,t)\;,
			\end{equation}
			where 
			\begin{equation}
				\label{eq:remainderzero}
				|\mathcal{R}_{\theta,0,n}f(p,t)|\leq\frac{8e\pi+\kappa_0}{\theta} (n^{2}+1)\norm{f}_{\mathscr{C}^{1}}e^{-t}\;.
			\end{equation}
			\item If $\mu<0$ then, for every $p \in M$ and $t\geq 1$,
			\begin{equation}
				\label{eq:estimatediscreteseries}
				|k_{f,\theta}(p,t)|\leq \frac{1}{\theta}\biggl(\frac{4\kappa_0}{(\nu-1)(\nu+1)}+\frac{2e\pi(3+\nu)}{\nu}\biggr)(n^{2}+1) \norm{f}_{\mathscr{C}^{1}}e^{-t}\;.
			\end{equation}
		\end{enumerate}
	\end{thm}

	\begin{rmk}
		The asymptotic expansions for $k_{f,\theta}(p,t)$ are easily transferred to the case of large negative values of the time parameter $t$, by noticing that
		\begin{equation*}
			k_{f,\theta}(p,-t)=k_{f \circ r_{\pi},\theta}(r_{\pi}(p),t) 
		\end{equation*}
		for any $p \in M$ and $t\geq 1$.
	\end{rmk}
	
	Taking advantage of Theorem~\ref{thm:case_Theta_eigenfn} and of standard harmonic analysis on the group $\SL_2(\R)$, we establish an asymptotic expansion of $k_{f,\theta}(p,t)$ for all sufficiently regular, but otherwise arbitrary test functions $f$.
	Define the Laplacian on $M$ to be the second-order linear differential operator $\Delta=\square-2\Theta^{2}$. 
	For any $s\in \R_{>0}$, let $W^{s}(M)$ be the Sobolev space of order $s$ on the manifold $M$, that is, the Hilbert-space completion of the complex vector space $\cC^{\infty}(M)$ of smooth functions on $M$ endowed with the inner product 
	\begin{equation*}
		\langle \phi,\psi \rangle_{W^s}=\langle (1+\Delta)^{s}\phi,\psi\rangle\;,\quad \phi,\psi\in \cC^{\infty}(M).
	\end{equation*}
	
	As $M$ is compact, the well-known Sobolev Embedding Theorem (which we recall in Theorem~\ref{thm:Sobolevembedding}) ensures the existence of a continuous embedding $W^{s}(M)\hookrightarrow \cC^{r}(M)$ whenever $s\in \R_{>0}$ and $r\in \N$ are such that $s>r+3/2$; explicitly, there is a constant $C_{r,s}\in \R_{>0}$ (which for definiteness we take equal to the operator norm of the corresponding embedding) such that $\norm{f}_{\cC^{r}}\leq C_{r,s}\norm{f}_{W^{s}}$ for any $f\in W^{s}(M)$. Hereinafter, an element $f\in W^{s}(M)$ for $s>3/2$ is always identified with its unique continuous representative.
	
	Set
	\begin{equation}
		\label{eq:epszero}
		\eps_{0}=
		\begin{cases}
			1 &\text{ if }1/4\in \text{Spec}(\square)\;,\\
			0 &\text{ if }1/4\notin \text{Spec}(\square)\;.\\
		\end{cases}
	\end{equation}
	
	\begin{thm}
		\label{thm:mainexpandingtranslates}
		There exist real constants $C_{\emph{Spec}}$ and $\;C_{\emph{Spec}}'$, depending only on the spectrum of the Casimir operator on $L^{2}(M)$, such that the following holds. Let $\theta\in (0,4\pi]$ and $s>11/2$ be real numbers, and suppose $f\in W^{s}(M)$.
		Then there exist continuous functions $D^{+}_{\theta,\mu}f,\;D^{-}_{\theta,\mu}f\colon M\to \C$ for any positive Casimir eigenvalue $\mu$, with 
		\begin{equation}
			\label{eq:boundDthetamu}
			\sum_{\mu \in \emph{Spec}(\square)\cap \R_{>0}}\norm{D^{+}_{\theta,\mu}f}_{\infty}+\norm{D^{-}_{\theta,\mu}f}_{\infty}\leq \frac{C_{\emph{Spec}}'C_{1,s-3}}{\theta}\norm{f}_{W^{s}},
		\end{equation}
		such that, for every $p\in M$ and $t\geq 1$,
		\begin{equation}
			\label{eq:asymptoticgeneral}
			\begin{split}
				\frac{1}{\theta}\int_0^{\theta}f\circ \phi^{X}_{t}\circ r_s(p)\;\emph{d}s=&\int_{M}f\;\emph{d}\vol\\
				&+e^{-\frac{t}{2}}\biggl( \sum_{\mu\in \emph{Spec}(\square),\;\mu>1/4}\cos{\biggl(\frac{\Im{\nu}}{2}t\biggr)} D^{+}_{\theta,\mu}f(p)+\sin{\biggl(\frac{\Im{\nu}}{2}t\biggr)} D^{-}_{\theta,\mu}f(p)\biggr)\\
				&+\sum_{\mu\in \emph{Spec}(\square),\;0<\mu<1/4}e^{-\frac{1+\nu}{2}t}D^{+}_{\theta,\mu}f(p)+e^{-\frac{1-\nu}{2}t}D^{-}_{\theta,\mu}f(p)\\
				&+\varepsilon_0\bigl(e^{-\frac{t}{2}}D^{+}_{\theta,1/4}f(p)+te^{-\frac{t}{2}}D^{-}_{\theta,1/4}f(p)\bigr)+\mathcal{R}_{\theta}f(p,t)\;,
			\end{split}
		\end{equation}
		where
		\begin{equation}
			\label{eq:globalremainderestimate}
			|\mathcal{R}_{\theta}f(p,t)|\leq \frac{C_{\emph{Spec}}C_{1,s-3}}{\theta}\norm{f}_{W^{s}}(t+1)e^{-t} \;.
		\end{equation}
	\end{thm}   
	
	\begin{rmk}
			An asymptotic expansion of the same form can be deduced from the work of Bufetov and Forni \cite{BuFo}, where the authors study the asymptotic distribution properties of translates of general rectifiable arcs\footnote{We thank Giovanni Forni for drawing our attention to this earlier work.}. The results are there stated only for positive Casimir parameters; in order to treat the components supported on the discrete series, it would be necessary to adapt the method used by Flaminio and Forni in~\cite{Flaminio-Forni}, which however would only allow to establish upper bounds contributing to the error term.

			Comparing the coefficients in Theorem \ref{thm:mainexpandingtranslates} and in the the expansion which follows from the results in \cite{BuFo}, it is possible to infer that the functions $D^{\pm}_{\theta, \mu}f$ are proportional, in a sense clarified for instance by the second author in~\cite{Rav} in the corresponding setup, to the invariant distributions for the unstable horocycle flow or, in the language of the work~\cite{Dyatlov-Faure-Guillarmou} by Dyatlov, Faure and Guillarmou, to the resonant states for the geodesic flow. 
			The method employed in the present article, which differs substantially from the one of~\cite{BuFo}, affords a sharper control on those terms in the asymptotic expansion coming from the components associated to non-positive Casimir parameter, as well as finer information on the regularity of the coefficients with respect to the base point.
		
	\end{rmk}

	We record here below the ensuing effective equidistribution statement, in which we single out the two main terms of the asymptotic expansion, thereby highlighting the dependence of the latter on the spectral gap of the underlying hyperbolic surface $S=\Ga\bsl \Hyp$, defined as 
	\begin{equation*}
		\mu_*=\inf(\text{Spec}(\square)\cap \R_{>0})=\inf(\text{Spec}(\Delta_S)\setminus\{0\})\;;
	\end{equation*}
	its corresponding parameter is denoted by $\nu_*$. 
	
	\begin{term}
		We adopt the classical Landau notation $o(\eta(t))$ for $\eta\colon \R_{> 0}\to \R_{>0}$ tending to zero at infinity, to indicate a function $\lambda\colon \R_{>0}\to \C$ such that $|\lambda(t)|/\eta(t)\to 0$ as $t\to\infty$. 
	\end{term}

	\begin{cor}
		\label{cor:effective}
		Let $\theta,s,C_{1,s-3},C'_{\emph{Spec}}$ and $f$ be as in Theorem~\ref{thm:mainexpandingtranslates}. Then there exists a function $D^{\emph{main}}_{\theta}f\colon M\to \C$ with
		\begin{equation}
			\label{eq:boundDthetamain}
			\norm{D^{\emph{main}}_{\theta}f}_{\infty}\leq \frac{C'_{\emph{Spec}}C_{1,s-3}}{\theta}\norm{f}_{W^{s}}
		\end{equation}
		such that, for every $p \in M$ and $t\geq 1$, 
		\begin{equation}
			\label{eq:effective}
			\frac{1}{\theta}\int_0^{\theta}f\circ \phi_t^{X}\circ r_s(p)\;\emph{d}s=\int_{M}f\;\emph{d}\vol+\;D^{\emph{main}}_{\theta}f(p)\;t^{\eps_0}e^{-\frac{1-\Re{\nu_*}}{2}t}+o(e^{-\frac{1-\Re{\nu_*}}{2}t})\;.
		\end{equation}
	\end{cor}

	\begin{rmk}
		\label{rmk:equivalentmetrics}
		We collect here below further comments about Theorem~\ref{thm:case_Theta_eigenfn}, Theorem~\ref{thm:mainexpandingtranslates} and Corollary~\ref{cor:effective}.
		\begin{enumerate}
			\item The H\"{o}lder-continuity claims concerning the coefficients $D^{\pm}_{\theta,\mu,n}$ appearing in Theorem~\ref{thm:case_Theta_eigenfn} tacitly involve the choice of a distance function $d$ on $M$. It is intended that $d$ is the Riemannian distance function induced on the connected manifold $M$ by a Riemannian metric $g$. The H\"{o}lder-continuity property, as well as the H\"{o}lder exponent, of $D^{\pm}_{\theta,\mu,n}$ is independent of the choice of such a $g$, as any two Riemannian metrics on a compact manifold induce Lipschitz-equivalent metrics (see \cite[Lem.~13.28]{Lee}).
			\item We point out the analogy of Theorem~\ref{thm:case_Theta_eigenfn} with the asymptotics of horocycle ergodic integrals for Casimir eigenfunctions established in~\cite[Thm.~1]{Rav} (in the same way, Theorem~\ref{thm:mainexpandingtranslates} mirrors~\cite[Thm.~2]{Rav}). This similarity stems computationally from the application of the exact same spectral method in both circumstances; as a matter of fact, it comes as no surprise from a geometric standpoint, for large hyperbolic circles approximate orbits of the unstable horocycle flow. 
			
			\item For $\theta=4\pi$, we recover the qualitative equidistribution of expanding circles towards the uniform measure $\vol$ obtained by Margulis in~\cite{Margulis-thesis}. 
			\item The equidistribution rate in Corollary~\ref{cor:effective} matches exactly the mixing rate of the geodesic flow on $M$ obtained by Ratner in~\cite[Thm.~2]{Ratner}.
			\item As suggested by the underlying geometric picture, the various upper bounds in the statements of Theorem~\ref{thm:case_Theta_eigenfn} and~\ref{thm:mainexpandingtranslates} indicate that the speed of equidistribution improves as $\theta$ increases to $4\pi$, that is, as the length of the initial circle arc gets larger. 
		\end{enumerate}
	\end{rmk}
	
	As a special case of Theorem \ref{thm:mainexpandingtranslates}, we obtain an asymptotic expansion for the equidistribution rate of $\Theta$-invariant functions; in other words, we provide the precise asymptotic behaviour of large circles on the compact hyperbolic surface $S$, identified, here and afterwards whenever convenient, with the double coset space $\Ga\bsl \SL_2(\R)/\SO_2(\R)$.
	
	Let $d_{\Hyp}$ denote the distance function on $\Hyp$ arising from the hyperbolic Riemannian metric. For any $z\in \Hyp$ and $t>0$, let $C_{\Hyp}(z,t)=\{z'\in \Hyp:d_{\Hyp}(z,z')=t \}$ be the hyperbolic circle of radius $t$ centered at $z$, and denote by $C_{S}(z,t)$ its projection to $S$. With $m_{C_S(z,t)}$ we indicate the projection to $S$ of the unique isometrically-invariant Borel probability measure supported on $C_{\Hyp}(z,t)$. Finally, let $m_{S}$ be the hyperbolic area measure on $S$, normalized to be a probability measure.
	
	\begin{thm}
		\label{thm:expandingonsurface}
		Let $\theta\in (0,4\pi],\;s>9/2$, $\tilde{f}\colon S \to \C$ a function such that the $\SO_2(\R)$-invariant function $f\colon M\to \C$ defined by $f(\Ga g)=\tilde{f}(\Ga g \SO_2(\R))$ for any $g\in \SL_2(\R)$ is in $W^{s}(M)$. Then, for  every $z=\Ga g\SO_2(\R)\in S$,  $p=\Ga g \in M$ and $t\geq 1$,
		\begin{equation*}
			\begin{split}
				\int_{S}\tilde{f}\;\emph{d}m_{C_S(z,t)}&=\int_{S}\tilde{f}\;\emph{d}m_S
				+e^{-\frac{t}{2}}\biggl( \sum_{\mu\in \emph{Spec}(\Delta_S),\;\mu>1/4}\cos{\biggl(\frac{\Im{\nu}}{2}t\biggr)} D^{+}_{4\pi,\mu}f(p)+\sin{\biggl(\frac{\Im{\nu}}{2}t\biggr)} D^{-}_{4\pi,\mu}f(p)\biggr)\\
				&+\sum_{\mu\in \emph{Spec}(\Delta_S),\;0<\mu<1/4}e^{-\frac{1+\nu}{2}t}D^{+}_{4\pi,\mu}f(p)+e^{-\frac{1-\nu}{2}t}D^{-}_{4\pi,\mu}f(p)\\
				&+\varepsilon_0\bigl(e^{-\frac{t}{2}}D^{+}_{4\pi,1/4}f(p)+te^{-\frac{t}{2}}D^{-}_{4\pi,1/4}f(p)\bigr)+\mathcal{R}_{4\pi}f(p,t)\;.
			\end{split}
		\end{equation*}
		In particular, the functions $\mathcal{R}_{4\pi}f$ and $D^{\pm}_{4\pi,\mu}f,\;\mu\in \emph{Spec}(\Delta_S)\cap \R_{>0}$, defined as in Theorem~\ref{thm:mainexpandingtranslates}, are $\SO_2(\R)$-invariant. 
	\end{thm}  
	
	Clearly, an analogous statement holds for arbitrary sub-arcs of the circles $C_{S}(z,t)$.
	
	\begin{rmk}
		It is worth highlighting that, in the case of an $\SO_2(\R)$-invariant observable $f$, we require mildly less restrictive assumptions on its Sobolev regularity; this will become relevant in Section~\ref{sec:latticepoint} when we deal with the error rate for the hyperbolic lattice point counting problem. 
		
		Furthermore, we emphasize that Theorem~\ref{thm:expandingonsurface} improves upon the equidistribution results in~\cite{Randol} (which, on the other hand, apply to any dimension) in a twofold way: it demands less restrictive conditions on the regularity of test functions and, more importantly, it refines Randol's upper bound on the equidistribution rate by giving a precise asymptotic expansion.
	\end{rmk}
	
	Leveraging the explicit dependence of the asymptotics in Theorem~\ref{thm:mainexpandingtranslates} on the upper bound $\theta$ of the domain of parametrization of the circle arc under consideration, we deduce a sufficient quantitative condition for the equidistribution of shrinking pieces of expanding circles; this is in the vein of Str\"{o}mbergsson's results in~\cite{Strombergsson-closedhorocycles}, where the analogous question is investigated for shrinking portions of closed horocycles on non-closed hyperbolic surfaces of finite volume.
	
	\begin{cor}
		\label{cor:shrinkingarcs}
		Let $p \in M$, $\theta_1,\theta_2\colon \R_{>0} \to (0,4\pi)$ two functions with $\theta_1(t)\leq \theta_2(t)$ for any $t>0$, and consider the circle sub-arcs $\gamma_{t}=\{\phi^{X}_t\circ r_{s}(p):\theta_1(t)\leq s \leq \theta_2(t)\}$. For any $t>0$, let $\mu_t$ be the normalized restriction to $\gamma_t$ of the unique isometrically-invariant measure on the corresponding full circle. Assume that there exist $t_0>0$ and a function $\eta\colon \R_{>0}\to \R_{>0}$ with $\eta(t)\to\infty$ as $t\to\infty$ such that $\theta_2(t)-\theta_1(t)\geq\eta(t) e^{-\frac{1-\Re{\nu_*}}{2}t}$ for any $t\geq t_0$. Then the circle arcs $\gamma_t$ equidistribute as $t\to\infty$: more precisely, the measures $\mu_t$ converge in the weak$^*$ topology, as $t\to\infty$, towards the uniform measure $\vol$ on $M$. 
	\end{cor} 
	
	\subsection{Statistical limit theorems for deviations from the average}
	\label{sec:introductionCLT}
	
	The asymptotics in Theorem~\ref{thm:mainexpandingtranslates} affords the means to examine the long-term statistical behaviour of the averages of a given observable along expanding circle arcs. Historically, a momentous discovery in the twentieth century was the realization that the long-term evolution of deterministic systems  frequently obeys the same statistical laws governing the asymptotic behaviour of random processes. Specifically, this feature is a typical characteristic of dynamical systems with hyperbolic behaviour, among which geodesic flows on negatively curved compact manifolds feature prominently; we refer the reader to the survey in the introduction to~\cite{Dolgopyat-Sarig}, as well as to the references therein, for an extensive discussion of the topic.  
	
	Because of the exponential mixing properties of the geodesic flow $(\phi^{X}_t)_{t\in \R}$ on $M$ (cf.~\cite[Thm.~2]{Ratner}),  at first it stands to reason to expect, for a real-valued function $f$ on $M$ with finite first moment with respect to the volume measure $\vol$, the distribution of the deviations from the average
	\begin{equation*}
		\frac{1}{\theta}\int_0^{\theta}f\circ \phi_T^{X}\circ r_s(p)\;\text{d}s-\int_{M}f\;\text{d}\vol\;,
	\end{equation*}
	when the base point $p$ is randomly chosen according to the law $\vol$, (something which is henceforth indicated with $p\sim \vol$), to mimic for large values of $T$ the law of the empirical mean
	\begin{equation*} \frac{1}{N}\sum_{n=1}^{N}X_n
	\end{equation*}
	of an increasing number $N$ of independent real-valued random variables $X_n$. More precisely, since the hyperbolic length of a circle of radius $T$ is proportional to $e^{T}$ (see the explanation below~\eqref{eq:integralbounds}), a full analogoue of the classical Central Limit Theorem in this case would affirm that the random variables 
	\begin{equation*}
		e^{\frac{T}{2}}\biggl(\frac{1}{\theta}\int_{0}^{\theta}f\circ \phi^{X}_T\circ r_s(p)\;\text{d}s-\int_{M}f\;\text{d}\vol \biggr)\;,\quad p\sim \vol
	\end{equation*} 
	converge in law, as $T$ tends to infinity, to a normally distributed random variable. However, the geometric resemblance of large hyperbolic circles to orbits of the unstable horocycle flow, for which similar phenomena occur (cf.~\cite[Thm.~1.4,~1.5]{BuFo} and~\cite[Thm.~4]{Rav}) accounts both for the emergence of other types of limiting distributions, and for possibly different renormalization factors, depending on the spectral properties of the observable under consideration.
	
	In order to minimize the difference with the classical probabilistic setup of sums of independent random variables, we shall state all the results in this subsection for real-valued observables, though the extension to complex-valued ones is immediate.   
	\begin{thm}
		\label{thm:CLT}
		Let $\theta\in (0,4\pi]$, $s>11/2$, $f\in W^{s}(M)$ a real-valued function. Assume that 
		\begin{equation*}
			\mu_f=\inf\{\mu\in \emph{Spec}(\square)\cap \R_{>0}:D^{-}_{\theta,\mu}f \emph{ does not vanish identically on }M \}
		\end{equation*}
		is finite, and let $\nu_f$ be the unique complex number in $\R_{\geq 0}\cup i\R_{>0}$ satisfying $1-\nu_f^{2}=4\mu_f$.
		\begin{enumerate}
			\item If $0<\mu_f<1/4$, the random variables
			\begin{equation*}
				e^{\frac{1-\nu_f}{2}T}\biggl(\frac{1}{\theta}\int_0^{\theta}f\circ \phi^{X}_T\circ r_s(p)\;\emph{d}s-\int_{M}f\;\emph{d}\vol\biggr)\;,\quad p\sim \vol
			\end{equation*}
			converge in distribution to $D^{-}_{\theta,\mu_f}f$ as $T\to\infty$.
			\item If $\mu_f=1/4$, the random variables
			\begin{equation*}
				T^{-1}e^{\frac{T}{2}}\biggl(\frac{1}{\theta}\int_0^{\theta}f\circ \phi^{X}_T\circ r_s(p)\;\emph{d}s-\int_{M}f\;\emph{d}\vol\biggr)\;,\quad p\sim \vol
			\end{equation*}
			converge in distribution to $D^{-}_{\theta,1/4}f$ as $T\to\infty$.
			\item If $\mu_f>1/4$, the random variables
			\begin{equation*}
				e^{\frac{T}{2}}\biggl(\frac{1}{\theta}\int_0^{\theta}f\circ \phi^{X}_T\circ r_s(p)\;\emph{d}s-\int_{M}f\;\emph{d}\vol\biggr)\;,\quad p\sim \vol
			\end{equation*}
			converge in distribution, as $T\to\infty$,  to the quasi-periodic motion
			\begin{equation*} \eps_0D^{+}_{\theta,1/4}f(p)+\sum_{\mu\in \emph{Spec}(\square),\;\mu>1/4}\cos{\biggl(\frac{\Im{\nu}}{2}T\biggr)} D^{+}_{\theta,\mu}f+\sum_{\mu \in \emph{Spec}(\square),\;\mu\geq \mu_f}\sin{\biggl(\frac{\Im{\nu}}{2}T\biggr)} D^{-}_{\theta,\mu}f
			\end{equation*}
			on the set of real-valued random variables defined on the probability space $(M,\vol)$.
		\end{enumerate}
	\end{thm}
	
	Observe the remarkable fact that the limiting distributions appearing in the statement of Theorem~\ref{thm:CLT} are compactly supported on the real line, owing to the fact that the coefficients $D^{\pm}_{\theta,\mu}f$ are bounded. This stands in stark contrast with the classical versions of the Central Limit Theorem in probability theory, where in non-trivial situations the distribution of errors is governed by the fully supported Gaussian distribution. 
	Furthermore, it is straightforward to check, at least when the Casimir components of $f$ are eigenfunctions of $\Theta$ and using the explicit expressions of the coefficients $D^{\pm}_{\theta,\mu}f$ in~\eqref{eq:Dplusabovequarter},~\eqref{eq:Dminusabovequarter},~\eqref{eq:Dquarter} and~\eqref{eq:Dbelowquarter}, that the limit law is non-trivial\footnote{As soon as $f$ is not almost-surely constant, obviously.}, that is, not a Dirac mass. In the general case, the coefficients are given by infinite superpositions of the previous ones; though we shall refrain from a detailed verification, there is no reason to expect cancellation phenomena to come about and produce limiting random variables which are constant almost-surely.
	
	Theorem~\ref{thm:CLT} is a consequence of its quantitative version which we presently discuss. For $f$ fulfilling the conditions in Theorem~\ref{thm:CLT}, let $\mathbf{P}_{\theta,f}$ denote the law of the random variable $D^{-}_{\theta,\mu_f}f$ when $\mu_f\leq 1/4$, and $\mathbf{P}_{\theta,f}(T)$ the law of the random variable
	\begin{equation*}
		\eps_0D^{+}_{\theta,1/4}f(p)+\sum_{\mu \in \text{Spec}(\square),\;\mu>1/4}\cos{\biggl(\frac{\Im{\nu}}{2}T\biggr)} D^{+}_{\theta,\mu}f+\sum_{\mu \in \text{Spec}(\square),\;\mu\geq \mu_f}\sin{\biggl(\frac{\Im{\nu}}{2}T\biggr)} D^{-}_{\theta,\mu}f\;,\quad T>0,
	\end{equation*}
	when $\mu_f>1/4$. Furthermore, for any $T\geq 1$ we let $\mathbf{P}_{\theta,f}^{\text{circ}}(T)$ be:
	\begin{enumerate} 
		\item the law of
		\begin{equation*}
			e^{\frac{1-\nu_f}{2}T}\biggl(\frac{1}{\theta}\int_0^{\theta}f\circ \phi^{X}_T\circ r_s(p)\;\text{d}s-\int_{M}f\;\text{d}\vol\biggr) \quad \text{if }0<\mu_f<1/4,
		\end{equation*}
		\item the law of
		\begin{equation*}
			T^{-1}e^{\frac{T}{2}}\biggl(\frac{1}{\theta}\int_0^{\theta}f\circ \phi^{X}_T\circ r_s(p)\;\text{d}s-\int_{M}f\;\text{d}\vol\biggr) \quad \text{if }\mu_f=1/4,
		\end{equation*}
		\item and the law of
		\begin{equation*}
			e^{\frac{T}{2}}\biggl(\frac{1}{\theta}\int_0^{\theta}f\circ \phi^{X}_T\circ r_s(p)\;\text{d}s-\int_{M}f\;\text{d}\vol\biggr)\quad \text{if }\mu_f>1/4.
		\end{equation*}
	\end{enumerate}
	
	Denote by $d_{LP}$ the L\'{e}vi-Prokhorov distance on the set of Borel probability measures on $\R$ (cf.~Section~\ref{sec:spatialDLT}), which induces on the latter the topology of weak convergence. Recall also that $\mu_*$ denotes the spectral gap of $S=\Ga\bsl \Hyp$, with associated parameter $\nu_*$.
	
	\begin{prop}
		\label{prop:CLT}
		Let the assumptions be as in Theorem~\ref{thm:CLT}, and the constants $C_{\emph{Spec}}$ and $C_{\emph{Spec}}'$ be as in Theorem~\ref{thm:mainexpandingtranslates}. 
		\begin{enumerate}
			\item If $0<\mu_f<1/4$, then there is an explicit constant $\eta_f>0$, depending only on $\mu_f$ and on $\emph{Spec}(\square)$, such that 
			\begin{equation*}
				d_{LP}(\mathbf{P}_{\theta,f}^{\emph{circ}}(T),\mathbf{P}_{\theta,f})\leq \frac{C'_{\emph{Spec}}C_{1,s-3}}{\theta}\norm{f}_{W^{s}}Te^{-\eta_f T}
			\end{equation*}
			for every $T\geq 1$.
			\item If $\mu_f=1/4$, then there exists a constant $C_{\emph{pos}}$, depending only on $\emph{Spec}(\square)\cap \R_{>0}$, such that 
			\begin{equation*}
				d_{LP}(\mathbf{P}_{\theta,f}^{\emph{circ}}(T),\mathbf{P}_{\theta,f}(T))\leq \frac{C_{\emph{pos}}C_{1,s-3} }{\theta}\norm{f}_{W^{s}}T^{-1} 
			\end{equation*}
			for every $T\geq 1$.
			\item If $\mu_f>1/4$, then:
			\begin{enumerate}
				\item when $\mu_*<1/4$,
				\begin{equation*}
					d_{LP}(\mathbf{P}_{\theta,f}^{\emph{circ}}(T),\mathbf{P}_{\theta,f}(T))\leq \frac{C_{\emph{Spec}}'C_{1,s-3}}{\theta}\norm{f}_{W^{s}}e^{-\frac{\nu_*}{2}T}
				\end{equation*}
				for every $T\geq 1$;
				\item when $\mu_*\geq 1/4$,
				\begin{equation*}
					d_{LP}(\mathbf{P}_{\theta,f}^{\emph{circ}}(T),\mathbf{P}_{\theta,f}(T))\leq \frac{ C_{\emph{Spec}}C_{1,s-3} }{\theta}\norm{f}_{W^{s}}(T+1)e^{-\frac{T}{2}}
				\end{equation*}
				for every $T\geq 1$.
			\end{enumerate}
		\end{enumerate}	
	\end{prop}
	
	Entirely analogous deductions, of which we omit the details (cf.~Section~\ref{sec:spatialDLT}), can be made in the case where the $D^{-}_{\theta,\mu}f$ vanish everywhere for any Casimir eigenvalue $\mu>0$ but $D^{+}_{\theta,\mu}f$ is not identically zero for at least one such $\mu$.
	
	\begin{rmk}
		As the proof of Proposition~\ref{prop:CLT} shall clearly illustrate (see, in particular, Lemma~\ref{lem:LPnearbyvariables}), the assumption that the base point $p$ is sampled according to the uniform measure $\vol$ is immaterial, as far as the validity of Proposition~\ref{prop:CLT} and Theorem~\ref{thm:CLT} is concerned. It is possible to replace the measure $\vol$ with any other Borel probability measure $\mu$ on $M$, without affecting the quantitative rate of convergence, provided that the laws of the limiting random variables are modified accordingly. 
	\end{rmk}
	
	In the remaining case when the $D^{\pm}_{\theta,\mu}f$ vanish identically on $M$ for any positive Casimir eigenvalue $\mu$, the explicit expressions of the coefficients appearing in Theorem~\ref{thm:case_Theta_eigenfn} and~\ref{thm:mainexpandingtranslates} enable us to rule out the existence of a non-trivial distributional limit in the case of full circles, that is, when $\theta=4\pi$. More precisely, we establish the following result.
	
	\begin{thm}
		\label{thm:noCLT}
		Let $s>11/2$, $f\in W^{s}(M)$ a real-valued function. Assume that, for any positive Casimir eigenvalue $\mu$,  the functions $D^{\pm}_{4\pi,\mu}f$ defined as in Theorem~\ref{thm:mainexpandingtranslates} vanish identically on $M$. Then, for any collection $(B_T)_{T>0}$ of positive real numbers such that $B_T\to\infty$ as $T\to\infty$, the distributional limit of the random variables 
		\begin{equation*}
			\frac{e^{T}\bigl( \frac{1}{4\pi}\int_0^{4\pi}f\circ \phi_T^{X}\circ r_s(p)\;\emph{d}s-\int_{M}f\;\emph{d}\vol\bigr)}{B_T}\;,\quad p\sim \vol
		\end{equation*} 
		as $T\to\infty$ exists and is almost surely equal to zero.
	\end{thm}
	
	Theorem~\ref{thm:noCLT} results from approximating averages of $f$ along expanding circle arcs with the difference of the ergodic integrals of $f$ along two geodesic orbits, which in the case of complete circles happen to coincide; details are carried out in Section~\ref{sec:noCLT}.
	
	\begin{rmk}
		\label{rmk:geodesiccoboundary}
		We hasten to add that the argument we conduct enables to show that Theorem~\ref{thm:noCLT} holds true for arbitrary circle arcs (that is, it is possible to replace $4\pi$ with an arbitrary $\theta\in (0,4\pi]$) provided that, in addition to the vanishing hypothesis on the $D^{\pm}_{4\pi,\mu}$ for $\mu\in \text{Spec}(\square)\cap \R_{>0}$, the derivative $Uf_0$ along the stable horocycle flow of the projection $f_0$ of $f$ onto the Casimir eigenspace of eigenvalue $0$ is assumed to be a coboundary for the geodesic flow (cf.~Section~\ref{sec:noCLT}). 
	\end{rmk}

	\subsection{Asymptotics for arbitrary translates of compact orbits and the circle problem in the hyperbolic plane}
	\label{sec:circleproblemhyperbolic}
	
	The celebrated Gauss circle problem asks for the precise asymptotic behaviour of the discrepancy between the number of integer points in a disk of radius $R$ in the Euclidean plane and the area of the disk, as $R$ tends to infinity. More precisely, define the integer-point counting function 
	\begin{equation*}
		\cN(R)=|\{(m,n)\in \Z^{2}:m^{2}+n^2\leq R^{2}  \}|,\quad R\in \R_{>0},
	\end{equation*}
	where $|A|$ denotes, here and henceforth, the cardinality of a finite set $A$.
	Tessellating the Euclidean plane with $\Z^{2}$-translates of $[0,1)^{2}$, which is a fundamental domain for the $\Z^{2}$-action by translations on $\R^{2}$, leads to the main term $\pi R^{2}$, equal to the area of the disk of radius $R$, for the asymptotics of $\cN(R)$, as well as to the upper bound (due to Gauss himself)
	\begin{equation*}
		|\cN(R)-\pi R^{2}|\leq 2\pi(\sqrt{2}R+1)
	\end{equation*} 
	for the discrepancy. Despite considerable successive improvements on Gauss' original bound, for the history of which we refer to the comprehensive survey~\cite{Ivic}, it is a notoriously unsolved problem to attain the conjectural sharpest upper bound, deemed to be of the order of $R^{1/2+\eps}$ for any $\eps>0$.
	
	We consider the analogous question in the hyperbolic plane. For $\Ga<\SL_2(\R)$ a cocompact lattice, we examine the asymptotics of the function
	\begin{equation}
		\label{eq:defcountingfunction}
		N(R)=|\{z\in \Gamma\cdot i:d_{\Hyp}(z,i)\leq R\}|
	\end{equation}
	as $R$ tends to infinity, where $\Ga\cdot i$ denotes the (discrete) orbit\footnote{We thus count the number of actual lattice-orbit elements; we remark that, in the literature, the count of lattice elements $\gamma\in \Ga$ such that $d_{\Hyp}(\gamma\cdot i,i)\leq R$ often appears instead; the two quantities are proportional by a factor $|\text{Stab}_{\Ga}(i)|$.} of $i$ under the $\Ga$-action on the hyperbolic plane $\Hyp$ and we recall that $d_{\Hyp}$ is the hyperbolic distance function on $\Hyp$.
	
	\begin{rmk}
		Upon replacing $\Ga$ by a conjugate, there is no loss of generality in choosing $i\in \Hyp$ as the base point: an elementary algebraic computation, together with the fact that $\SL_2(\R)$ acts by $d_{\Hyp}$-isometries, leads to the equality
		\begin{equation*}
			|\{z\in \Ga\cdot w:d_{\Hyp}(z,w)\leq R  \}|=|\{z\in g^{-1}\Ga g\cdot i:d_{\Hyp}(z,i)\leq R\}|
		\end{equation*}
		for any $w=g\cdot i \in \Hyp$ and  $g\in \SL_2(\R)$.
	\end{rmk}	
	There is a compact fundamental domain for the action of $\Ga$ on $\Hyp$, and a transposition of Gauss' tesselation argument to this setup provides a rationale for the heuristics concerning the main term of the asymptotics, which once again should be proportional to the hyperbolic area measure of the ball $B_R=\{z\in \Hyp:d_{\Hyp}(z,i)\leq R \}$, which we denote by $m_{\Hyp}(B_R)$. However, a consequence of the peculiar features of hyperbolic geometry is that boundary effects become relevant, as opposed to the Euclidean setting: more precisely, the growth rate of the length of the boundary $\partial{B_R}$ turns out to be equal to the growth rate of $m_{\Hyp}(B_R)$. The error rate resulting from the tesselation approach is consequently of the same order of the main term, and as such meaningless.
	
	As for its Euclidean counterpart, the circle problem in the hyperbolic plane has been the subject of intensive research over the course of the twentieth century, with fundamental contributions due to Delsarte (\cite{Delsarte}), Selberg (\cite{Selberg}), Margulis (\cite{Margulis}), Lax and Phillips (\cite{Lax-Phillips}) and Phillips and Rudnick (\cite{Phillips-Rudnick}, see its introduction for a detailed history of the problem). To a large extent, the state of the art concerning the best estimate on the error term $|N(R)-c_{\Ga}m_{\Hyp}(B_R)|$,
	where $c_{\Ga}$ is an explicit constant which we identify in Theorem~\ref{thm:countingproblem}, 
	is represented by Selberg's upper bound
	\begin{equation}
		\label{eq:Selbergbound} |N(R)-c_{\Ga}m_{\Hyp}(B_R)|\leq e^{(\sup\{2/3,(1+\Re{\nu_*})/2\})R}\;,
	\end{equation}
	where recall that $(1-\nu_*^{2})/4$ is the spectral gap of $S=\Ga\bsl \Hyp$. The estimate\footnote{As a matter of fact, Selberg's asymptotics is more accurate than the one recorded here, featuring a main term which involves, beside the hyperbolic area of the balls, additional terms depending on Laplace eigenfunctions for small eigenvalues; see~\cite[Eq.~1.13]{Phillips-Rudnick}.} in~\eqref{eq:Selbergbound} (which is equally valid for non-uniform lattices $\Ga<\SL_2(\R)$) is obtained by means of a deep analysis of a  class of integral operators commuting with hyperbolic isometries (cf.~\cite{Selberg}).
	
	It was Margulis' realization (see~\cite{Margulis-thesis}) that lattice point counting problems of the type we are examining are intimately interwoven with questions of equidistribution of translates of subgroup orbits on homogeneous spaces. Subsequently, this novel perspective was profitably pursued and vastly generalized in the works of Duke, Rudnick and Sarnak~\cite{Duke-Rudnick-Sarnak} and Eskin and McMullen~\cite{Eskin-McMullen}. Specializing to our current setup, it turns out that the distribution properties of translates of $\SO_2(\R)$-orbits\footnote{We consider here the canonical left action $g\cdot \Ga g'=\Ga g'g^{-1}$ of $\SL_2(\R)$ on $M$.} on $M= \Ga\bsl \SL_2(\R)$ lead to meaningful information on the growth rate of $N(R)$, in a way that is amenable to quantitative refinements (see Section~\ref{sec:latticepoint} for an extensive treatment of the connection in its quantitative form).    
	
	We are thus lead to study effective equidistribution properties of $\SO_2(\R)$-orbits on $M$, which can be readily derived from Theorem~\ref{thm:mainexpandingtranslates} via the standard Cartan decomposition for $\SL_2(\R)$.  For any $g_0\in \SL_2(\R)$, define the right-translation map $R_{g_0}\colon M\to M$ by $R_{g_0}(\Ga g)=\Ga g g_0$ for any $g\in \SL_2(\R)$. For every $p \in M$, we indicate with $m_{\SO_2(\R)\cdot p}$ the unique $\SO_2(\R)$-invariant Borel probability measure on $M$ which is fully supported on the (compact) $\SO_2(\R)$-orbit of the point $p$; furthermore, for any $g\in \SL_2(\R)$, the notation $g_*m_{\SO_2(\R)\cdot p}$ stands for the push-forward of $m_{\SO_2(\R)\cdot p}$ under the action of $g$, which clearly depends only on the left coset of $g$ modulo $\SO_2(\R)$. Making use of the mixing properties of the global $\SL_2(\R)$-action on $M$ via a clever thickening argument, Margulis proved  (see~\cite{Margulis,Margulis-thesis}) that arbitrary translates of $m_{\SO_2(\R)\cdot p}$ equidistribute towards the $\SL_2(\R)$-invariant measure $\vol$; this amounts to the fact that, for any continuous function $f\colon M\to \C$, 
	\begin{equation*}
		\int_{M}f\;\text{d}g_*m_{\SO_2(\R)\cdot p}\longrightarrow\int_{M}\;f\;\text{d}\vol
	\end{equation*} 
	as $g\SO_2(\R)$ tends to infinity in the quotient $\SL_2(\R)/\SO_2(\R)$.
	
	In order to phrase a quantitative version of the previous statement conveniently, we introduce the notation $\norm{g}_{\text{op}}$ to indicate operator norm of an element $g\in \SL_2(\R)$ with respect to the standard Euclidean norm on $\R^{2}$; such a specific choice, while obviously immaterial, arises naturally over the course of the proof. 
	
	
	\begin{thm}
		\label{thm:mainarbitrarytranslates}
		Let $C_{\emph{Spec}},C'_{\emph{Spec}}$ be as in Theorem~\ref{thm:mainexpandingtranslates}, $s>11/2$ a real number. There exists a real constant $C_{s,\emph{rot}}$, depending only on $s$ and on $M$, such that the following holds:
		if $f\in W^{s}(M)$, then there exist, for any positive Casimir eigenvalue $\mu$, functions $D^{+}_{\mu}f,\;D^{-}_{\mu}f\colon M\times \SL_2(\R)\to \C$ with
		\begin{equation*}
			\sum_{\mu \in \emph{Spec}(\square)\cap \R_{>0}}\sup\limits_{p \in M,\;g\in \SL_2(\R)}|D^{\pm}_{\mu}f(p,g)|\leq C_{s,\emph{rot}}C_{\emph{Spec}}' C_{1,s-3}\norm{f}_{W^{s}}\;,
		\end{equation*}
		such that, for every $p \in M$ and every $g\in \SL_2(\R)$ with $\norm{g}_{\emph{op}}\geq \sqrt{e}$,
		\begin{equation*}
			\begin{split}
				\int_{M}f\;\emph{d}&g_{*}m_{\SO_2(\R)\cdot p}=\int_{M}f\;\emph{d}\vol\\
				&+\norm{g}_{\emph{op}}^{-1}\biggl(\sum_{\mu\in \emph{Spec}(\square),\;\mu>1/4}\cos{(\Im{\nu}\log{\norm{g}_{\emph{op}}})}D^{+}_{\mu}f(p,g)+\sin{(\Im{\nu}\log{\norm{g}_{\emph{op}}})}D^{-}_{\mu}f(p,g)\biggr)\\
				&+\sum_{\mu\in \emph{Spec}(\square),\;0<\mu<1/4}\norm{g}_{\emph{op}}^{-(1+\nu)}D^{+}_{\mu}f(p,g)+\norm{g}_{\emph{op}}^{-(1-\nu)}D^{-}_{\mu}f(p,g)\\
				&+\varepsilon_0\bigl(\norm{g}_{\emph{op}}^{-1}D^{+}_{1/4}f(p,g)+2\norm{g}_{\emph{op}}^{-1}\log{\norm{g}_{\emph{op}}}D^{-}_{1/4}f(p,g)\bigr)+\mathcal{R}f(p,g)\;,
			\end{split}
		\end{equation*}
		where 
		\begin{equation*}
			|\mathcal{R}f(p,g)|\leq C_{\emph{Spec}}C_{1,s-3}C_{s,\emph{rot}} \norm{f}_{W^{s}}(2\log{\norm{g}_{\emph{op}}}+1)\norm{g}_{\emph{op}}^{-2}\;.
		\end{equation*}
	\end{thm} 
	
	\begin{rmk}
		\begin{enumerate}
			\item The problem of effective equidistribution of translates of finite-volume orbits, in the vastly more general context of affine symmetric spaces, was thoroughly explored by Benoist and Oh in~\cite{Benoist-Oh}. Their approach relies crucially on effective bounds for the mixing rates\footnote{Specializing to our setup, effective mixing for the $\SL_2(\R)$-action on finite-volume quotients $\Ga\bsl \SL_2(\R)$ was first worked out in detail, to the best of our knowledge, by Kleinbock and Margulis in~\cite[Sec.~2.4]{Kleinbock-Margulis} (see also \cite{Ratner}).} of the relevant global action, and systematically developes quantitative versions of the geometric properties which play a major role in the original work~\cite{Eskin-McMullen} of Eskin and McMullen.  In the specific instance of the affine symmetric space being the hyperbolic plane $\Hyp$,  Theorem~\ref{thm:mainarbitrarytranslates} improves upon~\cite[Thm.~1.10]{Benoist-Oh} in that it quantifies the exponent governing the equidistribution rate and spells out additional terms in the asymptotic expansion.
			\item Just as in the case of Theorem~\ref{thm:mainexpandingtranslates} and Corollary~\ref{cor:effective}, the asymptotic expansion in Theorem~\ref{thm:mainarbitrarytranslates} delivers at once the (optimal) effective equidistribution bound
			\begin{equation*}
				\biggl|\int_{M}f\;\text{d}g_*m_{\SO_2(\R)\cdot p}-\int_{M}f\;\text{d}\vol\biggr|\leq D^{\text{main}}f(p,g)\;(\log{\norm{g}_{\text{op}}})^{\eps_0}\norm{g}_{\text{op}}^{-(1-\Re{\nu_*})}
			\end{equation*}
			for every $p \in M$ and $g\in \SL_2(\R)$ with $\norm{g}_{\text{op}}\geq \sqrt{e}$, 
			where the function $D^{\text{main}}\colon M\times \SL_2(\R)\to \C$ is uniformly bounded in terms of an appropriate Sobolev norm of $f$ and of spectral data depending only on $M$.  Following the thread of the observations expressed in Remark~\ref{rmk:equivalentmetrics}, it is instructive to  compare it with the decay rates for matrix coefficients of unitary representations of $\SL_2(\R)$ computed by Venkatesh in~\cite[Sec.~9.1.2]{Venkatesh}.
		\end{enumerate}
	\end{rmk}
	
	Theorem~\ref{thm:mainarbitrarytranslates} affords a precise asymptotic formula for the averaged counting of points on translates of $\Ga$-orbits inside balls of increasing radius, that is, for quantities of the form\footnote{As it is customary in the literature addressing such themes, we shall choose to work with spaces of left cosets whenever dealing with the lattice point counting problem. The map $g\Ga\mapsto \Ga g^{-1}$ establishes an $\SL_2(\R)$-equivariant diffeomorphism, between $\SL_2(\R)/\Ga$ and $\Ga\bsl \SL_2(\R)$; every object we have defined on $\Ga\bsl \SL_2(\R)$ shall thus be identified (without altering notation) with the corresponding object in $\SL_2(\R)/\Ga$ without further comment.}
	\begin{equation}
		\label{eq:averagecounting}
		\int_{\SL_2(\R)/\Ga}\frac{|g\Ga\cdot i\cap B_R|}{m_{\Hyp}(B_R)}\psi(g\Ga )\;\text{d}\vol(g\Ga)
	\end{equation}
	as $\norm{g}_{\text{op}}$ tends to infinity, where $\psi$ is a sufficiently regular function on $\SL_2(\R)/\Ga$. Prior to the statement of the result, we shall fix advantageous normalizations for the various invariant measures involved (see Section~\ref{sec:averagecounting} for the details).
	
	Let $m_{\SL_2(\R)}$ be the unique choice of Haar measure on $\SL_2(\R)$ such that, if $m_{\SO_2(\R)}$ is the probability Haar measure on the compact subgroup $\SO_2(\R)$, then $m_{\SL_2(\R)}$ is the (formal) product (cf.~Proposition~\ref{prop:foldingunfolding}) of $m_{\SO_2(\R)}$ and of the  $\SL_2(\R)$-invariant measure on the homogeneous space $\SL_2(\R)/\SO_2(\R)$ which corresponds, under the canonical identification of the latter space with $\Hyp$, to the hyperbolic area measure $m_{\Hyp}$. We then indicate with $m_{\SL_2(\R)/\Ga}$ the unique $\SL_2(\R)$-invariant measure on $\SL_2(\R)/\Ga$ such that $m_{\SL_2(\R)}$ is the product of $m_{\SL_2(\R)/\Ga}$ and the counting measure on the discrete group $\Ga$. Observe that $m_{\SL_2(\R)/\Ga}$ is a scalar multiple of the probability measure $\vol$, the multiplying factor being equal to the covolume
	\begin{equation*} \text{covol}_{\SL_2(\R)}(\Ga)=m_{\SL_2(\R)/\Ga}(\SL_2(\R)/\Ga)
	\end{equation*}
	of the lattice $\Ga$ inside $\SL_2(\R)$. We denote similarly by $\text{covol}_{\SO_2(\R)}(\Ga\cap \SO_2(\R))$ the volume\footnote{A straightfoward application of the formula in~\eqref{eq:foldingunfolding} shows that this equals the reciprocal of the cardinality of the finite group $\Ga\cap \SO_2(\R)$.} of the compact quotient $\SO_2(\R)/(\Ga\cap \SO_2(\R))$ with respect to the $\SO_2(\R)$-invariant measure induced by $m_{\SO_2(\R)}$ and the counting measure on $\Ga\cap \SO_2(\R)$.
	
	Lastly, recall that $B_R$ is the closed hyperbolic ball in $\Hyp$ of radius $R>0$ centered at $i$.
	
	\begin{prop}
		\label{prop:averagedcounting}
		Let $C_{\emph{Spec}},C_{\emph{Spec}}'$ be as in Theorem~\ref{thm:mainexpandingtranslates}. Suppose given a real number $s>11/2$ and a function $\psi\in W^{s}(\SL_2(\R)/\Ga)$. There exist, for any positive Casimir eigenvalue $\mu$, functions $\beta_{\psi,\mu}^{+},\beta_{\psi,\mu}^{-}\colon \R_{>0}\to \C$  with 
		\begin{equation}
			\label{eq:betabound}
			\sum_{\mu\in \emph{Spec}(\square)\cap \R_{>0}}\sup_{R>0}|\beta_{\psi,\mu}^{\pm}(R)|\leq \frac{C_{\emph{Spec}}'C_{1,s-3} }{2\pi}\norm{\psi}_{W^{s}}
		\end{equation}
		such that, for every $R\geq 1$, 
		\begin{equation}
			\label{eq:averagedcountingfunction}
			\begin{split}
				\frac{1}{\emph{covol}_{\SO_2(\R)}(\Ga\cap \SO_2(\R))}&\int_{\SL_2(\R)/\Ga} \frac{|g\Gamma\cdot i \cap B_R |}{m_{\Hyp}(B_R)}\;\psi(g\Ga )\;\emph{d}m_{\SL_2(\R)/\Ga}(g\Ga)=\\
				&\frac{1}{\emph{covol}_{\SL_2(\R)}(\Ga)}\int_{\SL_2(\R)/\Ga}\psi\;\emph{d}m_{\SL_2(\R)/\Ga}\\
				&
				+ e^{-\frac{R}{2}}\sum_{\mu\in \emph{Spec}(\square),\;\mu>1/4}\beta_{\psi,\mu}^{+}(R)+\beta_{\psi,\mu}^{-}(R)\\
				&
				+\sum_{\mu\in \emph{Spec}(\square),\;0<\mu<1/4}e^{-\frac{1+\nu}{2}R}\beta^{+}_{\psi,\mu}(R)+e^{-\frac{1-\nu}{2}R}\beta^{-}_{\psi,\mu}(R)\\
				& +\varepsilon_0\biggl(e^{-\frac{R}{2}}\beta^{+}_{\psi,1/4}(R)+Re^{-\frac{R}{2}}\beta^{-}_{\psi,1/4}(R)\biggr)+\gamma_{\psi}(R)\;,
			\end{split}
		\end{equation}
		where 
		\begin{equation}
			\label{eq:gammabound}
			|\gamma_{\psi}(R)|\leq \frac{5C_{\emph{Spec}}C_{1,s-3}}{4\pi}\norm{\psi}_{W^{s}} (R+1)e^{-R}\quad .
		\end{equation}
	\end{prop}
	
	Recall now our definition of the counting function $N(R)$ in~\eqref{eq:defcountingfunction}. An optimization argument involving approximate identities on the homogeneous space $\SL_2(\R)/\Ga$ enables us to derive information on the asymptotic behaviour of the error term in the pointwise counting problem discussed at the beginning of this subsection. We remind the reader that $\nu_*$ is the unique complex number in $\R_{\geq 0}\cup i\R_{>0}$ such that $\frac{1-\nu_*^{2}}{4}$ equals the spectral gap of $S=\Ga\bsl \Hyp$.
	
	\begin{thm}
		\label{thm:countingproblem}
		Let $\Sigma\colon \R_{>0}\to \R_{>0}$ be the function defined by
		\begin{equation*}
			\Sigma(R)=\frac{\emph{covol}_{\SO_2(\R)}(\Ga\cap \SO_2(\R))}{\emph{covol}_{\SL_2(\R)}(\Ga)}m_{\mathbb{H}}(B_R)\;, \quad R>0.
		\end{equation*}
		Set $\eta_*=\frac{1}{13}(1-\Re{\nu_*})$. Then,
		for every $\eps>0$, 
		\begin{equation*}
			N(R)=\Sigma(R)+o(e^{(1-\eta_*+\eps)R})
		\end{equation*}
		as $R$ tends to infinity.
	\end{thm}
	
	\begin{rmk}
		As it will emerge in the proof of Theorem~\ref{thm:countingproblem}, which is detailed in Section~\ref{sec:countingproblem}, the appearance of the quantity $1/13$ in the exponent is ultimately an outgrowth of the minimal Sobolev regularity of the test function $f$ we need to impose in Theorem~\ref{thm:expandingonsurface}, which in turn is needed because of the upper bounds in~\eqref{eq:boundDthetamu} and~\eqref{eq:globalremainderestimate} depending on the Sobolev norm $\norm{f}_{W^s}$ for some $s>9/2$. In this regard, we observe the following: suppose that, in the latter two bounds, the norm $\norm{f}_{W^{s}}$ can be replaced by $\norm{f}_{W^1}$, as it is the case for joint eigenfunctions of $\square$ and $\Theta$ (cf.~Theorem~\ref{thm:case_Theta_eigenfn}); then Theorem~\ref{thm:countingproblem} would hold with $1/6$ in place of $1/13$. 
	\end{rmk}
	
	\subsection{Outline of the proofs and layout of the article}
	
	The method we employ to prove Theorem \ref{thm:case_Theta_eigenfn} was originally devised by Ratner in~\cite{Ratner}, who realized that the problem of finding mixing rates for geodesic and horocycle flows can be reduced to solving a family of linear second-order ordinary differential equations\footnote{After the completion of a first draft of the present article, the authors were made aware of the unpublished manuscript~\cite{Edwards-unpublished} by S.~Edwards, in which a weaker formulation of the quantitative equidistribution result in Theorem~\ref{thm:mainexpandingtranslates} is provided. The strategy of proof is entirely analogous to the one pursued here, and is there applied, more generally, to the quantitative investigation of the equidistribution properties of translated orbits of symmetric subgroups on homogeneous spaces of semisimple Lie groups.}. 
	This ingenuous and yet fairly elementary approach has been further developed by Burger in~\cite{Bur} to prove polynomial bounds for the equidistribution of horocycle orbits in compact quotients of $\SL_2(\R)$. Later, Str\"{o}mbergsson \cite{Str} and Edwards \cite{Edw} exploited the same idea to study effective equidistribution properties of unipotent orbits in more general settings.
	In the same spirit, the second author recently provided (see~\cite{Rav}), using Ratner's strategy, an alternative proof of Flaminio-Forni's asymptotic expansion for horocycle ergodic integrals \cite{Flaminio-Forni}. 
	
	We begin in Section~\ref{sec:preliminaries} with an overview of the required notions concerning hyperbolic surfaces, Sobolev spaces and harmonic analysis on the Lie group $\SL_2(\R)$, nailing down notation to be employed throughout the manuscript.
	Assuming that a $\cC^{2}$-observable $f$ is a joint eigenfunction of the Casimir operator and of the vector field $\Theta$, we then show in Section~\ref{sec:reductiontoODE} that the behaviour of the circle-arc averages $k_{f, \theta}(p, \cdot)$ (cf.~\eqref{eq:ktheta}), viewed as functions of the time $t$ for a fixed base point $p \in \Ga\bsl \SL_2(\R)$, fulfill a second order linear ODE, solving which explicitly leads to the proof of Theorem~\ref{thm:case_Theta_eigenfn} presented in Section~\ref{sec:jointeigenfn}; incidentally, we may arrange computations so that the latter takes on the same form of the ODE satisfied by time rescalings of horocycle averages in~\cite{Rav} (see, in particular,~\cite[Prop.~8]{Rav}), which accounts for the similarity between Theorem~\ref{thm:case_Theta_eigenfn} with~\cite[Thm.~1]{Rav}.
	In Section~\ref{sec:arbitraryfn}, the asymptotic expansion of Theorem \ref{thm:mainexpandingtranslates} is deduced from Theorem \ref{thm:case_Theta_eigenfn} taking advantage of a few elementary facts from the classical harmonic analysis of $\SL_2(\R)$. Additional regularity on $f$ is required in order to ensure the absolute convergence of the expansion in \eqref{eq:asymptoticgeneral}; see, in particular, Section~\ref{sec:sumestimates}. The asymptotics for arbitrary translates in Theorem~\ref{thm:mainarbitrarytranslates} is derived from Theorem~\ref{thm:mainexpandingtranslates} in Section~\ref{sec:arbitrarytranslates}.
	Building on Theorem~\ref{thm:mainexpandingtranslates} once more, we establish in Sections~\ref{sec:spatialDLT} and~\ref{sec:noCLT} the distributional limit Theorems~\ref{thm:CLT} and~\ref{thm:noCLT} for the random variable $k_{f,\theta}(p,t)$, appropriately centered and normalized, when the initial point $p$ is taken randomly with respect to the uniform measure.
	These limit theorems mirror those for horocycle ergodic integrals, for which we refer the reader to~\cite{BuFo, Rav}. Section~\ref{sec:temporalDLT} hosts a few considerations concerning the point of view of temporal distributional limit theorems (see~\cite{Dolgopyat-Sarig}) on the problem of analyzing the statistical behaviour of circle averages. 
	Finally, in Section~\ref{sec:latticepoint}, we provide a quantitative treatment of the approach of Duke-Rudnick-Sarnak~\cite{Duke-Rudnick-Sarnak} and Eskin-McMullen~\cite{Eskin-McMullen}, which allows to prove both Proposition~\ref{prop:averagedcounting} and Theorem~\ref{thm:countingproblem} on the hyperbolic lattice point counting problem.

	\section{Preliminaries on harmonic analysis on $\SL_2(\R)$}
	\label{sec:preliminaries}
	It is the aim of this section to carefully describe the setting of our main results as well as to review the required notions on the representation theory of the group $\SL_2(\R)$ which will play a central role throughout the article.

	\subsection{Hyperbolic surfaces and their unit tangent bundles}
	\label{sec:hyperbolic}
	The subject of this subsection is classical: detailed treatments can be found, for instance, in~\cite{Bekka-Mayer,Bergeron,Borel,Buser,Einsiedler-Ward,Iwaniec,Katok}.
	
	The special linear group $\SL_2(\R)$ is the group of $2\times 2$ real matrices with determinant $1$. It is a three-dimensional Lie group, whose Lie algebra we denote by $\sl_2(\R)$ and identify canonically with the Lie algebra of traceless $2\times 2$  matrices with real entries. The identity matrix in $\SL_2(\R)$ is denoted by $I_2$. A basis of $\sl_2(\R)$ as a real vector space is given by the elements
	\begin{equation*}
		X=
		\begin{pmatrix}
			1/2&0\\0&-1/2
		\end{pmatrix}
		, \quad U=
		\begin{pmatrix}
			0&1\\
			0&0
		\end{pmatrix}
		, \quad V=
		\begin{pmatrix}
			0&0\\
			1&0
		\end{pmatrix}
		.
	\end{equation*}
	With $\exp\colon \sl_2(\R)\to \SL_2(\R)$ we indicate the exponential map, and with $\Ad\colon \SL_2(\R)\to \GL(\sl_2(\R))$, $g\mapsto \Ad_g$ the adjoint representation of $\SL_2(\R)$ onto $\sl_2(\R)$.

	Let $\Hyp=\{z=x+iy\in \C:y>0 \}$ be the Poincar\'{e} upper-half plane, endowed with the Riemannian metric
	\begin{equation*}
		g_{(x,y)}=\frac{(\text{d}x)^{2}+(\text{d}y)^{2}}{y^2}\;, \quad (x,y)\in \Hyp.
	\end{equation*} 
	The Riemannian manifold $(\Hyp,g)$ is a model of the hyperbolic plane, that is, of the unique complete simply connected two-dimensional Riemannian manifold of constant sectional curvature equal to $-1$ (cf.~\cite[Part 1 Chap.~6]{Bridson-Haefliger}). 
	The Lie group $\SL_2(\R)$ acts smoothly by orientation-preserving\footnote{Actually, the action is by analytic transformations of the Riemann surface $\Hyp$.} isometries of the hyperbolic plane. The action is given by the M\"{o}bius transformations
	\begin{equation*}
		\begin{pmatrix}
			a&b\\
			c&d
		\end{pmatrix}
		\cdot z=\frac{az+b}{cz+d}\;,\quad a,b,c,d\in \R, \;ad-bc=1,\;z\in \Hyp;
	\end{equation*}
	it is transitive, and thus gives rise to an $\SL_2(\R)$-equivariant diffeomorphism between $\Hyp$ and the quotient manifold $\SL_2(\R)/\SO_2(\R)$, where the special orthogonal group $\SO_2(\R)$ is the stabilizer of the point $i\in \Hyp$.
	
	Let $\Ga<\SL_2(\R)$ be a cocompact lattice\footnote{We recall that a lattice in a locally compact Hausdorff topological group $G$ is a discrete subgroup $\La<G$ such that the quotient space $\La\bsl G$ admits a non-zero $G$-invariant Radon probability measure. The lattice $\La$ is cocompact if $\La\bsl G$ is a compact topological space.}. If the projection of $\Ga$ to the projective special linear group $\PSL_2(\R)=\SL_2(\R)/\{\pm I_2 \}$ has no non-trivial torsion elements, then the quotient space $S=\Ga\bsl \Hyp$, homeomorphic to the double coset space $ \Ga\bsl \SL_2(\R)/\SO_2(\R)$, is a compact, connected, orientable smooth surface inheriting from $\Hyp$ a Riemannian metric $g_{S}$ of constant sectional curvature equal to $-1$, that is, a hyperbolic Riemannian metric.  Conversely, it is a well-known consequence of Poincar\'{e}-Koebe's Uniformization Theorem (cf.~\cite[Chap.~IV]{Farkas-Kra}) that any compact, connected, orientable hyperbolic surface is isomorphic, as a Riemannian manifold, to a quotient $\Ga\bsl\Hyp$, with $\Ga<\SL_2(\R)$ a cocompact lattice with torsion-free projection on $\PSL_2(\R)$. In case the projection of $\Ga$ to $\PSL_2(\R)$ has non-trivial torsion elements, the quotient $S=\Ga\bsl \Hyp$ is, more generally, a hyperbolic \emph{orbifold} (see~\cite[Chap.~13]{Ratcliffe} and~\cite[Chap.~13]{Thurston}). Unifying, though mildly abusing terminology, we shall throughout refer to $S$ as a surface in both the previous cases. 
	
	The simply transitive action of $\PSL_2(\R)$ on the unit tangent bundle of $\Hyp$ allows to identify, as smooth manifolds, the unit tangent bundle $T^{1}S=\{(p,v)\in TS:\norm{v}_{g_S}=1 \}$ of $S$ with the homogeneous space $M=\Ga\bsl \SL_2(\R)$. Throughout this manuscript, we shall solely appeal to the algebraic structure of $M$ as a homogeneous space of $\SL_2(\R)$, and not to its geometric nature of principal circle bundle over the surface $S$.
	
	For any $r\in \N\cup\{\infty\}$, we denote by $\mathscr{C}^r(M)$ the vector space of complex-valued functions of class $\mathscr{C}^{r}$ defined on $M$. It is endowed with the norm $\norm{\cdot}_{\cC^{r}}$ defined in~\eqref{eq:Cknorm}. The $\mathscr{C}^{0}$-norm is abbreviated with $\norm{\cdot}_{\infty}$.
	
	A vector $W\in \sl_2(\R)$ gives rise to the one-parameter subgroup $(\exp(tW))_{t\in \R}$ of $\SL_2(\R)$, which in turn acts as a smooth flow on $M$ by right translations. Whenever convenient, we identify $W$ with the infinitesimal generator of such a flow, and denote its action as a derivation on the function space $\mathscr{C}^{1}(M)$ by $f\mapsto Wf$. This extends to an isomorphism of associative $\R$-algebras between the universal enveloping algebra $U(\sl_2(\R))$ of $\sl_2(\R)$ and the algebra of $\SL_2(\R)$-invariant differential operators on $\mathscr{C}^{\infty}(M)$. More generally, any element $Y\in U(\sl_2(\R))$ of order $k\in \N$ acts naturally on $\mathscr{C}^{k}(M)$, the action being also denoted by $f\mapsto Yf$. Observe that, for any $k\in \N$, the definition of the $\mathscr{C}^k$-norm on $\mathscr{C}^k(M)$ implies that $\norm{Yf}_{\infty}\leq \norm{f}_{\mathscr{C}^{k}}$ for any $Y\in U(\sl_2(\R))$ of order at most $k$ and any $f\in \mathscr{C}^{k}(M)$.
	
	\subsection{Unitary representations, the Casimir and Laplace operators, Weyl's law}
	\label{sec:unitaryrepresentations}
	
	Convenient sources for the material presented in this subsection, in addition to those cited in Section~\ref{sec:hyperbolic}, are~\cite{Lang, Mackey,Sarnak, Terras}.
	
	Let us denote by $H$ be the complex Hilbert space $L^{2}(M)$ of functions which are square-integrable with respect to the $\SL_2(\R)$-invariant measure $\vol$ on $M$; the inner product defining the Hilbert space structure on $H$ shall be denoted by $\langle \cdot, \cdot \rangle$. With $\mathscr{U}(H)$ we indicate the group of unitary operators $T\colon H\to H$.
	The standard smooth (left) action of $\SL_2(\R)$ on the homogeneous space $M$, given by $g_0\cdot \Ga g=\Ga gg_0^{-1}$ for every $g_0,g\in \SL_2(\R)$, preserves the measure $\vol$, and hence gives rise to a unitary representation $\rho\colon \SL_2(\R)\to \mathscr{U}(H),\;g\mapsto \rho_g$, called the quasi-regular representation of $\SL_2(\R)$ on $H$. 
	
	For any $r\in \N\cup\{\infty\}$, denote by $\mathscr{C}^{r}(H)$ the linear subspace of $\mathscr{C}^{r}$-vectors on $H$, that is,
	\begin{equation*}
		\mathscr{C}^{r}(H)=\{v\in H: \text{ the map }\SL_2(\R)\ni g\mapsto \rho_g(v)\in H \text{ is of class }\mathscr{C}^{r}  \}.
	\end{equation*}
	
	An element $W\in \sl_2(\R)$ induces a linear operator $L_{W}\colon \mathscr{C}^{1}(H)\to H$, called the Lie derivative with respect to $W$, defined by
	\begin{equation}
		\label{eq:Liederivative}
		L_W(v)=\lim_{t\to 0}\frac{\rho_{\exp(tW)}(v)-v}{t}\;, \quad v\in \mathscr{C}^{1}(H).
	\end{equation}
	Observe that, for a $\mathscr{C}^{1}$-function $f$ on $M$, we have $L_W(f)=Wf$, regarding $\mathscr{C}^{1}(M)$ as canonically embedded in $H$.
	
	The center of the universal enveloping algebra $U(\sl_2(\R))$ is  generated, as an algebra, by any of its non-zero elements, called Casimir elements. For reasons to be clarified shortly, we choose the normalization of the Casimir element to be $\square=-X^{2}+X+UV$, where $\{X,U,V\}$ is identified with a set of generators of the $\R$-algebra $U(\sl_2(\R))$. By the universal property of the universal enveloping algebra, the Casimir element $\square$ acts as a second-order linear differential operator $\mathscr{C}^{2}(M)\to \mathscr{C}(M)$, and (in a compatible manner) as an unbounded operator on $H$, densely defined on the subspace $\mathscr{C}^{2}(H)$. Slightly abusing notation, we shall again indicate with $\square$ the Casimir operator arising in this fashion. It is a consequence of the Casimir element lying in the center of $U(\sl_2(\R))$ that the operator $\square$ commutes with all unitary operators $\rho_g,\;g\in \SL_2(\R)$ and all Lie derivatives $L_{W},\;W\in \sl_2(\R)$, wherever they are simultaneously defined.
	
	The Casimir operator $\square$ is an essentially self-adjoint operator on $H$; as $M$ is compact, its spectrum $\text{Spec}(\square)$ consists solely of eigenvalues and is a discrete subset of $\R$. Our choice of the normalization of the Casimir element implies that, on the subspace $\mathscr{C}^{2}(S)$ of continuously twice-differentiable functions defined on the surface $S=\Ga\bsl \Hyp$, the Casimir operator acts as the Laplace-Beltrami operator $\Delta_{S}$ with respect to the hyperbolic Riemannian metric on $S$. As a consequence, the spectrum of $\Delta_{S}$ is contained in the spectrum of $\square$. As a matter of fact, it holds that $\text{Spec}(\Delta_S)=\text{Spec}(\square)\cap \R_{\geq 0}$. 
	
	A great deal of research has revolved around the properties of the spectrum of $\Delta_S$; for the purposes of this article, we shall only need a weaker version (cf.~the proof of Lemma~\ref{lem:finitespectralconstant}) of the celebrated Weyl's law concerning the asymptotics of the eigenvalues, which we shall now recall for completeness. 
	
	\begin{thm}[Weyl's law]
		\label{thm:Weyllaw}
		Let $\Ga<\SL_2(\R)$ be as above,  $S=\Ga\bsl \Hyp$, and 
		\begin{equation*}
			\lambda_0<\lambda_1\leq \cdots \leq \lambda_n\leq \cdots \to\infty
		\end{equation*} 
		be the eigenvalues, counted with multiplicity, of the Laplace-Beltrami operator $\Delta_S$ associated to the hyperbolic Riemannian metric on $S$. Then
		\begin{equation*}
			\lim\limits_{R\to\infty}\frac{|\{j\in \N:\lambda_j\leq R\}|}{ R}=\frac{\emph{Area}(S)}{4\pi}\;,
		\end{equation*}
		where $\emph{Area}(S)$ is the total mass of $S$ with respect to the hyperbolic area measure.
	\end{thm}
	A proof of Theorem~\ref{thm:Weyllaw} relying on Selberg's trace formula is given in~\cite[Prop.~10]{Marklof}.
	
	\medskip
	On the other hand, the negative eigenvalues of $\square$ are completely understood, being of the form $-m(m+2)/4$, with $m$ ranging across the set $\N^{\ast}$ of strictly positive integers. 
	
	In general, any unitary representation of $\SL_2(\R)$ is unitarily equivalent to a direct integral of irreducible representations. Owing to compactness of $M$, the representation space $H$ of $\rho$ actually decomposes as a Hilbert direct sum of (non-zero) irreducible $\rho$-invariant subspaces:
	\begin{equation*}
		H=\bigoplus_{i \in I}H_i\;,\quad \rho(H_i)\subset H_i,\; \rho|_{H_i} \text{ irreducible},\;I \text{ countable index set}.
	\end{equation*}
	As the Casimir operator commutes with all the $\rho_g$ for $g\in \SL_2(\R)$, it restricts to an operator $\mathscr{C}^2(H_i)\to H_i$ for each $i \in I$. By virtue of Schur's lemma for unbounded linear operators, there exists $\mu \in \text{Spec}(\square)$, depending on $i$, such that $\square v=\mu v$ for any $v\in\mathscr{C}^{2}(H_i)$. Regrouping all irreducible subspaces according to the corresponding Casimir eigenvalue, we obtain an orthogonal splitting into $\rho$-invariant subspaces
	\begin{equation}
		\label{eq:deccasimireigenvalues}
		H=\bigoplus\limits_{\mu \in \text{Spec}(\square)}H_{\mu}\;,\quad H_{\mu}=\{v\in \mathscr{C}^{2}(H):\square v =\mu v \}.
	\end{equation}
	
	Let $\Theta\in \sl_2(\R)$ be the element defined in~\eqref{eq:rotationflow}, inducing the unbounded linear operator $L_{\Theta}$ on $H$. As it is the case for any representation space of a unitary representation of $\SL_2(\R)$, each subspace $H_{\mu}$ is the internal Hilbert direct sum 
	\begin{equation}
		\label{eq:decthetaeigenvalues}
		H_{\mu}=\bigoplus_{n\in \Z}H_{\mu,n}\;,\quad H_{\mu,n}=\biggl\{v\in \mathscr{C}^{1}(H_{\mu}):L_\Theta v=\frac{i}{2}nv \biggr\}.
	\end{equation}  
	
	\begin{rmk}
		Observe that, in the previous decomposition, it might occur (as it manifestly emerges in the equation~\eqref{eq:differentSobolev} below) that $H_{\mu,n}$ is the trivial subspace for some $\mu\in \text{Spec}(\square)$ and $n\in \Z$. For this reason, in the sequel, we let the index set of the direct sum in~\eqref{eq:decthetaeigenvalues} be rather $I(\mu)=\{n\in \Z:H_{\mu,n}\neq \{0\} \}\subset \Z$, bearing in mind this caveat.
	\end{rmk}
	
	\subsection{Sobolev spaces and the Sobolev Embedding Theorem}
	\label{sec:Sobolev}

	The classical theory of Sobolev spaces is thoroughly discussed in~\cite{Adams}. In the context of arbitrary Riemannian manifolds, it is presented, for instance, in~\cite{Aubin,Hebey}; the extension to fractional Sobolev orders is treated e.g.~in~\cite{Strichartz-Sobolev,Triebel}. Here we shall confine ourselves to homogeneous spaces of $\SL_2(\R)$, placing emphasis on how this theory is brought to bear on the study of unitary representations of the special linear group. 
	
	Define the element
	\begin{equation*}
		Y=
		\begin{pmatrix}
			0&-1/2\\
			-1/2&0
		\end{pmatrix}
		\in \sl_2(\R),
	\end{equation*}
	so that $\{X,\Theta,Y \}$ forms a basis of the real vector space $\sl_2(\R)$.
	Define the second-order linear operator $\Delta= -(X^{2}+Y^{2}+\Theta^{2}) =\square -2\Theta^{2}$ on the subspace $\mathscr{C}^{2}(H)$. As each Lie derivative $L_W$ ($W\in \sl_2(\R)$) satisfies\footnote{This follows at once from the definition of $L_W$ in~\eqref{eq:Liederivative} and the fact that $\rho_g$ is a unitary operator for any $g\in \SL_2(\R)$.} $\langle L_W u,v \rangle =-\langle u,L_Wv\rangle$ for any $u,v \in \mathscr{C}^{1}(H)$, it is clear that $\langle \Delta u,v\rangle =\langle u, \Delta v\rangle$ for any $u,v\in \mathscr{C}^{2}(H)$, that is, $\Delta$ is self-adjoint on its domain of definition. We shall refer to it as the Laplace operator on $H$. Akin to the Casimir operator, $\Delta$ acts as the Laplace-Beltrami operator $\Delta_S$ on the subspace $\mathscr{C}^{2}(S)$. Indeed, any vector $u \in \mathscr{C}^{2}(S)$ is invariant under the action of $\SO_2(\R)$ by right translations, and hence lies in the kernel of $L_{\Theta}$; as a consequence, 
	\begin{equation*}
		\Delta u=(\square -2\Theta^{2})u=\square u-2L_{\Theta}^{2}u=\square u=\Delta_S u\;. 
	\end{equation*}
	
	For any $s\in \R_{>0}$, we define the Sobolev space of order $s$ on $H$, denoted by $W^{s}(H)$, as the maximal linear subspace of $H$ on which the unbounded linear operator $\Delta^{s/2}$ can be defined, and endow it with the inner product given by
	\begin{equation}
		\label{eq:Sobolevinnerproduct}
		\langle u,v\rangle_{W^{s}}=\langle (I+\Delta)^{s}u,v\rangle\;, \quad u,v \in W^{s}(H),
	\end{equation}
	where $I$ denotes the identity operator on $H$.
	This assignment turns $W^{s}(H)$ into a Hilbert space, whose associated norm is denoted by $\norm{\cdot }_{W^{s}}$. Similarly, we define the Sobolev spaces $W^{s}(H_{\mu})$ and $W^{s}(H_{\mu,n})$ for any $\mu \in \text{Spec}(\Z)$ and $n\in I(\mu)$. It is a fact that the decompositions in~\eqref{eq:deccasimireigenvalues} and~\eqref{eq:decthetaeigenvalues} induce analogous decompositions on the level of Sobolev spaces, namely there are orthogonal\footnote{Clearly, we intend that the closed subspaces $W^{s}(H_{\mu,n})$ are orthogonal with respect to the $W^{s}$-inner product defined in~\eqref{eq:Sobolevinnerproduct}. } splittings
	\begin{equation}
		\label{eq:Sobolevdecomposition}
		W^{s}(H)=\bigoplus_{\mu \in \text{Spec}(\square)}W^{s}(H_{\mu})=\bigoplus_{\mu \in \text{Spec}(\square)}\bigoplus_{n\in I(\mu)}\;W^{s}(H_{\mu,n})\;.
	\end{equation}
	
	The argument in Section~\ref{sec:arbitraryfn} crucially hinges upon the following elementary relationship between Sobolev norms of different order:
	
	\begin{lem}
		\label{lem:diffSobolevnorms}
		Let $s\in \R_{>0},\;k\in \N$, and assume $u\in W^{s+k}(H_{\mu,n})$ for some $\mu \in \emph{Spec}(\square)$ and $n \in I(\mu)$. Then
		\begin{equation}
			\label{eq:differentSobolev}
			\norm{u}_{W^{s+k}}^2=\biggl(1+\mu+\frac{n^2}{2}\biggr)^{k}\norm{u}_{W^s}^2\;.
		\end{equation}
	\end{lem} 
	\begin{proof}
		Suppose $k=1$. We may write, using self-adjointness of $\Delta$ with respect to the $L^{2}$-inner product,
		\begin{equation*}
			\norm{u}_{W^{s+1}}^{2}=\langle u,u\rangle_{W^{s+1}}=\langle (I+\Delta)^{s+1}u,u\rangle=\langle (I+\Delta)^{s}u,(I+\Delta)u\rangle=\langle(1+\Delta)^{s}u,u\rangle +\langle (1+\Delta)^{s}u,\Delta u \rangle\;.
		\end{equation*}
		By the assumption on $u$, it holds that $\Delta u=(\square-2\Theta^{2})u=\mu-2(\frac{i}{2}n)^{2}u=(\mu+\frac{n^2}{2})u$. Therefore, we infer
		\begin{equation*}
			\norm{u}_{W^{s+1}}^{2}=\norm{u}_{W^{s}}^2+\biggl(\mu+\frac{n^{2}}{2}\biggr)\langle (1+\Delta)^{s}u,u\rangle = \biggl(1+\mu+\frac{n^{2}}{2}\biggr)\norm{u}_{W^{s}}^{2}\;,
		\end{equation*}
		as desired.
		
		The statement for an arbitrary $k\in \N$ is immediately achieved arguing by induction.
	\end{proof}
	
	Observe that~\eqref{eq:differentSobolev} readily implies the following: if $u \in W^{s+k}(H)$ for some $s\in \R_{>0}$ and $k\in \N$,   and 
	\begin{equation*}
		u=\sum_{\mu\in \text{Spec}(\square)}\sum_{n\in I(\mu)}u_{\mu,n}
	\end{equation*}
	is its decomposition into joint eigenvectors for $\square$ and $\Theta$ provided by~\eqref{eq:Sobolevdecomposition}, then the larger the integer $k$ is, the faster the decay of the Sobolev norms $\norm{u_{\mu,n}}_{W^{s}}$ as $|n|$ and $|\mu|$ tend to infinity. This phenomenon\footnote{The counterpart of this relationship in classical Fourier analysis is well-known; the regularity of a function is closely interwoven with the decay rate at infinity of its Fourier coefficients.} is going to be essential in our estimates over the course of the proof of Theorem~\ref{thm:mainexpandingtranslates}.
	
	\smallskip
	We conclude this section recalling a version for compact three-manifolds of the celebrated Sobolev Embedding Theorem, which will be sufficient for our purposes.

	\begin{thm}[Sobolev Embedding Theorem]
		\label{thm:Sobolevembedding}
		For any $r\in \N$ and $s\in \R_{>0}$ fulfilling the inequality $s-r>3/2$, there is a continuous embedding of $W^{s}(M)$ into the Banach space $\mathscr{C}^{r}(M)$: in particular, there exists a constant $C_{r,s}>0$ such that 
		\begin{equation*}
			\norm{f}_{\mathscr{C}^{r}}\leq C_{r,s}\norm{f}_{W^{s}}
		\end{equation*}
		for every $f\in W^{s}(M)$.
	\end{thm}
	
	\section{Reduction to an ordinary differential equation}
	\label{sec:reductiontoODE}
	This section presents the gist of the approach we pursue in order to prove Theorem~\ref{thm:case_Theta_eigenfn}, which concerns the asymptotic behaviour of circle-arc averages of joint eigenfunctions of the operators $\square$ and $\Theta$ (cf.~Section~\ref{sec:unitaryrepresentations}); the partial differential equations (classically known in the literature as eigenvalue equations) expressing the eigenfunction condition are here shown to give rise to ordinary differential equations for the corresponding circle averages, when the latter are seen as functions of the time parameter.
	
	We fix a function $f\colon M\to \C$ of class $\cC^{2}$ and a parameter $\theta\in (0,4\pi]$. Recall from~\eqref{eq:ktheta} the definition of the averages $k_{f,\theta}(p,t)$, for $p \in M$ and $t\in \R$. As we shall work with a fixed, arbitrary base point $p \in M$, we shall abbreviate, for notational convenience, $k_{f,\theta}(p,t)$ with $k_{\theta}(t)$ in the computations that follow.
	
	Our goal is to show that the function $k_{\theta}(t)$ satisfies a second-order linear ODE. In the upcoming computations, the following lemma will be of use. For any left-invariant vector field $W\in \sl_2(\R)$, we indicate with $(\phi_t^{W})_{t\in \R}$ be the one-parameter flow on $M$ defined by $\phi^{W}_t(\Ga g)=\Ga g \exp{tW}$ for any $t\in \R$ and $g\in \SL_2(\R)$. For any pair $Y,W\in \sl_2(\R)$ and any point $q\in M$, the derivative of the smooth curve $s\mapsto \phi_{t}^{Y}\circ \phi^{W}_s(q)$ (seen as a function from $\R$ to the tangent bundle of $M$), where $t\in \R$ is fixed, is denoted by $\frac{\text{d}}{\text{d}s}\;\phi_t^{Y}\circ \phi_s^{W}(q)$.
	Lastly, if $W\in \sl_2(\R)$ and $q\in M$, we denote by $W_q$ the value at $q$ of the infinitesimal generator of the smooth flow $(\phi_t^{W})_{t\in \R}$ on $M$.
	\begin{lem}
		\label{lem:shiftedderivatives}
		For every $Y,W\in \sl_2(\R)\setminus \{ 0\}$ and every $q\in M$, it holds
		\begin{equation*}
			\frac{\emph{d}}{\emph{d}s}\;\phi_t^{Y}\circ \phi^{W}_s(q)=\Ad_{\exp{(-tY)}}(W)_{\phi^{Y}_t\circ \phi_s^{W}(q)}\;.
		\end{equation*} 
	\end{lem}
	\begin{proof}
		It follows from elementary algebraic manipulations, see~\cite[Lem.~4]{Rav-arcs}.
	\end{proof}
	
	We may now state the main result of this section.
	
	\begin{prop}
		\label{prop:ODE}
		Let $\mu \in \emph{Spec}(\square)$, $n\in \Z$ and $f\in \cC^{2}(M)$ be a function satisfying $\square f=\mu f$, $\Theta f=\frac{i}{2}n f$. For every $p \in M$ and $\theta\in (0,4\pi]$, there is a bounded continuous function $G_{\theta,n}f(p,\cdot)\colon \R_{>0}\to \C$ such that the function $k_{\theta}(t)=\frac{1}{\theta}\int_0^{\theta}f\circ \phi_t^{X}\circ r_s(p)\;\emph{d}s$ satisfies the linear ordinary differential equation
		\begin{equation}
			\label{eq:ode}
			k''_{\theta}(t)+k'_{\theta}(t)+\mu k_{\theta}(t)=e^{-t}G_{\theta,n}(p,t)
		\end{equation}
		for any $t>0$.
	\end{prop}
	\begin{proof}
		Fix $f$, $p$ and $\theta$ as in the assumptions. Since $f$ is of class $\cC^{2}$ on $M$, differentiation under the integral sign gives
		\begin{equation*}
			k'_{\theta}(t)=\frac{1}{\theta} \int_0^{\theta}Xf\circ \phi_t^{X}\circ r_s(p)\;\text{d}s,\quad k''_{\theta}(t)=\frac{1}{\theta} \int_0^{\theta}X^{2}f\circ \phi_t^{X}\circ r_s(p)\;\text{d}s
		\end{equation*}
		for any $t\in \R$, as the geodesic flow $(\phi_t^{X})_{t\in \R}$ on $M$ is generated by the vector field $X$. Therefore, the assumption $\square f=\mu f$, i.e.~$-X^{2}f+Xf-UVf=\mu f$, translates readily into
		\begin{equation*}
			k''_{\theta}(t)-k'_{\theta}(t)+\mu k_{\theta}(t)=-\frac{1}{\theta}\int_0^{\theta}UVf\circ\phi_t^{X}\circ r_{s}(p)\;\text{d}s\;.
		\end{equation*}
		As $V=U-2\Theta$, we have 
		\begin{equation}
			\label{ode}
			k''_{\theta}(t)-k'_{\theta}(t)+\mu k_{\theta}(t)=	-\frac{1}{\theta}\int_0^{\theta}U^{2}f\circ\phi_t^{X}\circ r_{s}(p)\;\text{d}s+\frac{in }{\theta}\int_0^{\theta}Uf\circ\phi_t^{X}\circ r_{s}(p)\;\text{d}s
		\end{equation}
		by the assumption $\Theta f=\frac{i}{2}nf$.
	Now, by Stokes' theorem\footnote{Here, it really boils down to the fundamental theorem of calculus.}, we get that
	\begin{equation*}
		f\circ\phi_{t}^{X}\circ r_{\theta}(p)-f\circ \phi_t^{X}(p)=\int_{0}^{\theta}\frac{\text{d}}{\text{d}s}(f\circ \phi_t^{X}\circ r_{s}(p))\;\text{d}s=\int_0^{\theta}\text{d}f_{\phi_t^{X}\circ r_{s}(p)}\biggl(\frac{\text{d}}{\text{d}s}(\phi_t^{X}\circ r_{s}(p))\biggr)\;\text{d}s\;,
	\end{equation*}
	the latter equality following from the chain rule for differentials. Recalling that the flow $(r_s)_{s\in \R}$ is generated by the vector field $\Theta$, Lemma~\ref{lem:shiftedderivatives} delivers
	\begin{equation*}
		\frac{\text{d}}{\text{d}s}\phi_t^{X}\circ r_s(q)=\Ad_{\exp(-tX)}(\Theta)_{\phi_{t}^{X}\circ r_s(p)}\;,
	\end{equation*}
	so that
	\begin{equation*} 
		\begin{split}
			f\circ\phi_{t}^{X}\circ r_{\theta}(p)-f\circ \phi_t^{X}(p)&=
			\int_0^{\theta}\text{d}f_{\phi_t^{X}\circ r_{s}(p)}\bigl((\Ad_{\exp{(-tX)}}(\Theta))_{\phi_t^{X}\circ r_s(p)}\bigr)\;\text{d}s\\
			&=\int_0^{\theta}\text{d}f_{\phi_t^{X}\circ r_{s}(p)}\bigl(((-\sinh{t})U+e^{t}\Theta))_{\phi_t^{X}\circ r_s(p)}\bigr)\;\text{d}s\\
			&=-\sinh{t}\int_0^{\theta}Uf\circ\phi_t^{X}\circ r_{s}(p)\;\text{d}s+\frac{i}{2}n\theta e^{t}k_{\theta}(t)\;,
		\end{split}
	\end{equation*}
	the second equality being obtained by straightforward matrix multiplications.
	
	From now we choose $t$ strictly positive. We may thus write
	\begin{equation}
		\label{u}
		\int_0^{\theta}Uf\circ\phi_t^{X}\circ r_{s}(p)\;\text{d}s=\frac{1}{\sinh{t}}\biggl(\frac{i}{2}n\theta e^{t}k_{\theta}(t)-A_{\theta}(t)\biggr)
	\end{equation}
	where
	\begin{equation*} A_{\theta}(t)=f\circ\phi_{t}^{X}\circ r_{\theta}(p)-f\circ \phi_t^{X}(p)\;,\quad t>0.
	\end{equation*}
	Arguing as before, we also deduce
	\begin{equation}
		\label{dedusquare}
		\begin{split}
			Uf\circ\phi_{t}^{X}\circ r_{\theta}(p)-Uf\circ \phi_t^{X}(p)&=\int_{0}^{\theta}\frac{\text{d}}{\text{d}s}(Uf\circ \phi_t^{X}\circ r_{s}(p))\;\text{d}s\\
			&=\frac{1}{2}e^{-t}\int_0^{\theta}U^{2}f\circ \phi_t^{X}\circ r_s(p)\;\text{d}s-\frac{1}{2}e^{t}\int_0^{\theta}VUf\circ \phi_t^{X}\circ r_s(p)\;\text{d}s\;.
		\end{split}
	\end{equation}
	From $UV-VU=2X$ we get $VU=UV-2X=U(U-2\Theta)-2X=U^{2}-2U\Theta-2X$, so that 
	\begin{equation}
		\label{vu}
		\int_0^{\theta} VUf\circ \phi_t^{X}\circ r_s(p)\;\text{d}s=\int_0^{\theta} U^2f\circ \phi_t^{X}\circ r_s(p)\;\text{d}s-in\int_0^{\theta}Uf\circ \phi_t^{X}\circ r_s(p)\text{d}s-2\theta k'_{\theta}(t)\;.
	\end{equation}
	Combining~\eqref{u},~\eqref{dedusquare} and~\eqref{vu} yields
	\begin{equation}
		\label{usquare}
		\begin{split}
			\int_0^{\theta}U^{2}f\circ\phi_t^{X}\circ r_{s}(p)\;\text{d}s&=\frac{1}{\sinh{t}}\biggl(\frac{i}{2}ne^{t}\int_0^{\theta}Uf\circ \phi_t^{X}\circ r_s(p)\;\text{d}s+\theta e^{t}k'_{\theta}(t) -B_{\theta}(t)\biggr)\\
			&=\frac{1}{\sinh{t}}\biggl(\frac{ine^{t}}{4\sinh{t}}(in\theta e^{t} k_{\theta}(t)-A_{\theta}(t)) +\theta e^{t}k'_{\theta}(t) -B_{\theta}(t)\biggr)\;,
		\end{split}
	\end{equation}
	where
	\begin{equation*} B_{\theta}(t)=Uf\circ\phi_{t}^{X}\circ r_{\theta}(p)-Uf\circ \phi_t^{X}(p)\;,\quad t>0.
	\end{equation*}
	From~\eqref{ode},~\eqref{u} and~\eqref{usquare} we infer that 
	\begin{equation*}
		k_{\theta}''(t)-k'_{\theta}(t)+\mu k_{\theta}(t)=\frac{n^2 e^{2t} k_{\theta}(t)}{4\sinh^2{t}}+\frac{ine^{t}A_{\theta}(t)}{4\theta\sinh^2{t}} - \frac{e^{t}k'_{\theta}(t)}{\sinh{t}} -\frac{n^2 e^{t}k_{\theta}(t)}{2\sinh{t}}+\frac{-inA_{\theta}(t)+B_{\theta}(t)}{\theta \sinh{t}}
	\end{equation*}
	Finally, adding $2k_{\theta}'(t)$ on both sides gives
	\begin{equation*}
		\begin{split}
			k_{\theta}''(t)+k'_{\theta}(t)+\mu k_{\theta}(t)&=\biggl(\frac{n^{2}}{(1-e^{-2t})^{2}}-\frac{n^{2}}{1-e^{-2t}}\biggr)k_{\theta}(t)+\biggl(2-\frac{2}{1-e^{-2t}}\biggr)k'_{\theta}(t)\\
			&+\biggl(\frac{in}{2\theta \sinh{t}(1-e^{-2t})}-\frac{in}{\theta\sinh{t}}\biggr)A_{\theta}(t)+\frac{B_{\theta}(t)}{\theta \sinh{t}}\;,
		\end{split}
	\end{equation*}
	so that $k_{\theta}''(t)+k'_{\theta}(t)+\mu k_{\theta}(t)=e^{-t}G_{\theta,n}f(p,t)$ for 
	\begin{equation}
		\label{eq:constantterm}
		G_{\theta,n}f(p,t)=\frac{n^{2}e^{-t}}{(1-e^{-2t})^{2}}k_{\theta}(t)-\frac{2e^{-t}}{1-e^{-2t}}k'_{\theta}(t)+\frac{2ine^{-2t}}{\theta(1-e^{-2t})^{2}}A_{\theta}(t)+\frac{2}{\theta (1-e^{-2t})}B_{\theta}(t)\;.
	\end{equation}
	The function $f$ being of class $\cC^{2}$, it is clear that the functions $k_{\theta},k'_{\theta},A_{\theta}$ and $B_{\theta}$ are continuous, hence so is the function $t\mapsto G_{\theta,n}f(p,t)$.  Furthermore, the trivial upper bounds
	\begin{equation*}
		|k_{\theta}(t)|\leq \norm{f}_{\infty},\;|k'_{\theta}(t)|\leq \norm{Xf}_{\infty},\;|A_{\theta}(t)|\leq 2\norm{f}_{\infty},\; |B_{\theta}(t)|\leq 2\norm{Uf}_{\infty}
	\end{equation*}
	imply that it is uniformly bounded on $\R_{>0}$.
\end{proof}
For later purposes, we estimate explicitly the uniform norm of $G_{\theta,n}f$. Using the bounds $e^{-t}\leq 1$ and $1-e^{-2t}\geq 1-e^{-1}$, valid for all $t\geq 1/2$, together with the fact that the three quantities $\norm{f}_{\infty},\norm{Xf}_{\infty},\norm{Uf}_{\infty}$ are bounded from above by $\norm{f}_{\mathscr{C}^{1}}$ (cf.~Section~\ref{sec:hyperbolic}), we obtain that 
\begin{equation*}
	\sup_{t\geq 1/2}|G_{\theta,n}f(p,t)|\leq C_{\theta,n}\norm{f}_{\mathscr{C}^{1}}
\end{equation*}
with
\begin{equation*}
	C_{\theta,n}=\biggl(\frac{e}{e-1}\biggr)^{2}\frac{n(\theta n+2)}{\theta}+\frac{e}{e-1}\frac{2\theta+2}{\theta}\;.
\end{equation*}
Setting
\begin{equation*} \kappa_0=\frac{2e^{2}(1+4\pi)}{(e-1)^{2}}\;,
\end{equation*}
we may estimate
\begin{equation*}
	C_{\theta,n}\leq \biggl(\frac{e}{e-1}\biggr)^{2}\;\frac{2(\theta+1)}{\theta}(n^{2}+1)\leq \frac{\kappa_0}{\theta}(n^{2}+1)\;,
\end{equation*}
and hence conclude that 
\begin{equation}
	\label{eq:Gtheta}
	\sup_{t\geq 1/2}|G_{\theta,n}f(p,t)|\leq\frac{\kappa_0}{\theta}(n^{2}+1)\norm{f}_{\cC^{1}}
\end{equation}
for every choice of $p \in M$, $\theta\in (0,4\pi]$, $f\in \cC^{2}(M)$ and $n\in \Z$.

\section{Asymptotics for joint eigenfunctions}
\label{sec:jointeigenfn}

The purpose of this section is to prove Theorem~\ref{thm:case_Theta_eigenfn} by explictly solving the ODE established in Proposition~\ref{prop:ODE}. For definiteness, whe choose to impose initial conditions at time $t=1$ for the ensuing Cauchy problem. 

We thus start with the following:

\begin{lem}
	\label{lem:solution}
	Let $\mu$ be an eigenvalue of the Casimir operator. If $G\colon \R_{>0}\to \C$ is a continuous function, then for any complex numbers $y_1$ and $y_1'$ the solution to the Cauchy problem
	\begin{equation}
		\label{eq:Cauchy}
		\begin{cases}
			y''(t)+y'(t)+\mu y(t)=e^{-t}G(t) & \\
			y(1)=y_1& \\
			y'(1)=y'_1
		\end{cases}
	\end{equation}
	is given by
	\begin{equation}
		\label{eq:solutionnotquarter}
		\begin{split}
			y(t)=&e^{-\frac{1-\nu}{2}t}\biggl(\frac{(1+\nu)y_1+2y'_1}{2\nu e^{-\frac{1-\nu}{2}}} +\frac{1}{\nu}\int_1^{t}e^{-\frac{1+\nu}{2}\xi}G(\xi)\;\emph{d}\xi \biggr)\\
			&-e^{-\frac{1+\nu}{2}t}\biggl(\frac{(1-\nu)y_1+2y'_1}{2\nu e^{-\frac{1+\nu}{2}}} +\frac{1}{\nu}\int_1^{t}e^{-\frac{1-\nu}{2}\xi}G(\xi)\;\emph{d}\xi \biggr)
		\end{split}
	\end{equation}
	if $\mu\neq 1/4$, and by
	\begin{equation}
		\label{eq:solutionquarter}
		\begin{split}
			y(t)=&e^{-\frac{t}{2}}\biggl(\frac{\sqrt{e}(y_1-2y'_1)}{2}-\int_1^{t}\xi e^{-\frac{\xi}{2}}G(\xi)\;\emph{d}\xi\biggr)\\
			& +te^{-\frac{t}{2}}\biggl(\frac{\sqrt{e}(y_1+2y'_1)}{2}+\int_1^{t} e^{-\frac{\xi}{2}}G(\xi)\;\emph{d}\xi\biggr)
		\end{split}
	\end{equation}
	if $\mu=1/4$.
\end{lem}
\begin{proof}
	Let $\nu$ be the unique complex number in $\R_{\geq 0}\cup i \R_{>0}$ such that $1-\nu^{2}=4\mu$.
	The characteristic polynomial of the homogeneous equation $y''(t)+y'(t)+\mu y(t)=0$ is $P(Z)=Z^{2}+Z+\mu$, having two distinct roots $-\frac{1-\nu}{2},-\frac{1+\nu}{2}$ if $\mu\neq 1/4$, and a double root $-\frac{1}{2}$ if $\mu=1/4$. We examine the case $\mu\neq 1/4$; the case $\mu=1/4$ requires only minor modifications. A particular solution of the inhomogeneous equation is given by
	\begin{equation*}
		e^{-\frac{1-\nu}{2}t} +\int_1^{t}\frac{1}{\nu}e^{-\frac{1+\nu}{2}\xi}G(\xi)\;\text{d}\xi
		-e^{-\frac{1+\nu}{2}t}\int_1^{t}\frac{1}{\nu}e^{-\frac{1-\nu}{2}\xi}G(\xi)\;\text{d}\xi
	\end{equation*}
	as direct computations allow to verify. Hence, the general solution of
	\begin{equation*} y''(t)+y'(t)+\mu y(t)=e^{-t}G(t)
	\end{equation*}
	is given by
	\begin{equation*}
		y(t)=e^{-\frac{1-\nu}{2}t}\biggl(c_1 +\frac{1}{\nu}\int_1^{t}e^{-\frac{1+\nu}{2}\xi}G(\xi)\;\text{d}\xi \biggr)
		+e^{-\frac{1+\nu}{2}t}\biggl(c_2 -\frac{1}{\nu}\int_1^{t}e^{-\frac{1-\nu}{2}\xi}G(\xi)\;\text{d}\xi \biggr)\;, \quad c_1,c_2\in \C.
	\end{equation*}
	Imposing the conditions $y(1)=y_1,y'(1)=y'_1$ enables to determine the coefficients
	\begin{equation*}
		c_1=\frac{(1+\nu)y_1+2y'_1}{2\nu e^{-\frac{1-\nu}{2}}}\;,\quad c_2=-\frac{(1-\nu)y_1+2y'_1}{2\nu e^{-\frac{1+\nu}{2}}}\;.
	\end{equation*} 
\end{proof}

We may now apply Proposition~\ref{prop:ODE} in conjunction with Lemma~\ref{lem:solution} to determine the explicit analytic expression of the function $k_{f,\theta}(p,t)$ (cf.~\eqref{eq:ktheta}) in terms of the coefficient $G_{\theta,n}f(p,t)$ (cf.~\eqref{eq:constantterm}).
Before proceeding with this,  it will be convenient to set some useful notation first.

Let $\mu\in \text{Spec}(\square)$, $f\in \cC^{2}(M)$ and $\theta\in (0,4\pi]$. Define the functions $a_{\theta,\mu}^{+},a_{\theta,\mu}^{-}f\colon M\to \C$ by
\begin{equation}
	\label{eq:anquarter}
	a^{\pm}_{\theta,\mu}f(p)= \mp\frac{(1\mp\nu)\theta^{-1}\int_0^{\theta}f\circ \phi^{X}_1\circ r_s(p)\;\text{d}s+2\theta^{-1}\int_0^{\theta}Xf\circ \phi^{X}_1\circ r_s(p)\;\text{d}s}{2\nu e^{-\frac{1\pm\nu}{2}}}
\end{equation}
if $\mu\neq 1/4$ and
\begin{equation}
	\label{eq:aquarter}
	a^{\pm}_{\theta,1/4}f(p)= \frac{\sqrt{e}\bigl(\theta^{-1}\int_0^{\theta}f\circ \phi^{X}_1\circ r_s(p)\;\text{d}s\mp 2\theta^{-1}\int_0^{\theta}Xf\circ \phi^{X}_1\circ r_s(p)\;\text{d}s\bigr)}{2}
\end{equation}
if $\mu=1/4$. When $\mu\neq 1/4$, it holds 
\begin{equation}
	\label{eq:aestnquarter}
	\norm{a^{\pm}_{\theta,\mu}f}_{\infty}\leq \frac{(1+|\nu|)\theta^{-1}\int_0^{\theta}\norm{f}_{\infty}\text{d}s+2\theta^{-1}\int_0^{\theta}\norm{Xf}_{\infty}\text{d}s}{2e^{-1}|\nu|}\leq  \frac{e(3+|\nu|)}{2|\nu|}\norm{f}_{\mathscr{C}^{1}}\;,
\end{equation}
since $\norm{f}_{\infty}\leq \norm{f}_{\mathscr{C}^{1}}$ and $\norm{Xf}_{\infty}\leq \norm{f}_{\mathscr{C}^{1}}$.
If $\mu=1/4$, similar estimates lead readily to
\begin{equation}
	\label{eq:aestquarter}
	\norm{a^{\pm}_{\theta,1/4}f}_{\infty}\leq \frac{3\sqrt{e}}{2}\norm{f}_{\mathscr{C}^{1}}\;.
\end{equation}

We are now in a position to start the proof of Theorem~\ref{thm:case_Theta_eigenfn}, which will occupy the remainder of this section. We fix $\theta\in (0,4\pi]$ and a function $f\in \cC^{2}(M)$ satisfying $\Theta f=\mu f$ and $\Theta f=\frac{i}{2}nf$ for some $\mu \in \text{Spec}(\square)$ and $n\in \Z$. For any $p \in M$, the function $k_{f,\theta}(p,\cdot)\colon \R_{>0}\to \C$ we are interested in satisfies~\eqref{eq:Cauchy} with initial conditions
\begin{equation*}
	y_1=\frac{1}{\theta}\int_0^{\theta}f\circ \phi_1^{X}\circ r_s(p)\;\text{d}s\;,\quad y'_1=\frac{1}{\theta}\int_0^{\theta}Xf\circ \phi_1^{X}\circ r_s(p)\;\text{d}s\;.
\end{equation*} 

We distinguish five cases as in the statement of Theorem~\ref{thm:case_Theta_eigenfn}, that is, according to the value of the Casimir eigenvalue $\mu$. Recall that $\nu$ is the unique complex number in $\R_{\geq 0}\cup i \R_{>0}$ verifying $1-\nu^{2}=4\mu$.

\subsection{The case $\mu>1/4$}
\label{sec:caseabovequarter}
Suppose $\mu>1/4$, so that $\nu=i\Im{\nu}\in i\R_{>0}$. As follows from~\eqref{eq:solutionnotquarter}, the solution to~\eqref{eq:ode} with the prescribed initial conditions is given by
\begin{equation*}
	\begin{split}
		k_{f,\theta}(p,t)&=e^{-\frac{t}{2}}\cos{\biggl(\frac{\Im{\nu}}{2}t\biggr)}\biggl(a^{+}_{\theta,\mu}f(p)+a^{-}_{\theta,\mu}f(p)-\frac{2}{\Im{\nu}}\int_1^{t}e^{-\frac{\xi}{2}}\sin{\biggl(\frac{\Im{\nu}}{2}\xi\biggr)}G_{\theta,n}f(p,\xi)\;\text{d}\xi\biggr)\\
		&+e^{-\frac{t}{2}}\sin{\biggl(\frac{\Im{\nu}}{2}t\biggr)}\biggl(a^{-}_{\theta,\mu}f(p)-a^{+}_{\theta,\mu}f(p)-\frac{2i}{\Im{\nu}}\int_1^{t}e^{-\frac{\xi}{2}}\cos{\biggl(\frac{\Im{\nu}}{2}\xi\biggr)}G_{\theta,n}f(p,\xi)\;\text{d}\xi\biggr)\;.
	\end{split}
\end{equation*}
The functions
\begin{equation*}
	e^{-\xi/2}\cos{(\frac{\Im{\nu}}{2}\xi)}G_{\theta,n}f(p,\xi),\quad e^{-\xi/2}\sin{(\frac{\Im{\nu}}{2}\xi)}G_{\theta,n}f(p,\xi) 
\end{equation*} 
are integrable over the closed half-line $[1,+\infty)$, as $G_{\theta,n}f(p,\cdot)$ is bounded thereon; we may therefore define functions $D^{+}_{\theta,\mu,n}f,D^{-}_{\theta,\mu,n}f\colon M\to \C$ by setting
\begin{equation}
	\label{eq:Dplusabovequarter}
	D^{+}_{\theta,\mu,n}f(p)=a^{+}_{\theta,\mu}f(p)+a^{-}_{\theta,\mu}f(p)-\frac{2}{\Im{\nu}}\int_1^{\infty}e^{-\frac{\xi}{2}}\sin{\biggl(\frac{\Im{\nu}}{2}\xi\biggr)}G_{\theta,n}f(p,\xi)\;\text{d}\xi
\end{equation}
and 
\begin{equation}
	\label{eq:Dminusabovequarter}
	D^{-}_{\theta,\mu,n}f(p)=a^{-}_{\theta,\mu}f(p)-a^{+}_{\theta,\mu}f(p)-\frac{2i}{\Im{\nu}}\int_1^{\infty}e^{-\frac{\xi}{2}}\cos{\biggl(\frac{\Im{\nu}}{2}\xi\biggr)}G_{\theta,n}f(p,\xi)\;\text{d}\xi
\end{equation}
for every $p \in M$. Then~\eqref{eq:estabovequarter} is valid with 
\begin{equation}
	\label{eq:remainderabovequarter}
	\begin{split}
		\mathcal{R}_{\theta,\mu,n}f(p,t)&=e^{-\frac{t}{2}}\cos{\biggl(\frac{\Im{\nu}}{2}t\biggr)}\int_t^{\infty}\frac{2}{\Im{\nu}}e^{-\frac{\xi}{2}}\sin{\biggl(\frac{\Im{\nu}}{2}\xi\biggr)}G_{\theta,n}f(p,\xi)\;\text{d}\xi\\
		&+e^{-\frac{t}{2}}\sin{\biggl(\frac{\Im{\nu}}{2}t\biggr)}\int_t^{\infty}\frac{2i}{\Im{\nu}}e^{-\frac{\xi}{2}}\cos{\biggl(\frac{\Im{\nu}}{2}\xi\biggr)}G_{\theta,n}f(p,\xi)\;\text{d}\xi\;,\quad t\geq 1, \;p \in M.
	\end{split}
\end{equation}
Let us now estimate the uniform norms of $D^{\pm}_{\theta,\mu,n}f$ and of $\mathcal{R}_{\theta,\mu,n}f(\cdot,t)$ for any $t\geq 1$. From the explicit expressions in~\eqref{eq:Dplusabovequarter} and~\eqref{eq:Dminusabovequarter}, it follows at once that
\begin{equation}
	\label{eq:kappamu}
	\begin{split}
		\norm{D^{\pm}_{\theta,\mu,n}f}_{\infty}&\leq \norm{a^{+}_{\theta,\mu}f}_{\infty}+\norm{a^{-}_{\theta,\mu}}_{\infty}+\frac{2}{\Im{\nu}}\int_1^{\infty}e^{-\frac{\xi}{2}}\sup_{p\in M,\;\xi\geq 1}|G_{\theta,n}f(p,\xi)|\;\text{d}\xi\\
		&\leq \frac{1}{\Im{\nu}}\biggl(e(3+\Im{\nu})+\frac{4\kappa_0(n^{2}+1)}{\theta\sqrt{e}}\biggr)\norm{f}_{\mathscr{C}^{1}}\\
		&\leq\frac{\kappa(\mu)}{\theta}(n^{2}+1)\norm{f}_{\mathscr{C}^{1}}
	\end{split}
\end{equation}
applying~\eqref{eq:Gtheta} and~\eqref{eq:aestnquarter} in the second inequality, with $\kappa(\mu)=\frac{1}{\Im{\nu}}\bigl(4e\pi(3+\Im{\nu})+\frac{4\kappa_0}{\sqrt{e}}\bigr)$.

The remainder term defined in~\eqref{eq:remainderabovequarter} is bounded from above by
\begin{equation*}
	|\mathcal{R}_{\theta,\mu,n}f(p,t)|\leq 2e^{-\frac{t}{2}}\sup_{p \in M,\;\xi\geq 1}|G_{\theta,n}f(p,\xi)|\int_t^{\infty}\frac{2}{\Im{\nu}}e^{-\frac{\xi}{2}}\;\text{d}\xi\leq\frac{8\kappa_0 (n^{2}+1)}{\theta\;\Im{\nu}}\norm{f}_{\mathscr{C}^{1}}e^{-t}
\end{equation*}
for every $p\in M$ and $t\geq 1$, again relying on the upper bound in~\eqref{eq:Gtheta}. 

\smallskip
Up to the regularity claims on the coefficients $D^{\pm}_{\theta,\mu,n}f$, which are the subject of Section~\ref{sec:regularity}, the proof of Theorem~\ref{thm:case_Theta_eigenfn}(1) is complete.  

\subsection{The case $\mu=1/4$}

Suppose $\mu=1/4$, whence $\nu=0$. This time the solution to~\eqref{eq:ode} with the given initial conditions has the expression (cf.~\eqref{eq:solutionquarter}) 
\begin{equation*}
	k_{f,\theta}(p,t)=e^{-\frac{t}{2}}\biggl(a^{+}_{\theta,1/4}f(p)-\int_1^{t}\xi e^{-\frac{\xi}{2}}G_{\theta,n}f(p,\xi)\;\text{d}\xi\biggr)
	+te^{-\frac{t}{2}}\biggl(a^{-}_{\theta,1/4}f(p)+\int_1^{t} e^{-\frac{\xi}{2}}G_{\theta,n}f(p,\xi)\;\text{d}\xi\biggr)\;.
\end{equation*}
Following the steps carried out in Section~\ref{sec:caseabovequarter} almost verbatim, define functions $D^{+}_{\theta,1/4,n}, D^{-}_{\theta,1/4,n}\colon M\to \C$ via
\begin{equation}
	\label{eq:Dquarter}
	D^{+}_{\theta,1/4,n}f(p)=a^{+}_{\theta,1/4}f(p)-\int_1^{\infty}\xi e^{-\frac{\xi}{2}}G_{\theta,n}f(p,\xi)\;\text{d}\xi,\quad D^{-}_{\theta,1/4,n}f(p)=a^{-}_{\theta.1/4}f(p)+\int_1^{\infty} e^{-\frac{\xi}{2}}G_{\theta,n}f(p,\xi)\;\text{d}\xi
\end{equation}
for every $p \in M$. Then~\eqref{eq:estquarter} holds with 
\begin{equation}
	\label{eq:remainderquarter}
	\mathcal{R}_{\theta.1/4,n}f(p,t)=e^{-\frac{t}{2}}\int_t^{\infty}\xi e^{-\frac{\xi}{2}}G_{\theta,n}f(p,\xi)\;\text{d}\xi -te^{-\frac{t}{2}}\int_t^{\infty}e^{-\frac{\xi}{2}}G_{\theta,n}f(p,\xi)\;\text{d}\xi\;, \quad t\geq 1,\;p\in M.
\end{equation}
From~\eqref{eq:Dquarter} we estimate, by virtue of~\eqref{eq:Gtheta} and~\eqref{eq:aestquarter}, 
\begin{equation*}
	\begin{split}
		\norm{D^{\pm}_{\theta,1/4,n}f}_{\infty}&\leq \frac{3\sqrt{e}}{2}\norm{f}_{\mathscr{C}^{1}}+\frac{\kappa_0}{\theta}(n^{2}+1)\norm{f}_{\mathscr{C}^{1}}\int_1^{\infty}\xi e^{-\frac{\xi}{2}}\;\text{d}\xi\\
		&=\biggl(\frac{3\sqrt{e}}{2}+\frac{6\kappa_0}{\theta\sqrt{e}}(n^{2}+1)\biggr)\norm{f}_{\mathscr{C}^{1}}\leq \frac{\kappa(1/4)}{\theta}(n^{2}+1)\norm{f}_{\mathscr{C}^{1}}\;,
	\end{split}
\end{equation*}
where we may choose $\kappa(1/4)=36\pi \sqrt{e}\kappa_0$. Moreover, we deduce from~\eqref{eq:remainderquarter} that, for every $p \in M$ and $t\geq 1$,
\begin{equation*}
	\begin{split}
		|\mathcal{R}_{\theta,1/4,n}f(p,t)|&\leq \frac{\kappa_0}{\theta}(n^{2}+1)\norm{f}_{\mathscr{C}^{1}}\biggl(e^{-\frac{t}{2}}\int_t^{\infty}\xi e^{-\frac{\xi}{2}}\;\text{d}\xi+te^{-\frac{t}{2}}\int_t^{\infty}e^{-\frac{\xi}{2}}\;\text{d}\xi\biggr) \\ 
		&=\frac{4\kappa_0}{\theta}(n^{2}+1)\norm{f}_{\mathscr{C}^{1}}(t+1)e^{-t}\;.
	\end{split}
\end{equation*}
In conjunction with Section~\ref{sec:regularity}, this concludes the proof of Theorem~\ref{thm:case_Theta_eigenfn}(2).

\subsection{The case $0<\mu<1/4$} When $0<\mu<1/4$, we have $\nu\in (0,1)$. The solution to~\eqref{eq:ode} is given, as in~\eqref{eq:solutionnotquarter}, by
\begin{equation*}
	\begin{split}
		k_{f,\theta}(p,t)=&e^{-\frac{1+\nu}{2}t}\biggl(a^{+}_{\theta,\mu}f(p)-\frac{1}{\nu}\int_1^{t}e^{-\frac{1-\nu}{2}\xi}G_{\theta,n}f(p,\xi)\;\text{d}\xi\biggr)\\
		&+e^{-\frac{1-\nu}{2}t}\biggl(a^{-}_{\theta,\mu}f(p)+\frac{1}{\nu}\int_1^{t}e^{-\frac{1+\nu}{2}\xi}G_{\theta,n}f(p,\xi)\;\text{d}\xi\biggr)\;.
	\end{split}
\end{equation*}
Setting
\begin{equation}
	\label{eq:Dbelowquarter}
	\begin{split}
		&D^{+}_{\theta,\mu,n}f(p)=a^{+}_{\theta,\mu}f(p)-\frac{1}{\nu}\int_1^{\infty}e^{-\frac{1-\nu}{2}\xi}G_{\theta,n}f(p,\xi)\;\text{d}\xi\;, \\ &D^{-}_{\theta,\mu,n}f(p)=a^{-}_{\theta,\mu}f(p)+\frac{1}{\nu}\int_1^{\infty}e^{-\frac{1-\nu}{2}\xi}G_{\theta,n}f(p,\xi)\;\text{d}\xi
	\end{split}
\end{equation}
and 
\begin{equation*}
	\mathcal{R}_{\theta,\mu,n}f(p,t)=\frac{1}{\nu}\biggl(e^{-\frac{1+\nu}{2}t}\int_t^{\infty}e^{-\frac{1-\nu}{2}\xi}G_{\theta,n}f(p,\xi)\;\text{d}\xi-e^{-\frac{1-\nu}{2}t}\int_t^{\infty}e^{-\frac{1+\nu}{2}\xi}G_{\theta,n}f(p,\xi)\;\text{d}\xi\biggr)
\end{equation*}
for every $p \in M$ and $t\geq 1$, it is clear that~\eqref{eq:estpositivebelowquarter} holds. As far as estimates on the supremum norm are concerned, we have
\begin{equation*}
	\begin{split}
		\norm{D^{\pm}_{\theta,\mu,n}f}_{\infty}&\leq \frac{e(3+\nu)}{2\nu}\norm{f}_{\mathscr{C}^{1}}+\frac{\kappa_0}{\theta\nu}(n^{2}+1)\norm{f}_{\mathscr{C}^{1}}\int_1^{\infty}e^{-\frac{1\mp \nu}{2}\xi}\;\text{d}\xi\\
		&=\frac{1}{\nu}\biggl(\frac{e(3+\nu)}{2}+\frac{\kappa_0}{\theta}\frac{2}{1\mp \nu}e^{-\frac{1\mp \nu}{2}}(n^{2}+1)\biggr)\norm{f}_{\mathscr{C}^{1}}\\
		&\leq \frac{\kappa(\mu)}{\theta}(n^{2}+1)\norm{f}_{\mathscr{C}^{1}}\;,
	\end{split}
\end{equation*}
for $\kappa(\mu)=\frac{2\kappa_0e^{-\frac{1-\nu}{2}}}{\nu(1-\nu)}+\frac{2e\pi(3+\nu)}{\nu}$, and
\begin{equation*}
	\begin{split}
		|\mathcal{R}_{\theta,\mu,n}f(p,t)|&\leq \frac{\kappa_0}{\theta\nu}(n^{2}+1)\norm{f}_{\mathscr{C}^{1}}\biggl(e^{-\frac{1+\nu}{2}t}\int_t^{\infty}e^{-\frac{1-\nu}{2}\xi}\;\text{d}\xi+e^{-\frac{1-\nu}{2}t}\int_t^{\infty}e^{-\frac{1+\nu}{2}\xi}\;\text{d}\xi\biggr)\\
		&= \frac{4\kappa_0}{\theta \nu(1-\nu)(1+\nu)}(n^{2}+1)\norm{f}_{\mathscr{C}^{1}}e^{-t}\;,
	\end{split}
\end{equation*}
which achieves the proof of Theorem~\ref{thm:case_Theta_eigenfn}(3) except for the regularity of the coefficients which is addressed separately in Section~\ref{sec:regularity}.

\subsection{The case $\mu= 0$} As $\nu=1$ when $\mu=0$,  equation~\eqref{eq:solutionnotquarter} delivers the following expression for the solution to~\eqref{eq:ode}:
\begin{equation}
	\label{eq:kzero}
	\begin{split}
		k_{f,\theta}(p,t)&=a^{-}_{\theta,0}f(p)+\int_1^{t}e^{-\xi}G_{\theta,n}f(p,\xi)\;\text{d}\xi+e^{-t}\biggl(a^{+}_{\theta,0}f(p)-\int_1^{t}G_{\theta,n}f(p,\xi)\;\text{d}\xi\biggr)\\
		&= a^{-}_{\theta,0}f(p)+\int_1^{\infty}e^{-\xi}G_{\theta,n}f(p,\xi)\;\text{d}\xi+\biggl(-\int_t^{\infty}e^{-\xi}G_{\theta,n}f(p,\xi)\;\text{d}\xi\\
		& +e^{-t}a^{+}_{\theta,0}f(p)-e^{-t}\int_1^{t}G_{\theta,n}f(p,\xi)\;\text{d}\xi\biggr)\;.
	\end{split}
\end{equation}
Observe that the term between parentheses in the last expression is infinitesimal as $t$ tends to infinity, so that $k_{f,\theta}(p,t)$ has a well-defined limit, as $t$ tends to infinity, for every $p \in M$. 
We claim that\footnote{The claim amounts to the qualitative equidistribution statement that circle-arc averages of $f$ converge to its spatial average with respect to the uniform measure. This has been shown by Margulis (for complete circles) in~\cite{Margulis} via a thickening argument resting on the mixing properties of the geodesic flow. We prefer not to invoke Margulis' result here, and instead prove directly this special case of equidistribution using spectral considerations coupled with mixing.  }    
\begin{equation}
	\label{eq:constanttermzero}
	\lim_{t\to\infty}k_{f,\theta}(p,t)=\int_{M}f\;\text{d}\vol
\end{equation}
for any $p \in M$. For a start, we show the equality in~\eqref{eq:constanttermzero} holds on average with respect to the measure $\vol$. Indeed, Fubini's theorem gives, for any $t>0$,
\begin{equation*}
	\begin{split}
		\int_{M}k_{f,\theta}(p,t)\;\text{d}\vol(p)&=\frac{1}{\theta}\int_{M}\int_0^{\theta}f\circ \phi_t^{X}\circ r_s(p)\;\text{d}s\;\text{d}\vol(p)=\frac{1}{\theta}\int_{0}^{\theta}\int_{M}f\circ \phi_t^{X}\circ r_s(p)\;\text{d}\vol(p)\;\text{d}s\\
		&=\frac{1}{\theta}\int_0^{\theta}\int_{M}f\;\text{d}\vol\;\text{d}s=\int_{M}f\;\text{d}\vol\;,
	\end{split}
\end{equation*}
the second-to-last equality following from the fact that the transformation $\phi_t^{X}\circ r_s\colon M\to M$ preserves the measure $\vol$ for any $s,t\in \R$.  By dominated convergence,
\begin{equation}
	\label{eq:sameaverage}
	\int_{M}\lim_{t\to\infty}k_{f,\theta}(p,t)\;\text{d}\vol(p)=\lim_{t\to\infty}\int_{M}k_{f,\theta}(p,t)\;\text{d}\vol(p)=\lim_{t\to\infty}\int_{M}f\;\text{d}\vol=\int_{M}f\;\text{d}\vol\;.
\end{equation}

In order to finish the proof of the claim, it remains to show that the limit $\lim_{t\to\infty}k_{f,\theta}(p,t)$ does not depend on $p$. Choose a countable orthonormal basis $(u_{k})_{k\in I}$ of $L^{2}(M)$ consisting of smooth eigenfunctions of the operator $\Theta$ and containing a constant function $u_{k_0}$. If $\Theta u_k=\frac{i}{2}n_ku_k$ for $n_k\in \Z$, then $u_k\circ r_{s}=e^{2\pi i n_k s}u_k$ for every $s\in \R$ (cf.~Section~\ref{sec:unitaryrepresentations}). Therefore, we can compute for every $k \in I$ the $L^{2}$-inner product
\begin{equation*}
	\begin{split}
		\int_{M}\lim_{t\to\infty}k_{f,\theta}(p,t)\;\overline{u_k(p)}\;\text{d}\vol(p)&=\lim_{t\to\infty}\int_{M}\biggl(\frac{1}{\theta}\int_0^{\theta}f\circ \phi_t^{X}\circ r_s(p)\;\text{d}s\biggr) \overline{u_k(p)}\;\text{d}p\\
		&=\lim_{t\to\infty}\frac{1}{\theta}\int_{0}^{\theta}\int_{M}f\circ \phi_t^{X}\circ r_s(p)\;\overline{u_k(p)}\;\text{d}p\;\text{d}s\\
		&=\lim_{t\to\infty}\frac{1}{\theta}\int_0^{\theta}\int_{M}f\circ\phi_t^{X}(p)\;\overline{u_{k}}\circ r_{-s}(p)\;\text{d}p \;\text{d}s\\
		&=\lim_{t\to\infty}\frac{1}{\theta}\int_0^{\theta}e^{-2\pi i n_k s}\int_{M}f\circ\phi_t^{X}(p)\;\overline{u_{k}(p)}\;\text{d}p \;\text{d}s\\
		&=\frac{1}{\theta}\int_0^{\theta}e^{-2\pi i n_ks}\text{d}s\cdot \lim_{t\to\infty}\langle f\circ \phi_t^{X},u_k\rangle\;,
	\end{split}
\end{equation*}
where we used, in successive order, the dominated convergence theorem, Fubini's theorem, invariance of the measure $\vol$ under the transformation $r_{-s}$ and the definining property of $u_k$. Mixing of the geodesic flow $(\phi_t^{X})_{t\in \R}$ on $M$ (cf.~\cite[Cor.~2.3]{Bekka-Mayer}) delivers
\begin{equation*}
	\lim_{t\to\infty}\langle f\circ \phi_t^{X},u_k\rangle =\int_{M}f\;\text{d}\vol\int_{M}u_k\;\text{d}\vol
	\;.
\end{equation*}

As $u_k$ is orthogonal in $L^{2}(M)$ to the constant $u_{k_0}$ for any $k\neq k_0$, the last expression vanishes for any such $k$. Therefore, we have just shown that the function $p\mapsto \lim_{t\to\infty}k_{f,\theta}(p,t)$ is orthogonal to $u_k$ for every $k\neq k_0$; necessarily, it must be constant. 

\smallskip 

Define now
\begin{equation*}
	\mathcal{R}_{\theta,0,n}f(p,t)=a_{\theta,0}^{+}f(p)e^{-t}-\int_t^{\infty}e^{-\xi}G_{\theta,n}f(p,\xi)\;\text{d}\xi\;;
\end{equation*}
then~\eqref{eq:Gtheta} and~\eqref{eq:aestnquarter} give
\begin{equation*}
	|R_{\theta,0,n}f(p,t)|\leq 2e \norm{f}_{\mathscr{C}^{1}}e^{-t}+\frac{\kappa_0}{\theta}(n^{2}+1)\norm{f}_{\mathscr{C}^{1}}\int_t^{\infty}e^{-\xi}\;\text{d}\xi=\frac{1}{\theta}\bigl(8e\pi+\kappa_0(n^{2}+1)\bigr)\norm{f}_{\mathscr{C}^{1}}e^{-t}
\end{equation*}
and, combining~\eqref{eq:kzero} with~\eqref{eq:constanttermzero}, we obtain
\begin{equation*}
	k_{f,\theta}(p,t)=\int_{M}f\;\text{d}\vol+e^{-t}\int_1^{t}-G_{\theta,n}f(p,\xi)\;\text{d}\xi+\mathcal{R}_{\theta,0,n}f(p,t)\;,
\end{equation*} 
which establishes Theorem~\ref{thm:case_Theta_eigenfn}(4).

\subsection{The case $\mu<0$}
We now turn to the case $\mu<0$ or, equivalently, $\nu>1$. The solution to~\eqref{eq:ode} given in~\eqref{eq:solutionnotquarter} becomes
\begin{equation}
	\label{eq:kbelowzero}
	\begin{split}
		k_{f,\theta}(p,t)=& e^{\frac{\nu-1}{2}t}\biggl(a^{-}_{\theta,\mu}f(p)+\frac{1}{\nu}\int_1^{t}e^{-\frac{\nu+1}{2}\xi}G_{\theta,n}f(p,\xi)\;\text{d}\xi\biggr) \\
		&+e^{-\frac{\nu+1}{2}t}\biggl(a^{+}_{\theta,\mu}f(p)-\frac{1}{\nu}\int_1^{t}e^{-\frac{\nu-1}{2}\xi}G_{\theta,n}f(p,\xi)\;\text{d}\xi\biggr)\;.
	\end{split}
\end{equation}

It follows that the quantity
\begin{equation*}
	\begin{split}
		e^{\frac{\nu-1}{2}t}\biggl(a^{-}_{\theta,\mu}f(p)+\frac{1}{\nu}\int_1^{\infty}e^{-\frac{\nu+1}{2}\xi}G_{\theta,n}f(p,\xi)\;\text{d}\xi\biggr)=&-e^{-\frac{\nu+1}{2}t}\biggl(a^{+}_{\theta,\mu}f(p)-\frac{1}{\nu}\int_1^{t}e^{-\frac{\nu-1}{2}\xi}G_{\theta,n}f(p,\xi)\;\text{d}\xi\biggr)\\
		&-k_{f,\theta}(p,t)
		+\frac{e^{\frac{\nu-1}{2}t}}{\nu}\int_{t}^{\infty}e^{-\frac{\nu+1}{2}\xi}G_{\theta,n}f(p,\xi)\;\text{d}\xi
	\end{split}
\end{equation*}
is uniformly bounded in $t$, which forces 
\begin{equation*}
	a^{-}_{\theta,\mu}f(p)+\frac{1}{\nu}\int_1^{\infty}e^{-\frac{\nu+1}{2}\xi}G_{\theta,n}f(p,\xi)\;\text{d}\xi=0 
\end{equation*}
for every $p \in M$. Therefore~\eqref{eq:kbelowzero} results in
\begin{equation*}
	k_{f,\theta}(p,t)=e^{-t}\biggl(-\frac{e^{\frac{\nu+1}{2}t}}{\nu}\int_t^{\infty}e^{-\frac{\nu+1}{2}\xi}G_{\theta,n}f(p,\xi)\;\text{d}\xi +e^{-\frac{\nu-1}{2}t}a^+_{\theta,\mu}f(p)-\frac{e^{-\frac{\nu-1}{2}t}}{\nu}\int_1^{t}e^{\frac{\nu-1}{2}\xi}G_{\theta,n}f(p,\xi)\;\text{d}\xi \biggr)\;.
\end{equation*}

With the help of~\eqref{eq:Gtheta} and~\eqref{eq:aestnquarter}, we may estimate the three summands inside the parentheses. For the first, we have
\begin{equation*}
	\biggl|\frac{e^{\frac{\nu+1}{2}t}}{\nu}\int_t^{\infty}e^{-\frac{\nu+1}{2}\xi}G_{\theta,n}f(p,\xi)\;\text{d}\xi\biggr|\leq \frac{2\kappa_0}{\theta\nu(\nu+1)}(n^{2}+1)\norm{f}_{\mathscr{C}^{1}}\;,
\end{equation*}
while the second can be bounded as
\begin{equation*}
	\bigl|e^{-\frac{\nu-1}{2}t}a^{+}_{\theta,\mu}f(p)\bigr|\leq \frac{e(3+\nu)}{2\nu}\norm{f}_{\mathscr{C}^{1}}\;;
\end{equation*}
lastly,
\begin{equation*}
	\begin{split}
		\biggl|\frac{e^{-\frac{\nu-1}{2}t}}{\nu}\int_1^{t}e^{\frac{\nu-1}{2}\xi}G_{\theta,n}f(p,\xi)\;\text{d}\xi\biggr| &\leq \frac{2\kappa_0}{\theta\nu(\nu-1)}(n^{2}+1)\norm{f}_{\mathscr{C}^{1}}e^{-\frac{\nu-1}{2}t}(e^{\frac{\nu-1}{2}t}-e^{\frac{\nu-1}{2}}) \\
		& \leq \frac{2\kappa_0}{\theta\nu(\nu-1)}(n^{2}+1)\norm{f}_{\mathscr{C}^{1}}\;.
	\end{split}
\end{equation*}

We conclude that 
\begin{equation*}
	|k_{f,\theta}(p,t)|\leq\frac{1}{\theta} \biggl(\frac{4\kappa_0}{(\nu-1)(\nu+1)}+\frac{2 e\pi(3+\nu)}{\nu}\biggr)(n^{2}+1) \norm{f}_{\mathscr{C}^{1}}e^{-t}\;,
\end{equation*}
as desired. 

\smallskip
This settles Theorem~\ref{thm:case_Theta_eigenfn}(5).





\subsection{Regularity of the coefficients in the asymptotic expansion} 
\label{sec:regularity}
We now turn to the examination of the regularity properties of the coefficients $D^{\pm}_{\theta,\mu,n}f$ appearing in the asymptotic expansion of $k_{f,\theta}(p,t)$, as in Theorem~\ref{thm:case_Theta_eigenfn}, for $f\in \mathscr{C}^{2}(M)$ satisfying $\square f=\mu f $ and $\Theta f=\frac{i}{2}nf$ for some $\mu\in \text{Spec}(\square)\cap\R_{>0}$ and $n\in \Z$. In so doing, we shall complete the proof of Theorem~\ref{thm:case_Theta_eigenfn}.

Let us fix $\theta\in (0,4\pi]$ throughout this subsection. As follows readily from the definitions of the coefficients \eqref{eq:anquarter}, \eqref{eq:aquarter}, \eqref{eq:Dplusabovequarter}, \eqref{eq:Dminusabovequarter}, \eqref{eq:Dquarter} and \eqref{eq:Dbelowquarter}, it suffices to analyze the regularity of 
\begin{equation*}
	k_{f,\theta}(p,1)=\int_0^{\theta}f\circ \phi^{X}_1\circ r_s(p)\;\text{d}s, \; k'_{f,\theta}(p,1)=\int_0^{\theta}Xf\circ \phi^{X}_1\circ r_s(p)\;\text{d}s, \;\int_1^{\infty}g(\xi)G_{\theta,n}f(p,\xi)\;\text{d}\xi
\end{equation*}
as functions of $p\in M$, with $g(\xi)$ being a function of the following forms:
\begin{equation}
	\label{eq:g}
	e^{-\xi/2},\;\xi e^{-\xi/2}, \;e^{-\xi/2}\cos{\biggl(\frac{\Im{\nu}}{2}\xi\biggr)},\;e^{-\xi/2}\sin{\biggl(\frac{\Im{\nu}}{2}\xi\biggr)},\; e^{-\frac{1\pm\nu}{2}\xi}\;(0<\nu<1).
\end{equation}

We start with the following elementary lemma:

\begin{lem}
	\label{lem:diffunderint}
	Let $\theta$ be a positive real number. If $h\colon M\to \C$ is of class $\mathscr{C}^{1}$, then the function
	\begin{equation*}
		p\mapsto \int_0^{\theta}h\circ r_s(p)\;\emph{d}s\;, \quad p\in M
	\end{equation*}
	is of class $\mathscr{C}^{1}$ on $M$.
\end{lem}
\begin{proof}
	Fix a point $p_0\in M$, and let $(\partial_{x_i})_{i=1,2,3}$ be a local frame of the tangent bundle $TM$ around $p_0$. It suffices to prove that, for each $i=1,2,3$, the partial derivative 
	\begin{equation}
		\label{eq:partialderivative}
		p\mapsto\partial_{x_i}|_p\biggl(\int_0^{\theta}h\circ r_s\;\text{d}s\biggr)
	\end{equation}
	exists and is continuous in an open neighborhood of $p$. Upon passing to local smooth charts for $M$, the classical theorem of differentiation under the integral sign ensures the validity of the formal passage
	\begin{equation*}
		p\mapsto\partial_{x_i}|_p\biggl(\int_0^{\theta}h\circ r_s\;\text{d}s\biggr)=\int_0^{\theta}\partial_{x_i}|_p(h\circ r_s)\;\text{d}s
	\end{equation*}
	provided that there exists a positive real-valued function $\varphi$ on $[0,\theta]$, integrable with respect to the Lebesgue measure, such that $|\partial_{x_i}|_q(h\circ r_s)|\leq \varphi(s)$ for any $q$ in an open neighborhood of $p$. Notice that, by the dominated convergence theorem. this would yield continuity of the partial derivative in~\eqref{eq:partialderivative} at the same time. The chain rule for the differential gives $\mathrm{d}(h\circ r_s)_q=(\mathrm{d}h)_{r_s(q)}\circ (\mathrm{d}r_s)_q$ for any $q \in M$ and $s\in [0,\theta]$. As follows readily from the explicit expression for $r_s$ in~\eqref{eq:rotationflow} and direct computations, the operator norm\footnote{Formally, we would need to specify a Riemannian metric on the compact manifold $M$. For the purposes of the proof however, only boundedness of the relevant quantities matters, so that any such metric would serve our goal (cf.~Remark~\ref{rmk:equivalentmetrics}(1)).} of the linear operator $(\mathrm{d}r_s)_q$ is uniformly bounded in $q$ and $s$, as the entries of any Jacobian matrix associated to it only involve finite linear combinations of the sine and cosine functions. Therefore, there exists a constant $C>0$ such that $|\partial_{x_i}|_q(h\circ r_s)|\leq C\norm{h}_{\mathscr{C}^{1}}$ for any $s\in [0,\theta]$ and any $q$ in the domain of definition of the local frame $(\partial_{x_i})_{i=1,2,3}$. The conclusion is thus achieved by setting $\varphi$ to be constantly equal to $C\norm{h}_{\mathscr{C}^{1}}$. 
\end{proof}

As $(\phi_t^{X})_{t\in \R}$ is a smooth flow on $M$ and $f$ is of class $\mathscr{C}^{2}$, the functions $f\circ \phi_1^{X},\;Xf\circ \phi_1^{X}$ are of class $\cC^{1}$. By virtue of Lemma~\ref{lem:diffunderint}, the functions $p\mapsto \int_0^{\theta}f\circ \phi_1^{X}\circ r_s(p)\;\text{d}s,\;p\mapsto \int_0^{\theta}Xf\circ \phi_1^{X}\circ r_s(p)\;\text{d}s$ are of class $\cC^{1}$ on $M$. 

Therefore, it remains to deal with $\int_1^{\infty}g(\xi)G_{\theta,n}f(p,\xi)\;\text{d}\xi$ as a function of $p \in M$, $g$ being as in~\eqref{eq:g}. Expanding out the expression from~\eqref{eq:constantterm}, we obtain that it equals
\begin{equation*}
	\begin{split} \int_1^{\infty}&g(\xi)\biggl(\frac{n^{2}e^{-\xi}}{\theta(1-e^{-2\xi})^{2}}\int_0^{\theta}f\circ \phi^{X}_\xi\circ r_s(p)\text{d}s -\frac{2e^{-\xi}}{\theta(1-e^{-2\xi})}\int_0^{\theta}Xf\circ \phi^{X}_{\xi}\circ r_s(p)\text{d}s\\
		&+\frac{2ine^{-2\xi}}{\theta(1-e^{-2\xi})^{2}}(f\circ \phi^{X}_{\xi}\circ r_{\theta}(p)-f\circ \phi^{X}_{\xi}(p))+\frac{2}{\theta (1-e^{-2\xi})}(Uf\circ \phi^{X}_{\xi}\circ r_{\theta}(p)-Uf\circ \phi^{X}_{\xi}(p))\biggr)\;\text{d}\xi\;.
	\end{split}
\end{equation*}

We shall know treat the four summands separately. 

\begin{lem}
	\label{lem:regularterms}
	If $g$ takes one of the forms in~\eqref{eq:g}, the functions 
	\begin{align}
		\label{eq:firstreg}
		&p\mapsto \int_1^{\infty}\frac{g(\xi)e^{-\xi}}{(1-e^{-2\xi})^{2}}\biggl(\int_0^{\theta}f\circ \phi_{\xi}^{X}\circ r_s(p)\;\emph{d}s\biggr)\emph{d}\xi\;,\\
		\label{eq:secondreg}
		&p\mapsto \int_1^{\infty}\frac{g(\xi)e^{-\xi}}{1-e^{-2\xi}}\biggl(\int_0^{\theta}Xf\circ \phi_{\xi}^{X}\circ r_s(p)\;\emph{d}s\biggr)\emph{d}\xi\;, \\
		\label{eq:thirdreg}
		&p\mapsto \int_1^{\infty}\frac{g(\xi)e^{-\xi}}{(1-e^{-2\xi})^{2}}\biggl(f\circ \phi_{\xi}^{X}\circ r_\theta(p)-f\circ \phi_{\xi}^{X}(p)\biggr)\emph{d}\xi
	\end{align}
	are of class $\mathscr{C}^{1}$ on $M$.
\end{lem}
\begin{proof}
	The proof proceeds along the same lines as the proof of Lemma~\ref{lem:diffunderint}. The crucial point is that, for any point $q\in M$ and any $\xi\geq 1$, the operator norm of the differential $(\mathrm{d}\phi_\xi^{X})_{q}$ doesn't exceed (up to a constant factor depending only on the choice of a Riemannian metric on the tangent bundle $TM$) the quantity $e^{\xi/2}$, as direct computations allow to verify starting from the explicit expression of $\phi_{\xi}^{X}$ in~\eqref{eq:geodesicflow}. As a consequence, there exists a constant $C>0$ such that
	\begin{equation*}
		\biggl|\frac{g(\xi)e^{-\xi}}{(1-e^{-2\xi})^{2}}\biggl(\int_0^{\theta}\partial_{x_i}|_q(f\circ \phi^{X}_{\xi}\circ r_s)\;\text{d}s\biggr)\biggr|\leq C\frac{|g(\xi)|e^{-\xi}}{(1-e^{-2\xi})^{2}}\int_0^{\theta}\norm{f}_{\mathscr{C}^{1}}e^{\frac{\xi}{2}}\;\text{d}s=C\norm{f}_{\mathscr{C}^{1}}\frac{|g(\xi)|e^{-\frac{\xi}{2}}}{(1-e^{-2\xi})^{2}}
	\end{equation*}
	for every $\xi\geq 1$ and $i=1,2,3$, where $(\partial_{x_i})_{i=1,2,3}$ is a local frame of $TM$ around a given fixed point $p_0 \in M$. Since the function $|g(\xi)|e^{-\xi/2}/(1-e^{-2\xi})^{2}$ is integrable on the half-line $[1,\infty)$, we deduce as in the proof of Lemma~\ref{lem:diffunderint} that the function in~\eqref{eq:firstreg} is of class $\cC^{1}$ on $M$.
	
	The same assertion for the remaining two functions in~\eqref{eq:secondreg} and~\eqref{eq:thirdreg} follows by a similar argument.  
\end{proof}

What is left to investigate, up to multiplicative constants, is thus  the regularity of the function
\begin{equation}
	\label{eq:holdercontinuousterm}
	p\mapsto \int_1^{\infty}\frac{g(\xi)}{1-e^{-2\xi}}\bigl(Uf\circ \phi^{X}_{\xi}\circ r_{\theta}(p)-Uf\circ \phi^{X}_{\xi}(p)\bigr) \text{d}\xi
\end{equation}
on the manifold $M$. As we shall presently see, the latter depends on the function $g$.

\begin{lem}
	\label{lem:moreregular}
	If $g(\xi)=e^{-\frac{1+\nu}{2}\xi}$, then the function in~\eqref{eq:holdercontinuousterm} is of class $\mathscr{C}^{1}$ on $M$.
\end{lem}
\begin{proof}
	It suffices to argue as in the proof of Lemma~\ref{lem:regularterms} observing that, when $0<\nu<1$, the function
	\begin{equation*}
		\frac{e^{-\frac{1+\nu}{2}\xi}}{1-e^{-2\xi}}\;e^{\frac{\xi}{2}}=\frac{e^{-\frac{\nu}{2}\xi}}{1-e^{-2\xi}}
	\end{equation*}
	is integrable on the half-line $[1,\infty)$.
\end{proof}

It is straightforward to realize that the same argument does not carry over to the other possible forms of $g(\xi)$ listed in~\eqref{eq:g}. For those remaining cases, we instead establish H\"{o}lder-continuity of the function in~\eqref{eq:holdercontinuousterm} by a different argument.

Fix a Riemannian metric $g$ on the connected manifold $M$, inducing a Riemannian distance function $d$. The choice is immaterial for our purposes, as pointed out in Remark~\ref{rmk:equivalentmetrics}. We start with the following well-known properties of the flows $(\phi_t^{X})_{t\in \R},\;(r_{s})_{s\in \R}$.

\begin{lem}
	\label{lem:distanceflows}
	There exist real constants $C_{X,d},C_{\Theta,d}$, depending only on $d$, such that, for any pair of points $p,q\in M$, it holds 
	\begin{equation}
		\label{eq:geodesicdistance}
		d(\phi_t^{X}(p),\phi^{X}_t(q))\leq C_{X,d}\;e^{|t|}d(p,q)
	\end{equation}
	for every $t\in \R$ and
	\begin{equation}
		\label{eq:rotationdistance}
		d(r_{s}(p),r_s(q))\leq C_{\Theta,d}\;d(p,q) 
	\end{equation}
	for every $s\in \R$.
\end{lem}
\begin{proof}
	By compactness of $M$, we have the freedom to prove the lemma for a judicious choice of $d$. To profit most from the algebraic description of the flows $(\phi_t)_{t\in \R}$ and $(r_{s})_{s\in \R}$, we fix a left-invariant Riemannian metric $g_{\SL_2(\R)}$ on the Lie group $\SL_2(\R)$ and let $d$ be the Riemannian distance determined by the unique Riemannian metric $g$ on $M$ for which the projection map $(\SL_2(\R),g_{\SL_2(\R)})\to (M,g)$ is a Riemannian submersion (cf.~\cite[Cor.~2.29]{Lee-Riemannian}) or, equivalently for a covering map, a local isometry. As $M$ is compact, we can choose a finite open cover $(\tilde{U_{i}})_{i\in I}$ of $M$ and a collection $\mathcal{U}=(U_i)_{i\in I}$ of open subsets of $\SL_2(\R)$ such that, for any $i \in I$, the restriction of the projection to $U_i$ is an isometry from $U_i$ onto $\tilde{U_i}$. 
	
	The distance $d_{\SL_2(\R)}$ induced by $g_{\SL_2(\R)}$ is locally equivalent to the distance induced by the operator norm $\norm{\cdot}_{\text{op}}$ on the vector space of $2\times 2$ real matrices corresponding to the Euclidean norm on $\R^{2}$ (cf.~\cite[Lem.~9.12]{Einsiedler-Ward}): for every $g\in \SL_2(\R)$, there exists an open neighborhood  $W_g$ of $g$ and a constant $C_{d,g}$ such that 
	\begin{equation*}
		C_{d,g}^{-1}\norm{x-y}_{\text{op}}\leq d_{\SL_2(\R)}(x,y)\leq C_{d,g}\norm{x-y}_{\text{op}}
	\end{equation*}
	for any $x,y\in W_g$. Upon restricting the $U_i$'s (and the $\tilde{U_i}$'s) if necessary, we may assume that each $\tilde{U_i}$ is contained in $W_{g_i}$ for some $g_i \in \SL_2(\R)$; define $C_{d}$ to be the supremum of the $C_{d,g_i},\;i \in I$. We also select a second finite open cover $\mathcal{V}=(V_j)_{j\in J}$ in such a way that the closure of each $V_j$ is compact and contained in some $U_{i(j)}$. Observe that, for every $j \in J$, the function $\tau_j\colon \overline{V_j}\to (0,\infty]$ defined as the first exit time
	\begin{equation*}
		\tau_j(p)=\inf\{t>0:\phi_t^{X}(p)\notin U_{i(j)} \}\;, \quad  p\in V_{j}
	\end{equation*}
	is continuous, and as such attains a strictly positive minimal value $t_{j}$. Set $t_0=\inf_{j\in J}t_j$ and let $\delta_{\mathcal{V}}$ be a Lebesgue number for the covering $\mathcal{V}$ (cf.~\cite[Lem.~27.5]{Munkres}).
	
	Consider now two points $p,q\in M$, and suppose first that $d(p,q)\geq\delta_{\mathcal{V}}$; then 
	\begin{equation}
		\label{eq:largedistance}
		d(\phi_t^{X}(p),\phi_t^{X}(q))\leq \text{diam}_d(M)\leq \delta_{\mathcal{V}}^{-1}\text{diam}_{d}(M)d(p,q)\leq \delta_{\mathcal{V}}^{-1}\text{diam}_{d}(M)e^{|t|}d(p,q)
	\end{equation}
	for every $t\in \R$, where $\text{diam}_{d}(M)=\sup_{p',q'\in M}d(p,q)$ is the diameter of $M$ with respect to the distance $d$.
	
	Now assume that $d(p,q)<\delta_{\mathcal{V}}$ so that $p$ and $q$ both lie in some $V_j$. Choose representatives $x,y$ of $p,q$, respectively, inside $\tilde{U}_{i(j)}$; then, for every $0\leq t<t_0$, we have
	\begin{equation}
		\label{eq:smalldistance}
		\begin{split}
			d(\phi_t^{X}(p),\phi_t^{X}(q))&\leq d_{\SL_2(\R)}(x\exp{tX},y\exp{tX})\leq C_{d}\norm{x\exp{tX}-y\exp{tX}}_{\text{op}}\\
			&\leq C_{d}\norm{x-y}_{\text{op}}\norm{\exp{tX}}_{\text{op}}\leq C_{d}^{2}\;d_{\SL_2(\R)}(x,y) e^{|t|/2}\\
			&\leq C_d^{2}e^{t}d_{\SL_2(\R)}(x,y)=C_d^{2}e^{t}d(p,q)\;,
		\end{split}
	\end{equation}
	where we used the fact that $p,q,\phi_t^{X}(p),\phi_t^{X}(q)$ all belong to $U_{i}^{(j)}$. If now $t_0\leq t<2 t_0$, then we distinguishes two cases.
	\begin{itemize}
		\item If $d(\phi_{t_0}^{X}(p),\phi_{t_0}^{X}(q))\geq \delta_{\mathcal{V}}$, then~\eqref{eq:largedistance} applies giving
		\begin{equation*}
			d(\phi_t^{X}(p),\phi_t^{X}(q))\leq \delta_{\mathcal{V}}^{-1}\text{diam}_{d}(M)e^{t-t_0}d(\phi_{t_0}^{X}(p),\phi_{t_0}^{X}(q))\leq C_{d}^{2}\delta_{\mathcal{V}}^{-1}\text{diam}_d(M)e^{t}d(p,q)\;.
		\end{equation*}
		This estimate is actually valid for any $t\geq t_0$.
		\item If $d(\phi_{t_0}^{X}(p),\phi_{t_0}^{X}(q))< \delta_{\mathcal{V}}$, then the computations in~\eqref{eq:smalldistance} are valid for the given $t$ and yield
		\begin{equation*}
			d(\phi_t^{X}(p),\phi_t^{X}(q))\leq C_d^{2}e^{t}d(p,q)\;.
		\end{equation*} 
	\end{itemize}
	Subdividing the half-line $\R_{\geq 0}$ into the intervals $[kt_0,(k+1)t_0),\;k\in \N$, and arguing as above on each of them, we conclude that
	\begin{equation*}
		d(\phi_t^{X}(p),\phi_t^{X}(q))\leq \sup\{ C_d^{2}, C_d^{2}\delta_{\mathcal{V}}^{-1}\text{diam}_d(M)\}e^{t}d(p,q)
	\end{equation*}
	for any $t>0$. 
	
	The same analysis can be performed, with the appropriate modifications, for times $t<0$. This shows~\eqref{eq:geodesicdistance}.
	
	The inequality in~\eqref{eq:rotationdistance} is taken care of in an entirely analogous fashion, observing that $\norm{\exp(s\Theta)}_{\text{op}}=1$ for every $s\in \R$.
\end{proof}

Let us now fix two points $p,q \in M$. As in the previous proof, we denote by $\text{diam}_d(M)$ the diameter of $M$; with
\begin{equation*} \text{Lip}_d(Uf)=\sup\limits_{p',q'\in M,\; p'\neq q'}\frac{|Uf(p')-Uf(q')|}{d(p,q)}
\end{equation*}
we indicate the Lipschitz constant of the function $Uf$ with respect to the distance $d$. 

We may then estimate
\begin{equation}
	\label{eq:Holdercontinuity}
	\begin{split}
		&\biggl| \int_1^{\infty}\frac{g(\xi)}{1-e^{-2\xi}}\bigl(Uf\circ \phi^{X}_{\xi}\circ r_{\theta}(p)-Uf\circ \phi^{X}_{\xi}(p)\bigr) \text{d}\xi - \int_1^{\infty}\frac{g(\xi)}{1-e^{-2\xi}}\bigl(Uf\circ \phi^{X}_{\xi}\circ r_{\theta}(q)-Uf\circ \phi^{X}_{\xi}(q)\bigr) \text{d}\xi \biggr|\\
		&\leq \int_1^{\infty}\frac{|g(\xi)|}{1-e^{-2\xi}}\bigl(|Uf\circ \phi_{\xi}^{X}\circ r_{\theta}(p)-Uf\circ \phi^{X}_{\xi}\circ r_{\theta}(q)|+|Uf\circ \phi^{X}_{\xi}(p)-Uf\circ \phi^X_{\xi}(q)|\bigr)\text{d}\xi\\
		&\leq (1-e^{-2})^{-1}\text{Lip}(Uf) \int_1^{\infty}|g(\xi)|\bigl(\min\{\text{diam}_d(M),C_{X,d}\;e^{\xi}d(r_{\theta}(p),r_{\theta}(q) \}+\\ &\quad +\inf\{\text{diam}_d(M),C_{X,d}\;e^{\xi}d(p,q) \}\bigr)\text{d}\xi\\
		&\leq 2(1-e^{-2})^{-1}\text{Lip}(Uf)\int_1^{\infty}|g(\xi)|\inf\{\text{diam}_d(M),Cd(p,q)e^{\xi} \}\text{d}\xi\;,
	\end{split}
\end{equation}
where $C=C_{X,d}\sup\{1,C_{\Theta,d} \}$.

\smallskip
The following elementary estimates allow to finalize the argument.
\begin{lem}
	\label{lem:integralbound}
	Let $r,K\in \R_{>0},\;a\in (0,1)$. Then 
	\begin{equation}
		\label{eq:firstboundintegral}
		\int_1^{\infty}e^{-a\xi}\inf\{K,re^{\xi} \} \emph{d}\xi\leq \frac{1}{a(1-a)K^{a-1}}\;r^{a}\;.
	\end{equation}
	Furthermore,
	\begin{equation}
		\label{eq:secondboundintegral}
		\int_1^{\infty}\xi e^{-\frac{\xi}{2}}\inf\{K,re^{\xi} \}\emph{d}\xi= 4\sqrt{K}r^{\frac{1}{2}}(\log{k}+\log{r}^{-1})\;.
	\end{equation}
\end{lem}
\begin{proof}
	It suffices to split the integral as
	\begin{equation*}
		\begin{split}
			\int_1^{\infty}e^{-a\xi}\inf\{K,re^{\xi} \} \text{d}\xi&=\int_1^{\log{\frac{K}{r}}}e^{-a\xi}\cdot re^{\xi} \;\text{d}\xi+\int_{\log{\frac{K}{r}}}^{\infty}e^{-a\xi}\cdot K  \;\text{d}\xi\\
			&=\frac{r}{1-a}\biggl(\biggr(\frac{r}{K}\biggr)^{a-1}-e^{1-a}\biggr)+\frac{K}{a}\biggl(\frac{r}{K}\biggr)^{a}\\
			&\leq \frac{r^{a}}{(1-a)K^{a-1}}+\frac{r^{a}}{aK^{a-1}}= \frac{1}{a(1-a)K^{a-1}}\;r^{a} \;,
		\end{split}
	\end{equation*}
	which delivers the inequality in~\eqref{eq:firstboundintegral}. Analogous computations allow to establish~\eqref{eq:secondboundintegral}.
\end{proof}

Combining Lemmata~\ref{lem:regularterms},~\ref{lem:moreregular},~\ref{lem:integralbound} together with the estimate in~\eqref{eq:Holdercontinuity} and the explicit expressions for the coefficients $D^{\pm}_{\theta,\mu,n}f$ in~\eqref{eq:Dplusabovequarter},~\eqref{eq:Dminusabovequarter},~\eqref{eq:Dquarter},~\eqref{eq:Dbelowquarter} we deduce that:
\begin{itemize}
	\item when $\mu>1/4$, $D^{\pm}_{\theta.\mu,n}f$ are H\"{o}lder continuous with H\"{o}lder exponent $1/2$;
	\item when $\mu=1/4$, $D^{+}_{\theta,1/4,n}f$ and $D^{-}_{\theta,1/4,n}$ are H\"{o}lder continuous, the latter with H\"{o}lder exponent $1/2$, while the former with H\"{o}lder exponent $1/2-\eps$ for every $\eps>0$;
	\item when $0<\mu<1/4$, $D^{+}_{\theta,\mu,n}f$ is H\"{o}lder continuous with H\"{o}lder exponent $\frac{1-\nu}{2}$, while $D^{-}_{\theta,\mu,n}f$ is of class $\cC^{1}$ on $M$.
\end{itemize}

The proof of Theorem~\ref{thm:case_Theta_eigenfn} is achieved.

\section{Asymptotics for arbitrary functions}
\label{sec:arbitraryfn}

The bulk of this section is devoted to the deduction of Theorem~\ref{thm:mainexpandingtranslates}, which addresses the asymptotic equidistribution rate of sufficiently regular observables on $M$ not subject to any eigenvalue condition, from the special case of joint eigenfunctions of $\square$ and $\Theta$ phrased in Theorem~\ref{thm:case_Theta_eigenfn}. The argument is crucially based upon the orthogonal decompositions of Sobolev spaces into joint eigenspaces of $\square$ and $\Theta$, which is recalled in detail in Section~\ref{sec:unitaryrepresentations}. We then proceed by proving Theorem~\ref{thm:mainarbitrarytranslates}, concerning the asymptotic behaviour of arbitrary translates of compact orbits inside $M$; in light of the classical Cartan decomposition of the Lie group $\SL_2(\R)$, the result follows from Theorem~\ref{thm:mainexpandingtranslates} in a fairly straightforward manner. Along the way, we shall also clarify the steps needed to derive, from those two main results, Corollaries~\ref{cor:effective} and~\ref{cor:shrinkingarcs}.

\subsection{Sum estimates on Sobolev norms of eigenfunctions}
\label{sec:sumestimates}

Before proceeding with the proof of Theorem~\ref{thm:mainexpandingtranslates}, we collect in this subsection a few  estimates relating sums of norms of Sobolev eigenfunctions with a higher-order Sobolev norm of the sum of such functions, which will prove to be instrumental in the sequel. The rationale for those is the fact that the Hilbert-sum decompositions in Section~\ref{sec:unitaryrepresentations} only provide, by Bessel's inequality (cf.~\cite[2,~XXIII,~6;~14]{Schwartz}), estimates on the sum of squares of the components' norms, while our approach necessitates bounds on the $\ell^{1}$-norm (see Section~\ref{sec:proofexpandingtranslates}).   

Notation is as in Section~\ref{sec:unitaryrepresentations}. 

\begin{lem}
	\label{lem:finitespectralconstant}
	Let $k$ be a natural number. 
	\begin{enumerate}
		\item The infinite series 
		\begin{equation}
			\label{eq:serieseigenvalues}
			\sum_{\mu \in \emph{Spec}(\square)}\sum_{n\in I(\mu)} \frac{1}{(1+\mu+\frac{n^2}{2})^{k}}
		\end{equation}
		is summable if and only if $k\geq 2$.
		\item The infinite series 
		\begin{equation}
			\label{eq:serieseigenvaluessecond}
			\sum_{\mu \in \emph{Spec}(\square)}\sum_{n\in I(\mu)} \frac{(n^{2}+1)^{2}}{(1+\mu+\frac{n^2}{2})^{k}}
		\end{equation}
		is summable if and only if $k\geq 3$.
	\end{enumerate}
\end{lem}

Observe that the series in~\eqref{eq:serieseigenvalues} and~\eqref{eq:serieseigenvaluessecond} consist of nonnegative real numbers: see Lemma~\ref{lem:diffSobolevnorms}.
\begin{proof}
	It is convenient to examine separately the convergence properties of 
	\begin{equation}
		\label{eq:negativeandpositive}
		\sum_{\mu \in \text{Spec}(\square)\cap \R_{\geq 0}}\sum_{n\in I(\mu)} \frac{1}{(1+\mu+\frac{n^2}{2})^{k}}\;\text{  and  } \;\sum_{\mu \in \text{Spec}(\square)\cap \R_{<0}}\sum_{n\in I(\mu)} \frac{1}{(1+\mu+\frac{n^2}{2})^{k}}\;.
	\end{equation}
	We know (cf.~Section~\ref{sec:unitaryrepresentations}) that negative eigenvalues of the Casimir operator are of the form $\mu_m=-m(m+2)/4$ for $m$ ranging over the set of positive natural numbers; therefore
	\begin{equation}
		\label{eq:negativecasimir}
		\sum_{\mu \in \text{Spec}(\square)\cap \R_{<0}}\sum_{n\in I(\mu)} \frac{1}{(1+\mu+\frac{n^2}{2})^{k}}=\sum_{m \in \N^{\ast}}\sum_{n\in I(\mu_m)} \frac{1}{(1-\frac{m(m+2)}{4}+\frac{n^2}{2})^{k}}\;,
	\end{equation}
	which has the same convergence properties of the series
	\begin{equation*}
		\sum_{(m,n)\in \Z^{2}}\frac{1}{(1+m^2+n^{2})^{k}}\;.
	\end{equation*}
	By comparison with the integral
	\begin{equation*}
		\int_{\R^{2}}\frac{1}{(1+x^2+y^2)^{k}}\;\text{d}x\text{d}y\;,
	\end{equation*}
	which is convergent if and only if $k\geq 2$, as it is well-known, we infer that:
	\begin{enumerate}
		\item the series in~\eqref{eq:serieseigenvalues} cannot converge if $k=1$;
		\item the series in~\eqref{eq:negativecasimir} converges for any $k\geq 2$.
	\end{enumerate}
	As to the first summation in~\eqref{eq:negativeandpositive}, we may now suppose $k\geq 2$ and appeal to the Weyl law  for the positive eigenvalues of the Casimir operator (see Theorem~\ref{thm:Weyllaw}), which we list in increasing order as $\mu^{(p)}_0=0<\mu_1^{(p)}<\cdots<\mu^{(p)}_{m}<\cdots$, without multiplicity. Recall that $\text{area}(S)$ is the volume of the surface $S=\Ga\bsl \Hyp$ with respect to the hyperbolic area measure. Choose a real number $c>\text{area}(S)/4\pi$; then, there exists $R_0\in \R_{>0}$ such that $\mu^{(p)}_{m}>m/c$ for any integer $m>cR_0$. On the one hand, the quantity
	\begin{equation*}
		\sum_{0\leq m\leq cR_0}\sum_{n\in I(\mu^{(p)}_{m})}\frac{1}{(1+\mu_m^{(p)}+\frac{n^2}{2})^{k}}
	\end{equation*}
	is a finite sum of converging series. On the other hand, 
	\begin{equation*}
		\begin{split}
			\sum_{m>cR_0}\sum_{n\in I(\mu^{(p)}_{m})}\frac{1}{(1+\mu_m^{(p)}+\frac{n^2}{2})^{k}}&<\sum_{m>cR_0}\sum_{n\in I(\mu^{(p)}_{m})}\frac{1}{(1+\frac{m}{c}+\frac{n^2}{2})^{k}}\\
			&\leq \sum_{m>cR_0}\frac{1}{(1+\frac{m}{c})^{k}}+\sum_{m>cR_0}\sum_{n\in I(\mu^{(p)}_{m})\setminus\{0\}}\frac{1}{(1+\frac{m}{c}+\frac{n^2}{2})^{k}}\;.
		\end{split}
	\end{equation*}
	The first summand in the last expression converges obviously whenever $k\geq 2$, and so does the second by comparison with the integral
	\begin{equation*}
		\begin{split}
			\int_{cR_0}^{\infty}\int_0^{\infty}\frac{1}{\bigl(1+\frac{x}{c}+\frac{y^2}{2}\bigr)^{k}}\;\text{d}y\;\text{d}x&=\int_{cR_0}^{\infty}\frac{1}{\bigl(1+\frac{x}{c}\bigr)^{k}}\int_0^{\infty}\frac{\sqrt{2}\bigl(1+\frac{x}{c}\bigr)^{1/2}}{(1+u^{2})^{k}}\;\text{d}u\;\text{d}x\\
			&=\biggl(\int_{cR_0}^{\infty}\frac{\sqrt{2}}{\bigl(1+\frac{x}{c}\bigr)^{k-1/2}}\;\text{d}x\biggr)\biggl(\int_0^{\infty}\frac{1}{(1+u^{2})^{k}}\;\text{d}u\biggr)<\infty\;.
		\end{split}
	\end{equation*} 
	As to the assertion for the series in~\eqref{eq:serieseigenvaluessecond}, it follows readily by running the argument above with the appropriate modifications.
\end{proof}

Leveraging the estimates in Lemma~\ref{lem:finitespectralconstant}, we are now in a position to show:

\begin{prop}
	\label{prop:L1bound}
	Let $k\geq 2$ be an integer. There exists a constant $C_{\emph{Spec},k}>0$ depending only on $k$ and on the spectrum of the Laplace-Beltrami operator on the hyperbolic surface $S$, such that the following holds. Let $s$ be a positive real number, $f$ a function in $W^{s+k}(H)$; for any $\mu \in \emph{Spec}(\square)$ and $n\in I(\mu)$, let $f_{\mu,n}$ be the orthogonal projection of $f$, with respect to the inner product in $W^{s+k}(H)$, onto the closed subspace $W^{s+k}(H_{\mu,n})$. Then
	\begin{equation}
		\label{eq:L1bound}
		\sum_{\mu \in \emph{Spec}(\square)}\sum_{n\in I(\mu)}\norm{f_{\mu,n}}_{W^{s}}\leq C_{\emph{Spec},k}\norm{f}_{W^{s+k}}\;.
	\end{equation}
\end{prop}
\begin{proof}
	Recall from Lemma~\ref{lem:diffSobolevnorms} that, for any $\mu \in \text{Spec}(\square)$ and $n\in I(\mu)$,
	\begin{equation*}
		\norm{f_{\mu,n}}_{W^{s}}^2=\biggl(1+\mu+\frac{n^2}{2}\biggr)^{-k}\norm{f_{\mu,n}}_{W^{s+k}}^2\;.
	\end{equation*}
	Using the Cauchy-Schwartz inequality, we get
	\begin{equation*}
		\begin{split}
			\sum_{\mu \in \text{Spec}(\square)}\sum_{n\in I(\mu)}\norm{f_{\mu,n}}_{W^{s}}&=\sum_{\mu \in \text{Spec}(\square)}\sum_{n\in I(\mu)}\biggl(1+\mu+\frac{n^2}{2}\biggr)^{-k/2}\norm{f_{\mu,n}}_{W^{s+k}}\\
			&\leq \biggl(\sum_{\mu \in \text{Spec}(\square)}\sum_{n\in I(\mu)}\frac{1}{\bigl(1+\mu+\frac{n^{2}}{2}\bigr)^{k}}\biggr)^{1/2}\biggl(\sum_{\mu \in \text{Spec}(\square)}\sum_{n\in I(\mu)}\norm{f_{\mu,n}}_{W^{s+k}}^{2}\biggr)^{1/2}\;.
		\end{split}
	\end{equation*}
	The inequality in~\eqref{eq:L1bound} is thus a consequence of Parseval's identity (cf.~\cite[2,~XXIII,~6;~17]{Schwartz}) 
	\begin{equation*}
		\norm{f}_{W^{s+k}}^{2}=\sum_{\mu\in \text{Spec}(\square)}\sum_{n\in I(\mu)}\norm{f_{\mu,n}}_{W^{s+k}}^{2}\;,
	\end{equation*}
	where we define the constant $C_{\text{Spec},k}$ as 
	\begin{equation*}
		C_{\text{Spec},k}=\biggl(\sum_{\mu \in \text{Spec}(\square)}\sum_{n\in I(\mu)}\frac{1}{\bigl(1+\mu+\frac{n^{2}}{2}\bigr)^{k}}\biggr)^{1/2}\;,
	\end{equation*} 
	which is finite by Lemma~\ref{lem:finitespectralconstant} and satisfies the dependence properties claimed in the statement (cf.~Sections~\ref{sec:unitaryrepresentations},~\ref{sec:Sobolev}).
\end{proof}

\subsection{Equidistribution of expanding translates}
\label{sec:proofexpandingtranslates}

We are now ready to prove Theorem~\ref{thm:mainexpandingtranslates}.
Let $\theta$ be a real parameter in the interval $(0,4\pi]$, $s$ a real number satisfying $s>11/2$, $f$ a function in $W^{s}(M)$. Keeping with the notation introduced in the foregoing subsection, we denote by $f_{\mu,n}\in W^{s}(H_{\mu,n})$ the orthogonal projection of $f$ onto $W^{s}(H_{\mu,n})$, for any Casimir eigenvalue $\mu$ and any $n\in I(\mu)$. In what follows, the equivalence classes $f$ and $f_{\mu,n}$ are identified with their unique\footnote{Uniqueness is a result of the fact that the uniform measure $\vol$ on $M$ is fully supported.} continuous representatives. The asymptotic expansion in~\eqref{eq:asymptoticgeneral} will result from the sum of the contributions of each component $f_{\mu,n}$, which are provided by Theorem~\ref{thm:case_Theta_eigenfn}; we now expose the details.

Choose a real parameter $s'$ satisfying $3/2<s'\leq s-2$;
the Sobolev Embedding Theorem (Theorem~\ref{thm:Sobolevembedding}) gives the bound $\norm{f_{\mu,n}}_{\infty}\leq C_{0,s'}\norm{f_{\mu,n}}_{W^{s'}}$ for any $\mu$ and $n$ as before, for some constant $C_{0,s'}>0$ depending only on $s'$ and on the manifold $M$.
For any point $q \in M$, we estimate
\begin{equation}
	\label{eq:uniformbound}
	\begin{split}
		\sum_{\mu\in \text{Spec}(\square)}\sum_{n\in I(\mu)}|f_{\mu,n}(q)|&\leq \sum_{n\in \text{Spec}(\square)}\sum_{n\in I(\mu)}\norm{f_{\mu,n}}_{\infty}\leq C_{0,s'}\sum_{\mu\in \text{Spec}(\square)}\sum_{n\in I(\mu)}\norm{f_{\mu,n}}_{W^{s'}}\\
		\leq & C_{0,s'}C_{\text{Spec},2}\norm{f}_{W^{s'+2}}\;,
	\end{split}
\end{equation}
the last inequality being given by Proposition~\ref{prop:L1bound}. Select now a base point $p \in M$, which will remain fixed until the end of this subsection. By virtue of~\eqref{eq:uniformbound}, the dominated convegence theorem for infinite series yields
\begin{equation}
	\label{eq:sumintegral}
	\frac{1}{\theta}\int_0^{\theta}f\circ \phi_t^{X}\circ r_s(p)\;\text{d}s=\sum_{\mu \in \text{Spec}(\square)}\sum_{n\in I(\mu)}\frac{1}{\theta}\int_0^{\theta}f_{\mu,n}\circ \phi^{X}_t\circ r_s(p)\;\text{d}s\;.
\end{equation}
Observe that, since $s>11/2$, the components $f_{\mu,n}$ are of class $\cC^{2}$ by the Sobolev Embedding Theorem (Theorem~\ref{thm:Sobolevembedding}); to each summand on the right-hand side of~\eqref{eq:sumintegral}, we may thus apply Theorem~\ref{thm:case_Theta_eigenfn}, which delivers on a formal level the equality\footnote{Notice that $\int_{M}f_{\mu,n}\;\text{d}\vol=0$ for every Casimir eigenvalue $\mu\neq 0$ and every $n\in I(\mu)$, as $f_{\mu,n}$ is orthogonal to the joint eigenspace $H_{0,0}$ which contains the constant functions. For the same reason, $\int_{M}f_{0,n}\;\text{d}\vol=0$ for every $n\in I(0)\setminus\{0\}$. Therefore, dominated convergence gives $\int_{M}f_{0.0}\;\text{d}\vol=\int_{M}f\;\text{d}\vol$.}
\begin{equation}
	\label{eq:fullexpansion}
	\begin{split}
		\frac{1}{\theta}&\int_0^{\theta}f\circ \phi_t^{X}\circ r_s(p)\;\text{d}s=\int_{M}f\;\text{d}\vol\\ 
		&+\sum_{\mu\in \text{Spec}(\square),\;\mu\geq 1/4}e^{-\frac{t}{2}}\biggl(\cos{\biggl(\frac{\Im{\nu}}{2}t\biggr)}\biggl(\sum_{n\in I(\mu)}D^{+}_{\theta,\mu,n}f_{\mu,n}(p)\biggr)+\sin{\biggl(\frac{\Im{\nu}}{2}t\biggr)}\biggl(\sum_{n\in I(\mu)}D^{-}_{\theta,\mu,n}f_{\mu,n}(p)\biggr)\biggr)\\
		&+\sum_{\mu\in \text{Spec}(\square),\;0<\mu< 1/4}e^{-\frac{1+\nu}{2}t}\biggl(\sum_{n\in I(\mu)}D^{+}_{\theta,\mu,n}f_{\mu,n}(p)\biggr)+e^{-\frac{1-\nu}{2}t}\biggl(\sum_{n\in I(\mu)}D^{-}_{\theta,\mu,n}f_{\mu,n}(p)\biggr) \\
		&+ \eps_0\biggl(e^{-\frac{t}{2}}\biggl(\sum_{n\in I(1/4)}D^{+}_{\theta,1/4,n}f_{1/4,n}(p)\biggr)+te^{-\frac{t}{2}}\biggl(\sum_{n\in I(1/4)}D^{-}_{\theta,1/4,n}f_{1/4,n}(p)\biggr)\biggr)\\
		&+ \mathcal{R}_{\theta,\text{pos}}f(p,t)+ e^{-t}\mathcal{M}_{\theta,0}f(p,t)+\sum_{n\in I(0)}\mathcal{R}_{\theta,0,n}f_{0,n}(p,t) +\mathcal{R}_{\theta,\text{d}}f(p,t)
	\end{split}
\end{equation}
for every $t\geq 1$,
where $\eps_0$ is defined in~\eqref{eq:epszero}, $G_{\theta,n}f_{0,n}$ is as in~\eqref{eq:constantterm} and the quantities
\begin{equation*}
	\mathcal{R}_{\theta,\text{pos}}f(p,t),\quad\mathcal{M}_{\theta,0}f(p,t),\quad\mathcal{R}_{\theta,\text{d}}f(p,t)
\end{equation*} 
are defined as follows:
\begin{align}
	\label{eq:remainderpos}
	&\mathcal{R}_{\theta,\text{pos}}f(p,t)=\sum_{\mu \in \text{Spec}(\square),\;\mu>0}\;\sum_{n\in I(\mu)}\mathcal{R}_{\theta,\mu,n}f(p,t)\;,\\
	\label{eq:remaindernull}
	&\mathcal{M}_{\theta,0}f(p,t)=\sum_{n\in I(0)}\int_1^{t}-G_{\theta,n}f_{0,n}(p,\xi)\;\text{d}\xi \;,\\
	\label{eq:remainderneg}
	&\mathcal{R}_{\theta,\text{d}}f(p,t)=\sum_{\mu\in \text{Spec}(\square),\;\mu<0}\sum_{n\in I(\mu)}\frac{1}{\theta}\int_0^{\theta}f_{\mu,n}\circ\phi_t^{X}\circ r_{s}(p)\;\text{d}s\;.
\end{align}
The equality in~\eqref{eq:asymptoticgeneral} would follow directly from~\eqref{eq:fullexpansion} by defining
\begin{equation}
	\label{eq:defDthetamu}
	D^{\pm}_{\theta,\mu}f(p)=\sum_{n\in I(\mu)}D^{\pm}_{\theta,\mu,n}f_{\mu,n}(p)\;, \quad p\in M,\; \mu \in \text{Spec}(\square)\cap \R_{>0}
\end{equation} 
and
\begin{equation}
	\label{eq:remainder}
	\mathcal{R}_{\theta}f(p,t)=\mathcal{R}_{\theta,\text{pos}}f(p,t)+e^{-t}\mathcal{M}_{\theta,0}f(p,t)+\sum_{n\in I(0)}\mathcal{R}_{\theta,0,n}f_{0,n}(p,t)+\mathcal{R}_{\theta,\text{d}}f(p,t)\;,\;p\in M,\;t\geq 1.
\end{equation}
It is left to show that all the infinite sums we are considering with a formal meaning are actually convergent.

Let us begin by examining the sums in~\eqref{eq:defDthetamu}. Fix $\mu \in \text{Spec}(\square)\cap \R_{>0}$; for any $p \in M$ and $n\in I(\mu)$, we have from Theorem~\ref{thm:case_Theta_eigenfn} that
\begin{equation*}
	|D^{\pm}_{\theta,\mu,n}f_{\mu,n}(p)|\leq \norm{D^{\pm}_{\theta,\mu,n}f_{\mu,n}}_{\infty}\leq \frac{\kappa(\mu)}{\theta}(n^{2}+1)\norm{f_{\mu,n}}_{\mathscr{C}^{1}}\;.
\end{equation*}
Choose now a real parameter $s''$ so that $5/2<s''\leq s-3$; then Theorem~\ref{thm:Sobolevembedding} allows to deduce
\begin{equation}
	\label{eq:uniformboundD}
	\norm{D^{\pm}_{\theta,\mu,n}f_{\mu,n}}_{\infty}\leq C_{1,s''}\frac{\kappa(\mu)}{\theta}(n^{2}+1)\norm{f_{\mu,n}}_{W^{s''}}\;.
\end{equation}
Now, in view of Lemma~\ref{lem:diffSobolevnorms} and applying Cauchy-Schwartz's inequality and Parseval's identity (cf.~\cite[2,~XXIII,~6;~17]{Schwartz}), we get
\begin{equation}
	\label{eq:sameCasimir}
	\begin{split}
		\sum_{n\in I(\mu)}(n^{2}+1)\norm{f_{\mu,n}}_{W^{s''}}&=\sum_{n\in I(\mu)}\frac{n^{2}+1}{\bigl(1+\mu+\frac{n^2}{2})^{3/2}}\norm{f_{\mu,n}}_{W^{s''+3}}\\
		&\leq \biggl(\sum_{n\in I(\mu)}\frac{(n^2+1)^{2}}{\bigl(1+\mu+\frac{n^2}{2}\bigr)^{3}}\biggr)^{1/2}\biggl(\sum_{n\in I(\mu)}\norm{f_{\mu,n}}^{2}_{W^{s''+3}}\biggr)^{1/2}\\
		&= \biggl(\sum_{n\in I(\mu)}\frac{(n^{2}+1)^2}{\bigl(1+\mu+\frac{n^{2}}{2}\bigr)^{3}}\biggr)^{1/2} \norm{f_{\mu}}_{W^{s''+3}}
	\end{split}
\end{equation}
where $f_{\mu}$ is the orthogonal projection of $f$ onto the closed subspace $W^{s}(H_{\mu})$, and the infinite sum in the last expression converges (see Lemma~\ref{lem:finitespectralconstant}).

Defining
\begin{equation}
	\label{eq:spectralconstant}
	C_{\mu}= \biggl(\sum_{n\in I(\mu)}\frac{(n^{2}+1)^2}{\bigl(1+\mu+\frac{n^{2}}{2}\bigr)^{3}}\biggr)^{1/2}\;,
\end{equation}
our argument thus leads, combining~\eqref{eq:uniformboundD} and~\eqref{eq:sameCasimir}, to the estimate
\begin{equation}
	\label{eq:sumcoefficients}
	\sum_{n\in I(\mu)}\norm{D^{\pm}_{\theta,\mu,n}f_{\mu,n}}_{\infty}\leq C_{1,s''}C_{\mu} \frac{\kappa(\mu)}{\theta}\norm{f_{\mu}}_{W^{s''+3}}\;,
\end{equation}
which implies that the sums $\sum_{n\in I(\mu)}D^{\pm}_{\theta,\mu,n}f_{\mu,n}$ converge normally in the Banach space $\mathscr{C}(M)$, hence absolutely and uniformly for all $p \in M$. In particular, the functions $D^{\pm}_{\theta,\mu}f$ are well-defined, continuous on $M$ and fulfill the upper bound
\begin{equation}
	\label{eq:boundDmu}
	\norm{D^{\pm}_{\theta,\mu}f}_{\infty}\leq C_{1,s-3}C_{\mu} \frac{\kappa(\mu)}{\theta}\norm{f_{\mu}}_{W^{s}},
\end{equation}
obtained by picking $s''=s-3$ in~\eqref{eq:sumcoefficients}. 

We now consider sums over all positive Casimir eigenvalues. We have
\begin{equation*}
	\sum_{\mu\in \text{Spec}(\square),\;\mu>1/4} |D^{\pm}_{\theta,\mu}f(p)|\leq \sum_{\mu\in \text{Spec}(\square),\;\mu>1/4} \norm{D^{\pm}_{\theta,\mu}f}_{\infty}\leq \frac{C_{1,s-3}}{\theta}\sum_{\mu \in \text{Spec}(\square),\;\mu>1/4}C_{\mu}\kappa(\mu)\norm{f_{\mu}}_{W^{s}}\;.
\end{equation*}
Observe that $\kappa(\mu)$ is uniformly bounded by a constant $C_{\text{Spec},\text{pos}}$ depending only on the infimum of the set $ \text{Spec}(\square)\cap (1/4,\infty)$ (see~\eqref{eq:kappamu}); recalling the definition of $C_{\mu}$ in~\eqref{eq:spectralconstant}, we apply once again Cauchy-Schwartz's inequality and Parseval's identity to infer
\begin{equation}
	\label{eq:plustwo-three}
	\sum_{\mu\in \text{Spec}(\square),\;\mu>1/4} \norm{D^{\pm}_{\theta,\mu}f}_{\infty}\leq C_{\text{Spec},\text{pos}}C_{1,s-3}\norm{f}_{W^{s}}\biggl(\sum_{\mu \in \text{Spec}(\square),\;\mu>1/4}\;\sum_{n\in I(\mu)}\frac{(n^2+1)^2}{\bigl(1+\mu+\frac{n^{2}}{2}\bigr)^{3}}\biggr)^{1/2}\;,
\end{equation}
where the term between parentheses on the right-hand side is finite because of Lemma~\ref{lem:finitespectralconstant}. 

Since the spectrum of the Casimir operator is discrete, there are only finitely many distinct eigenvalues in the interval $(0,1/4)$, so that the series 
\begin{equation*}
	\sum_{\mu\in \text{Spec}(\square),\;\mu>1/4}\norm{D^{\pm}_{\theta,\mu}f}_{\infty}
\end{equation*}
involves only finitely many additional terms with respect to~\eqref{eq:plustwo-three}; each of those terms can be bounded with the help of~\eqref{eq:boundDmu}. The claim in~\eqref{eq:boundDthetamu} follows, by defining the constant $C'_{\text{Spec}}$ appropriately in terms of $C_{\text{Spec},\text{pos}}$ and of the $C_{\mu}$ for $0<\mu<1/4$.

\medskip
In order to finalize the proof of Theorem~\ref{thm:mainexpandingtranslates}, we address now the remainder terms defined in~\eqref{eq:remainderpos},\linebreak
\eqref{eq:remaindernull}, ~\eqref{eq:remainderneg} and~\eqref{eq:remainder}.

We start with the term in~\eqref{eq:remainderpos} stemming from positive Casimir eigenvalues. Define $\mu_{\text{princ}}$ to be the infimum of $\text{Spec}(\square)\cap (1/4,\infty)$, and let  $\nu_{\text{princ}}$ be the corresponding parameter fulfilling $1-\nu_{\text{princ}}^{2}=4\mu_{\text{princ}}$. Using the bounds for the remainder terms $\mathcal{R}_{\theta,\mu,n}f_{\mu,n}$ corresponding to the single components $f_{\mu,n}$, provided by Theorem~\ref{thm:case_Theta_eigenfn}, we estimate


\begin{equation}
	\label{eq:estimateremainder}
	\begin{split}
		&\sum_{\mu \in \text{Spec}(\square),\;\mu>1/4}\;\sum_{n\in I(\mu)}|\mathcal{R}_{\theta,\mu,n}f_{\mu,n}(p,t)|\\
		&\leq \sum_{\mu \in \text{Spec}(\square),\;\mu>1/4}\frac{8\kappa_0}{\theta\Im{\nu}}e^{-t}\sum_{n\in I(\mu)}(n^{2}+1)\norm{f_{\mu,n}}_{\cC^{1}}\\
		&\leq  \frac{8\kappa_0C_{1,s-3}}{\theta}e^{-t} \sum_{\mu \in \text{Spec}(\square),\;\mu>1/4}\frac{1}{\Im{\nu}}\sum_{n\in I(\mu)}(n^{2}+1)\norm{f_{\mu,n}}_{W^{s-3}}\\
		&\leq  \frac{8\kappa_0C_{1,s-3}}{\theta\; \Im{\nu_{\text{princ}}}}e^{-t} \sum_{\mu \in \text{Spec}(\square),\;\mu>1/4}\biggl(\sum_{n\in I(\mu)}\frac{(n^{2}+1)^2}{\bigl(1+\mu+\frac{n^{2}}{2}\bigr)^{3}}\biggr)^{1/2}\biggl(\sum_{n\in I(\mu)}\norm{f_{\mu,n}}_{W^{s}}^{2}\biggr)^{1/2}\\
		&\leq  \frac{8\kappa_0C_{1,s-3}}{\theta\; \Im{\nu_{\text{princ}}}}e^{-t} \biggl(\sum_{\mu \in \text{Spec}(\square),\;\mu>1/4}\sum_{n\in I(\mu)}\frac{(n^{2}+1)^2}{\bigl(1+\mu+\frac{n^{2}}{2}\bigr)^{3}}\biggr)^{1/2}\biggl(\sum_{\mu \in \text{Spec}(\square),\;\mu>1/4}\sum_{n\in I(\mu)}\norm{f_{\mu,n}}_{W^{s}}^{2}\biggr)^{1/2}\\
		&\leq \frac{8\kappa_0C_{1,s-3}C_{\text{Spec},3}}{\theta}\frac{1}{\Im{\nu_{\text{princ}}}}\norm{f}_{W^{s}}e^{-t}
	\end{split}
\end{equation}
for any $t\geq 1$, applying the bound in~\eqref{eq:estremainderabovequarter}, Theorem~\ref{thm:Sobolevembedding}, the Cauchy-Schwartz's inequality (twice) and Bessel's inequality (cf.~\cite[2,~XXIII,~6;~14]{Schwartz}) to $W^{s}(M)$. 
Similarly, the bounds in~\eqref{eq:estquarter} and~\eqref{eq:estremainderbelowquarter} yield, respectively,
\begin{equation*}
	\sum_{n\in I(1/4)}|\mathcal{R}_{\theta,1/4,n}f_{1/4,n}(p,t)|\leq \frac{4\kappa_0 C_{1,s-3}C_{\text{Spec},3}}{\theta}\norm{f}_{W^{s}}(t+1)e^{-t}
\end{equation*}
and
\begin{equation*}
	\sum_{\mu \in \text{Spec}(\square),\;0<\mu<1/4}\;\sum_{n\in I(\mu)}|\mathcal{R}_{\theta,\mu,n}f_{\mu,n}(p,t)|\leq \frac{4\kappa_0 C_{1,s-3}C_{\text{Spec},3}C_{\text{comp}}}{\theta}\norm{f}_{W^{s}}e^{-t}
\end{equation*}
for any $t\geq 1$, where we set 
\begin{equation*}
	C_{\text{comp}}=\sum_{\mu \in \text{Spec}(\square),\;0<\mu<1/4}\frac{1}{\nu(1-\nu)(1+\nu)}\;.\end{equation*}
Defining thus
\begin{equation}
	\label{eq:Cpos}
	C_{\text{pos}}=\frac{2}{\Im{\nu_{\text{princ}}}}+C_{\text{comp}}+1
\end{equation}
and applying the triangular inequality for infinite sums, we get from~\eqref{eq:remainderpos} that
\begin{equation}
	\label{eq:estimateposremainder}
	|\mathcal{R}_{\theta,\text{pos}}f(p,t)|\leq \frac{4\kappa_0C_{1,s-3}C_{\text{Spec},3}C_{\text{pos}}}{\theta}\norm{f}_{W^{s}}(t+1)e^{-t}\;, \quad t\geq 1.
\end{equation}
An entirely analogous argument, using the bound in~\eqref{eq:remainderzero}, shows that 
\begin{equation}
	\label{eq:estimatezeroremainder}
	\biggl|\sum_{n\in I(0)}\mathcal{R}_{\theta,0,n}f_{0,n}(p,t)\biggr|\leq\frac{ (8e\pi+\kappa_0)C_{1,s-3}C_{\text{Spec},3}}{\theta}\norm{f}_{W^{s}}e^{-t}\;,\quad t\geq 1.
\end{equation}
Define now 
\begin{equation*}
	C_{\text{disc}}=\biggl(\inf\limits_{\mu \in \text{Spec}(\square),\;\mu<0}|\mu|\biggr)^{-1}+\frac{2e\pi}{\kappa_0}\sup_{\mu \in \text{Spec}(\square),\;\mu<0}\frac{3+\nu}{\nu}\;.
\end{equation*}
Then, the bound in~\eqref{eq:estimatediscreteseries} leads to
\begin{equation*}
	\begin{split}
		&\sum_{\mu\in \text{Spec}(\square),\;\mu<0}\;\sum_{n\in I(\mu)}\frac{1}{\theta}\biggl|\int_0^{\theta}f_{\mu,n}\circ \phi^{X}_{t}\circ r_{s}(p)\;\text{d}s\biggr| \\
		&\leq \sum_{\mu \in \text{Spec}(\square),\;\mu<0}\frac{C_{1,s-3}}{\theta}\biggl(\frac{4\kappa_0}{(\nu-1)(\nu+1)}+\frac{2e\pi(3+\nu)}{\nu}\biggr)e^{-t}\sum_{n\in I(\mu)}(n^{2}+1)\norm{f_{\mu,n}}_{W^{s-3}}\\
		&\leq \frac{\kappa_0C_{1,s-3}C_{\text{Spec},3}C_{\text{disc}}}{\theta}\norm{f}_{W^{s}}e^{-t}
	\end{split}
\end{equation*}
for any $t\geq 1$, arguing as in~\eqref{eq:estimateremainder}. It follows at once from~\eqref{eq:remainderneg} that
\begin{equation}
	\label{eq:estimatenegremainder}
	|\mathcal{R}_{\theta,\text{d}}f(p,t)|\leq \frac{\kappa_0C_{1,s-3}C_{\text{Spec},3}C_{\text{disc}}}{\theta}\norm{f}_{W^{s}}e^{-t}
\end{equation}
for any $t\geq 1$.
We finally come to the estimate of the term $\cM_{\theta,0}f(p,t)$, defined in~\eqref{eq:remaindernull}.
The inequality in~\eqref{eq:Gtheta} gives
\begin{equation*}
	\begin{split}
		\sum_{n\in I(0)}\biggl|\int_1^{t}-G_{\theta,n}f_{0,n}(p,\xi)\;\text{d}\xi\biggr|&\leq \frac{\kappa_0}{\theta}\sum_{n\in I(0)}(n^{2}+1)\int_1^{t}\norm{f_{0,n}}_{\mathscr{C}^{1}}\;\text{d}\xi=\frac{\kappa_0}{\theta}(t-1)\biggl(\sum_{n\in I(0)}(n^{2}+1)\norm{f_{0,n}}_{\mathscr{C}^{1}}\bigg)\\
		&\leq \frac{\kappa_0C_{1,s-3}}{\theta}(t-1)\biggl(\sum_{n\in I(0)}(n^{2}+1)\norm{f_{0,n}}_{W^{s}}\biggr)\;,
	\end{split}
\end{equation*}
so that 
\begin{equation}
	\label{eq:estimatezeromain}
	|\mathcal{M}_{\theta,0}f(p,t)|\leq \frac{\kappa_0C_{1,s-3}C_{\text{Spec},3}}{\theta}\norm{f}_{W^{s}}(t-1)
\end{equation}
for any $t\geq 1$.
Recalling~\eqref{eq:remainder} and combining the estimates in~\eqref{eq:estimateposremainder},~\eqref{eq:estimatezeroremainder},~\eqref{eq:estimatenegremainder} and~\eqref{eq:estimatezeromain} we conclude that, for any $t\geq 1$,
\begin{equation*}
	\begin{split}
		|\mathcal{R}_{\theta}f(p,t)|&\leq |\mathcal{R}_{\theta,\text{pos}}f(p,t)|+e^{-t}|\mathcal{M}_{\theta,0}f(p,t)|+\sum_{n\in I(0)}|\mathcal{R}_{0,\theta,n}f_{0,n}(p,t)|+|\mathcal{R}_{\theta,\text{d}}f(p,t)|\\
		&\leq \frac{C_{1,s-3}C_{\text{Spec}}}{\theta}\norm{f}_{W^s}(t+1)e^{-t}\;,
	\end{split}
\end{equation*}
where we set
\begin{equation*}
	C_{\text{Spec}}=C_{\text{Spec},3}\bigl(8e\pi+\kappa_0(2+C_{\text{pos}}+C_{\text{disc}})\bigr)\;,
\end{equation*}
which ostensibly depends only on the spectrum of the Casimir operator.

The proof of Theorem~\ref{thm:mainexpandingtranslates} is complete.

\subsubsection*{Effective equidistribution and shrinking circle arcs}
In this paragraph we briefly comment on the proof of Corollaries~\ref{cor:effective} and~\ref{cor:shrinkingarcs}.

As to Corollary~\ref{cor:effective}, it suffices to define the function $D^{\text{main}}_{\theta}f\colon M\times \R_{\geq 0}\to \C$ as follows:
\begin{itemize}
	\item $ D^{\text{main}}_{\theta}f=D^{-}_{\theta,\mu_*}f$
	if the spectral gap $\mu_{*}\leq 1/4$;
	\item 
	$D^{\text{main}}_{\theta}f=D^{+}_{\theta,\mu_*}f+D^{-}_{\theta,\mu_*}f$ if $\mu_*>1/4$.
	
\end{itemize}

The effective equidistribution statement in~\eqref{eq:effective} then follows directly from the asymptotics in~\eqref{eq:asymptoticgeneral}.

Now suppose that we let the boundaries of the parametrization depend on the time $t$, so as to deal with a collection of time-varying subarcs
\begin{equation*}
	\gamma_t=\{\phi^{X}_{t}\circ r_s(p):\theta_1(t)\leq s\leq  \theta_2(t) \} 
\end{equation*} 
as in the statement of Corollary~\ref{cor:shrinkingarcs}. If, as in the assumptions to the latter, we suppose that $\theta_2(t)-\theta_1(t)\geq \eta(t)e^{-\frac{1-\Re{\nu_*}}{2}t}$ for every sufficiently large $t$, where $\nu_*$ corresponds to the spectral gap $\mu_*$ and $\eta\colon \R_{>0}\to \R_{>0}$ satisfies $\eta(t)\to\infty$ as $t\to\infty$, then we obtain
\begin{equation}
	\label{eq:varyingarc}
	\begin{split}
		\frac{1}{\theta_2(t)-\theta_1(t)}&\int_{\theta_1(t)}^{\theta_2(t)}f\circ \phi_t^{X}\circ r_{s}(p)\;\text{d}s = \frac{1}{\theta_2(t)-\theta_1(t)}\int_{0}^{\theta_2(t)-\theta_1(t)}f\circ \phi^{X}_t\circ r_{s+\theta_1(t)}(p)\;\text{d}s\\
		&=\int_{M}f\;\text{d}\vol+D^{\text{main}}_{\theta_2(t) - \theta_1(t)}f\left( r_{\theta_1(t)}(p)\right)t^{\eps_0}e^{-\frac{1-\Re{\nu_*}}{2}t}+o(e^{-\frac{1-\Re{\nu_*}}{2}t})\;\;.
	\end{split}
\end{equation}
The bound \eqref{eq:boundDthetamain} results into 
\begin{equation*}
	|D^{\text{main}}_{\theta_2(t) - \theta_1(t)}f\left( r_{\theta_1(t)}(p)\right)|\leq C_{1,s-3}C_{\text{Spec}}'\norm{f}_{W^{s}}\biggl(\frac{1}{\theta_2(t)-\theta_1(t)}\biggr)\leq  C_{1,s-3}C'_{\text{Spec}}\norm{f}_{W^s}  \eta(t)^{-1}e^{\frac{1-\Re{\nu_*}}{2}t}
\end{equation*}
for any $t\geq t_0$.  We deduce that the right-hand side of~\eqref{eq:varyingarc} is equal to $\int_{M}f\;\text{d}\vol+o(t)$ as $t$ tends to infinity. An elementary application of the Stone-Weierstrass' theorem (cf.~\cite[Thm.~4.51]{Folland}) gives that smooth functions are dense in the space of continuous functions on the compact manifold $M$; it follows that the convergence
\begin{equation*}
	\frac{1}{\theta_2(t)-\theta_1(t)}\int_{\theta_1(t)}^{\theta_2(t)}f\circ \phi_t^{X}\circ r_{s}(p)\;\text{d}s\overset{t\to\infty}{\longrightarrow}\int_{M}f\;\text{d}\vol
\end{equation*}
can be upgraded to hold for every $f\in \mathscr{C}(M)$, whereby the desired equidistribution is shown.

\subsubsection*{Equidistribution of circle arcs on the surface}
We conclude this subsection with a few comments concerning the statement of Theorem~\ref{thm:expandingonsurface}, which is nothing but a specialization of Theorem~\ref{thm:mainexpandingtranslates} to the case of observables defined on the underlying surface $S=\Ga\bsl \Hyp$, except for the lower regularity assumed on the test function $f$. First, we remark that $\SO_2(\R)$-invariance of the functions $D^{\pm}_{4\pi,\mu}f$ and $\mathcal{R}_{4\pi}f$ follows at once from their definition (see~\eqref{eq:defDthetamu} and~\eqref{eq:remainder}) and the fact that $f$ is assumed to be $\SO_2(\R)$-invariant. We are only left to show that we might take $s>9/2$, less restrictively in comparison to an arbitrary $f$
defined on $M$. The relevant observation here is that, for any $\SO_2(\R)$-invariant function $f\in L^{2}(M)$, the components $f_{\mu}$ appearing in the decomposition\footnote{Recall from Section~\ref{sec:unitaryrepresentations} that the Casimir operator $\square$ acts as the Laplace-Beltrami operator $\Delta_S$ on $\SO_2(\R)$-invariant functions.}
\begin{equation*}
	f=\sum_{\mu \in \text{Spec}(\Delta_S)}f_{\mu}\;,\quad f_{\mu}\in H_{\mu}
\end{equation*} 
are invariant under $\SO_2(\R)$, that is, they satisfy $\Theta f_{\mu}=0$. The estimate in~\eqref{eq:plustwo-three} thus only requires $s>9/2=11/2-1$, as the sum
\begin{equation*}
	\sum_{\mu \in \text{Spec}(\Delta_S)}\frac{1}{(1+\mu)^{k}}
\end{equation*}
converges already for $k=2$, and not only for $k=3$ as it is the case in~\eqref{eq:plustwo-three}.

\subsection{Equidistribution of arbitrary translates}
\label{sec:arbitrarytranslates}

In light of Theorem~\ref{thm:mainexpandingtranslates}, Theorem~\ref{thm:mainarbitrarytranslates} is a rather straightforward consequence of the classical Cartan decomposition for the semisimple Lie group $\SL_2(\R)$, for which the reader is referred to~\cite[Chap.~VI]{Knapp}. We present the details of the argument in this subsection. 

Let $A=\{\exp{tX}:t\in \R\}$ be the subgroup of $\SL_2(\R)$ consisting of diagonal matrices with positive entries (recall that $X$ is defined as in~\eqref{eq:geodesicflow}). The product map
\begin{equation*}
	\SO_2(\R)\times A\times \SO_2(\R)\to \SL_2(\R),\;(k_1,a,k_2)\mapsto k_1ak_2
\end{equation*} 
is surjective. For any $g\in \SL_2(\R)$, choose a decomposition $g=k_1(g)a(g)k_2(g)$, where $a(g)$ is the diagonal matrix having as entries the singular values of the matrix $g$, in decreasing order. In particular, if $t(g)\in \R_{\geq 0}$ is (uniquely) determined by the condition
\begin{equation}
	\label{eq:Cartanprojection}
	a(g)=
	\begin{pmatrix}
		e^{t(g)/2}&0\\
		0&e^{-t(g)/2}
	\end{pmatrix}
	,
\end{equation}
then it clearly holds that
\begin{equation}
	\label{eq:singvaluenorm} \norm{g}_{\text{op}}=e^{t(g)/2}\;, \quad \text{ or equivalently} \quad  t(g)=2\log{\norm{g}_{\text{op}}}\;.
\end{equation}

Fix now  a real number $s>11/2$ and a function $f$ in the Sobolev space $W^{s}(M)$. Recall that, for any $p \in M$, we indicate with $m_{\SO_2(\R)\cdot p}$ the unique $\SO_2(\R)$-invariant measure supported on the compact orbit $\SO_2(\R)\cdot p$; furthermore, $g_{*}\SO_2(\R)$ denotes the push-forward of the latter measure under the right translation map $R_{g}(\Ga g')=\Ga g'g$ on $M$. For any $p \in M$ and $g\in \SL_2(\R)$, we resort to the Cartan decomposition of $g$ and write
\begin{equation}
	\label{eq:rotationinvariance}
	\begin{split}
		\int_{M}f\;\text{d}g_*m_{\SO_2(\R)\cdot p}&=\int_{M}f\circ R_{g}\;\text{d}m_{\SO_2(\R)\cdot p}=\int_{M}f\circ R_{k_2(g)}\circ R_{a(g)}\circ R_{k_1(g)}\;\text{d}m_{\SO_2(\R)\cdot p}\\
		&=\int_{M}f\circ R_{k_2(g)}\circ R_{a(g)}\;\text{d}m_{\SO_2(\R)\cdot p}=\frac{1}{4\pi}\int_0^{4\pi}\bigl(f\circ R_{k_2(g)}\bigr)\circ \phi^{X}_{t(g)}\circ r_s(p)\;\text{d}s\;,
	\end{split}
\end{equation}
using the $R_{k_1(g)}$-invariance of $m_{\SO_2(\R)\cdot p}$ and the fact that $R_{a(g)}=\phi_{t(g)}^{X}$ in view of~\eqref{eq:Cartanprojection}. 

We may now make use of the asymptotic expansion provided by Theorem~\ref{thm:mainexpandingtranslates} for the function $f\circ R_{k_2(g)}$, which lies in the same Sobolev space $W^{s}(M)$ of $f$ since $R_{k_2(g)}$ is a smooth diffeomorphism of $M$. We thereby obtain, for a fixed base point $p \in M$,
\begin{equation}
	\label{eq:variabletranslates}
	\begin{split}
		&\frac{1}{4\pi}\int_0^{4\pi}\bigl(f\circ R_{k_2(g)}\bigr)\circ \phi^{X}_{t(g)}\circ r_s(p)\;\text{d}s=\int_{M}f\circ R_{k_2(g)}\;\text{d}\vol\\
		&+e^{-\frac{t(g)}{2}}\biggl( \sum_{\mu\in \text{Spec}(\square),\;\mu>1/4}\cos{\biggl(\frac{\Im{\nu}}{2}t(g)\biggr)} D^{+}_{4\pi,\mu}(f\circ R_{k_2(g)})(p)+\sin{\biggl(\frac{\Im{\nu}}{2}t(g)\biggr)} D^{-}_{4\pi,\mu}(f\circ R_{k_2(g)})(p)\biggr)\\
		&+\sum_{\mu\in \text{Spec}(\square),\;0<\mu<1/4}e^{-\frac{1+\nu}{2}t(g)}D^{+}_{4\pi,\mu}(f\circ R_{k_2(g)})(p)+e^{-\frac{1-\nu}{2}t(g)}D^{-}_{4\pi,\mu}(f\circ R_{k_2(g)})(p)\\
		&+\varepsilon_0\bigl(e^{-\frac{t(g)}{2}}D^{+}_{4\pi,1/4}(f\circ R_{k_2(g)})(p)+t(g)e^{-\frac{t(g)}{2}}D^{-}_{4\pi.1/4}(f\circ R_{k_2}(g))(p)\bigr)+\mathcal{R}_{4\pi}(f\circ R_{k_2(g)})(p,t(g))
	\end{split}
\end{equation}
for any $g\in \SL_2(\R)$ with $\norm{g}_{\text{op}}\geq \sqrt{e}$.
Define now, for any Casimir eigenvalue $\mu \in \R_{>0}$, the functions $D^{\pm}_{\mu}\colon M\times \SL_2(\R)\to \C$ by 
\begin{equation}
	\label{eq:differentcoefficients}
	D^{\pm}_{\mu}f(p,g)=D^{\pm}_{4\pi,\mu}(f\circ R_{k_2(g)})(p)\;,\quad p \in M,\;g\in \SL_2(\R),
\end{equation}
and set also
\begin{equation*}
	\mathcal{R}f(p,g)=\mathcal{R}_{4\pi}(f\circ R_{k_2}(g))(p,t(g))\;,\quad p \in M,\;g\in \SL_2(\R).
\end{equation*}

\smallskip
Then, combining~\eqref{eq:rotationinvariance} and~\eqref{eq:variabletranslates} and recalling~\eqref{eq:singvaluenorm} together with the fact that\\ $\int_{M}f\circ R_{k_2}\;\text{d}\vol=\int_{M}f\;\text{d}\vol$, we deduce
\begin{equation*}
	\begin{split}
		\int_{M}f\;\text{d}&g_{*}m_{\SO_2(\R)\cdot p}=\int_{M}f\;\text{d}\vol\\
		&+\sum_{\mu\in \text{Spec}(\square),\;\mu>1/4}\norm{g}^{-1}_{\text{op}}\bigl(\cos{(\Im{\nu}\log{\norm{g}_{\text{op}}})}D^{+}_{\mu}f(p,g)+\sin{(\Im{\nu}\log{\norm{g}_{\text{op}}})}D^{-}_{\mu}f(p,g)\bigr)\\
		&+\sum_{\mu\in \text{Spec}(\square),\;0<\mu<1/4}\norm{g}_{\text{op}}^{-(1+\nu)}D^{+}_{\mu}f(p,g)+\norm{g}_{\text{op}}^{-(1-\nu)}D^{-}_{\mu}f(p,g)\\
		&+\varepsilon_0\bigl(\norm{g}_{\text{op}}^{-1}D^{+}_{1/4}f(p,g)+2\norm{g}_{\text{op}}^{-1}\log{\norm{g}_{\text{op}}}D^{-}_{1/4}f(p,g)\bigr)+\mathcal{R}f(p,g)
	\end{split}
\end{equation*}
for any $g\in \SL_2(\R)$ with $\norm{g}_{\text{op}}\geq \sqrt{e}$,
which is precisely the asymptotic expansion appearing in the statement of Theorem~\ref{thm:mainarbitrarytranslates}.

As stated in Theorem~\ref{thm:mainexpandingtranslates}, the functions $D^{\pm}_{4\pi,\mu}(f\circ R_{k_2(g)})$ are continuous on $M$ for any fixed $\mu \in \text{Spec}(\square)\cap \R_{>0}$ and $g\in \SL_2(\R)$; equivalently, by~\eqref{eq:differentcoefficients}, $D^{\pm}_{\mu}f(\cdot,g)$ is continuous on $M$ for any fixed $g\in \SL_2(\R)$.

Also, for any $p \in M, \;g\in \SL_2(\R)$ and $\mu \in \text{Spec}(\square)\cap \R_{>0}$, we have 
\begin{equation*}
	|D^{\pm}_{\mu}f(p,g)|=|D^{\pm}_{4\pi,\mu}(f\circ R_{k_2(g)})(p)|\leq \norm{D^{\pm}_{4\pi,\mu}(f\circ R_{k_2(g)})}_{\infty}\leq \frac{C_{1,s-3}C_{\mu}\kappa(\mu)}{4\pi}\norm{f\circ R_{k_2(g)}}_{W^{s}}\;,
\end{equation*} 
where the last inequality is given by~\eqref{eq:boundDthetamu}. It remains to observe that compactness of $\SO_2(\R)$ implies that there exists a constant $C_{s,\text{rot}}>0$ such that $\norm{f\circ R_{k}}_{W^{s}}\leq C_{s,\text{rot}}\norm{f}_{W^s}$ for any $k\in \SO_2(\R)$. The proof of this assertion runs along the same lines of the proof of Lemma~\ref{lem:diffunderint}, with the appropriate modifications. Therefore, we get
\begin{equation*}
	\sum_{\mu \in \text{Spec}(\square)\cap \R_{>0}}\sup_{p \in M, \;g\in \SL_2(\R)}|D^{\pm}_{\mu}f(p,g)|\leq \frac{C_{1,s-3}C_{s,\text{rot}}C'_{\text{Spec}}}{4\pi}\norm{f}_{W^{s}}\;,
\end{equation*} 
where $C'_{\text{Spec}}$ is as in Theorem~\ref{thm:mainexpandingtranslates}.

To conclude the proof of Theorem~\ref{thm:mainarbitrarytranslates}, it is left to take care of the remainder term $\mathcal{R}f$. We easily estimate, from~\eqref{eq:globalremainderestimate},
\begin{equation*}
	\begin{split}
		|\mathcal{R}f(p,g)|&=|\mathcal{R}_{4\pi}(f\circ R_{k_2}(g))(p,t(g))|\leq \frac{C_{\text{Spec}}C_{1,s-3}}{4\pi}\norm{f\circ R_{k_2(g)}}_{W^{s}}(t(g)+1)e^{-t(g)}\\
		&\leq \frac{C_{\text{Spec}}C_{1,s-3}C_{s,\text{rot}}}{4\pi}\norm{f}_{W^{s}}(2\log{\norm{g}_{\text{op}}}+1)\norm{g}_{\text{op}}^{-2}
	\end{split}
\end{equation*}
for any $p\in M$ and $g\in \SL_2(\R)$ with $\norm{g}_{\text{op}}\geq \sqrt{e}$.

This achieves the proof of Theorem~\ref{thm:mainarbitrarytranslates}.

\section{Distributional limit theorems for deviations from the average}
\label{sec:CLT}

The purpose of this section is threefold, articulated in three subsections. First, we establish the quantitative distributional convergence claimed in Proposition~\ref{prop:CLT}, from which the qualitative statements in Theorem~\ref{thm:CLT} follow directly; secondly, we prove absence of a central limit theorem as phrased in Theorem~\ref{thm:noCLT}, and finally we explore further ways of examining the statistical behaviour of averages along circle arcs. 

\subsection{Quantitative distributional convergence}
\label{sec:spatialDLT}
Let us fix the length parameter $\theta\in (0,4\pi]$, and consider a real-valued function $f$ lying in the Sobolev space $W^{s}(M)$ for some real $s>11/2$. We are interested in the statistical behaviour of the deviations from the mean 
\begin{equation*}
	d_f(T,p)=\frac{1}{\theta}\int_{0}^{\theta}f\circ \phi^{X}_T\circ r_{s}(p)\;\text{d}s -\int_{M}f\;\text{d}\vol
\end{equation*}
appropriately renormalized, as the time parameter $T$ tends to infinity and when the base point $p$ is sampled according to the uniform probability measure $\vol$ on $M$.
Define
\begin{equation*}
	\mu_f=\inf\{\mu \in \text{Spec}(\square)\cap \R_{>0}:D^{-}_{\theta,\mu}f\text { does not vanish identically on }M  \}\;.
\end{equation*}
As in the hypotheses of Proposition~\ref{prop:CLT}, we assume that $\mu_f$ is finite, that is, the set of Casimir eigenvalues over which the previous infimum is taken is non-empty. Let $\nu_f$ be the corresponding parameter, namely $\nu_f\in \R_{\geq 0}\cup i \R_{>0}$ satisfies $1-\nu_f^{2}=4\mu_f$.

In order to quantify the rate of distributional convergence of the random variables under consideration, we make use of the L\'{e}vy-Prokhorov metric $d_{LP}$ on the set $\mathscr{P}(\R)$ of Borel probability measures on $\R$. We recall that this is defined as 
\begin{equation*}
	d_{LP}(\lambda,\rho)=\inf\{\eps>0:\lambda(Y)\leq \rho(Y_{\eps})+\eps \text{ and }\rho(Y)\leq \lambda(Y_{\eps})+\eps \text{ for every Borel set } Y\subset \R \}
\end{equation*}
for any $\lambda,\rho \in \mathscr{P}(\R)$, where $Y_{\eps}$ denotes the open $\eps$-neighborhood of $Y$ with respect to the Euclidean metric on $\R$. The distance $d_{LP}$ induces the topology of weak convergence of probability measures on $\mathscr{P}(\R)$, namely the coarsest topology for which the maps 
\begin{equation*}
	\mathscr{P}(\R)\ni \lambda \mapsto \int_{\R}\varphi\;\text{d}\lambda \in \R\;, \quad \varphi\colon \R \to \R \text{ continuous and bounded}
\end{equation*}
are continuous.

In the forthcoming estimates we shall make use of the following trivial upper bound for the L\'{e}vy-Prokhorov distance between the laws of two random variables defined on the same probability space and taking on nearby values almost surely.

\begin{lem}
	\label{lem:LPnearbyvariables}
	Let $(\Omega,\cF,\mathbf{P})$ be a probability space, $\eps>0$. Suppose $X,X'\colon \Omega\to \R$ are random variables satisfying
	$|X(\omega)-X'(\omega)|< \eps$ for $\mathbf{P}$-almost every $\omega\in \Omega$. If $\lambda_{X}$ and $\lambda_{Y}$ denote the laws of $X$ and $X'$, respectively, then $d_{LP}(\lambda_{X},\lambda_{X'})\leq \eps$.
\end{lem}
\begin{proof}
	Let $A\subset \R$ be a Borel subset. The event $\{X\in A \}$ is contained in the event $\{X'\in A_{\eps} \}$, up to a $\mathbf{P}$-negligible subset, by the assumption on the distance between $X$ and $X'$. Therefore, 
	\begin{equation*}
		\lambda_{X}(A)=\mathbf{P}(X\in A)\leq \mathbf{P}(X'\in A_{\eps})=\lambda_{X'}(A_{\eps})<\lambda_{X'}(A_{\eps})+\eps\;;
	\end{equation*}
	a similar inequality holds reversing the role of $X$ and $X'$, whence $d_{LP}(\lambda_X,\lambda_{X'})\leq \eps$.
\end{proof}

We now proceed with the proof of Proposition~\ref{prop:CLT} by distinguishing the three different cases $0<\mu_f<1/4\;,\mu_f=1/4$ and $\mu_f>1/4$.

\medskip
Suppose first $0<\mu_f<1/4$. We would then like to show that the random variables
\begin{equation*}
	e^{\frac{1-\nu_f}{2}T}\;d_f(T,p)\;,\quad p\sim \vol
\end{equation*}
converge in distribution, as $T$ tends to infinity, to the random variable $D^{-}_{\theta,\mu_f}f(p)$, $p\sim \vol$. Observe that, by virtue the asymptotic expansion in~\eqref{eq:asymptoticgeneral} and the assumption on $\mu_f$, we have
\begin{equation*}
	\begin{split}
		e^{\frac{1-\nu_f}{2}T}\;d_f(T,p)-D^{-}_{\theta,\mu_f}f(p)&=\sum\limits_{\mu\in \text{Spec}(\square),\;\mu_f<\mu<1/4}e^{-\frac{\nu_f-\nu}{2}T}D^{-}_{\theta,\mu}f(p)\\
		&+e^{-\frac{\nu_f}{2}T}\biggl( \sum_{\mu\in \text{Spec}(\square),\;\mu>1/4}\cos{\biggl(\frac{\Im{\nu}}{2}T\biggr)} D^{+}_{\theta,\mu}f(p)+\sin{\biggl(\frac{\Im{\nu}}{2}T\biggr)} D^{-}_{\theta,\mu}f(p)\biggr)\\
		&+\sum_{\mu\in \text{Spec}(\square),\;0<\mu<1/4}e^{-\frac{\nu_f+\nu}{2}T}D^{+}_{\theta,\mu}f(p)\\
		&+\varepsilon_0\bigl(e^{-\frac{\nu_f}{2}T}D^{+}_{\theta,1/4}f(p)+Te^{-\frac{\nu_f}{2}T}D^{-}_{\theta.1/4}f(p)\bigr)+e^{\frac{1-\nu_f}{2}T}\mathcal{R}_{\theta}f(p,T)\;,
	\end{split}
\end{equation*}
so that, because of the uniform bound in~\eqref{eq:boundDthetamu}, we may estimate
\begin{equation*}
	\biggl|e^{\frac{1-\nu_f}{2}T}\;d_f(T,p)-D^{-}_{\theta,\mu_f}f(p)\biggr|\leq \frac{C_{1,s-3}C'_{\text{Spec}}}{\theta} \norm{f}_{W^s}Te^{-\frac{\nu_f-\Re{\nu_f^{\text{next}}}}{2}T}
\end{equation*}
for any $p \in M$ and $T\geq 1$, where $\nu_f^{\text{next}}$ is the parameter corresponding to the smallest eigenvalue $\mu_f^{\text{next}}$ of the Casimir operator exceeding\footnote{Observe that we may dispense with the additional factor $T$ in the upper bound whenever $\nu_f^{\text{next}}\in \R$.} $\mu_f$. 

By Lemma~\ref{lem:LPnearbyvariables}, and recalling the definitions of $\mathbf{P}_{\theta,f}^{\text{circ}}(T)$ and $\mathbf{P}_{\theta,f}$ introduced in Section~\ref{sec:introductionCLT}, we get
\begin{equation*}
	d_{LP}(\mathbf{P}_{\theta,f}^{\text{circ}}(T),\mathbf{P}_{\theta,f})\leq \frac{C_{1,s-3}C_{\text{Spec}}'}{\theta} \norm{f}_{W^{s}}Te^{-\eta_f T}
\end{equation*}
for $\eta_f=\frac{\nu_f-\Re{\nu_f^{\text{next}}}}{2}$.

\medskip
Similarly, if $\mu_f=1/4$, we readily obtain from~\eqref{eq:asymptoticgeneral} that
\begin{equation*}
	\begin{split}
		T^{-1}e^{\frac{T}{2}}&\;d_f(T,p)-D^{-}_{\theta,1/4}f(p)=T^{-1}\biggl(D^{+}_{\theta,1/4}f(p)+\sum_{\mu \in \text{Spec}(\square),\;\mu>1/4}\biggl(\cos{\biggl(\frac{\Im{\nu}}{2}T\biggr)} D^{+}_{\theta,\mu}f(p)\\
		&+\sin{\biggl(\frac{\Im{\nu}}{2}T\biggr)}D^{-}_{\theta,\mu}f(p)\biggr)+\sum_{\mu \in \text{Spec}(\square),\;0<\mu<1/4}e^{-\frac{\nu}{2}T}D^{+}_{\theta,\mu}f(p)+e^{\frac{T}{2}}\mathcal{R}_{\theta}f(p,t)\biggr)
	\end{split}
\end{equation*}
for any $p \in M$ and $T\geq 1$; recalling the definition of the constant $C_{\text{pos}}$ in~\eqref{eq:Cpos}, we deduce the bound
\begin{equation*}
	\bigl|T^{-1}e^{\frac{T}{2}}\;d_f(T,p)-D^{-}_{\theta,1/4}f(p)\bigr|\leq \frac{C_{1,s-3}C_{\text{pos}}}{\theta} \norm{f}_{W^{s}}T^{-1}\;,
\end{equation*}
so that, again by Lemma~\ref{lem:LPnearbyvariables},
\begin{equation*}
	d_{LP}(\mathbf{P}_{\theta,f}^{\text{circ}}(T),\mathbf{P}_{\theta,f})\leq \frac{C_{1,s-3}C_{\text{pos}}}{\theta}\norm{f}_{W^{s}}T^{-1}
\end{equation*} 
for any $T\geq 1$, as desired.

\medskip
Finally, for $\mu_f>1/4$, we have from~\eqref{eq:asymptoticgeneral} that
\begin{equation}
	\label{eq:DLTabovequarter}
	\begin{split}
		&e^{\frac{T}{2}}d_{f}(T,p)-\eps_0D^{+}_{\theta,1/4}f(p)-\sum_{\mu\in \text{Spec}(\square),\;\mu>1/4}\cos{\biggl(\frac{\Im{\nu}}{2}T\biggr)} D^{+}_{\theta,\mu}f(p)
		+\sin{\biggl(\frac{\Im{\nu}}{2}T\biggr)}D^{-}_{\theta,\mu}f(p)\\\
		&= \sum_{\mu\in \text{Spec}(\square),\;0<\mu<1/4}e^{-\frac{\nu}{2}T}D^{+}_{\theta,\mu}f(p)+e^{\frac{T}{2}}\mathcal{R}_{\theta}f(p,T)
	\end{split}
\end{equation} 
for any $p \in M$ and $T\geq 1$, from which we deduce what follows. Let $\mu_*=\inf \bigl(\text{Spec}(\square)\cap \R_{>0}\bigr)$ be the spectral gap of $S=\Ga\bsl \Hyp$ and $\nu_*$  the corresponding parameter:
\begin{itemize}
	\item if $\mu_{*}< 1/4$, then~\eqref{eq:DLTabovequarter} and~\eqref{eq:boundDthetamu} give
	\begin{equation*}
		\begin{split}
			&\biggl|e^{\frac{T}{2}}d_{f}(T,p)-\eps_0D^{+}_{\theta,1/4}f(p)-\sum_{\mu\in \text{Spec}(\square),\;\mu>1/4}\cos{\biggl(\frac{\Im{\nu}}{2}T\biggr)} D^{+}_{\theta,\mu}f(p)
			+\sin{\biggl(\frac{\Im{\nu}}{2}T\biggr)}D^{-}_{\theta,\mu}f(p)\biggr|\\
			&\leq \frac{C_{1,s-3}C_{\text{Spec}}'}{\theta}\norm{f}_{W^{s}}e^{-\frac{\nu_*}{2}T}\;,
		\end{split}
	\end{equation*}
	whence
	\begin{equation*}
		d_{LP}(\mathbf{P}_{\theta,f}^{\text{circ}}(T),\mathbf{P}_{\theta,f}(T))\leq \frac{C_{1,s-3}C_{\text{Spec}}'}{\theta}\norm{f}_{W^{s}}e^{-\frac{\nu_*}{2}T}
	\end{equation*}
	for any $T\geq 1$.
	\item if $\mu_*\geq 1/4$, then~\eqref{eq:DLTabovequarter} and~\eqref{eq:globalremainderestimate} give
	\begin{equation*}
		\begin{split}
			&\biggl|e^{\frac{T}{2}}d_{f}(T,p)-\eps_0D^{+}_{\theta,1/4}f(p)-\sum_{\mu\in \text{Spec}(\square),\;\mu>1/4}\cos{\biggl(\frac{\Im{\nu}}{2}T\biggr)} D^{+}_{\theta,\mu}f(p)
			+\sin{\biggl(\frac{\Im{\nu}}{2}T\biggr)}D^{-}_{\theta,\mu}f(p)\biggr|\\
			&\leq \frac{C_{1,s-3}C_{\text{Spec}}}{\theta}\norm{f}_{W^{s}}(T+1)e^{-\frac{T}{2}}\;;
		\end{split}
	\end{equation*}
	we deduce that
	\begin{equation*}
		d_{LP}(\mathbf{P}_{\theta,f}^{\text{circ}}(T),\mathbf{P}_{\theta,f}(T))\leq\frac{C_{1,s-3}C_{\text{Spec}}}{\theta}\norm{f}_{W^{s}}(T+1)e^{-\frac{T}{2}}
	\end{equation*} 
	for any $T\geq 1$.
\end{itemize}

This completes the proof of Proposition~\ref{prop:CLT}.

\subsection{Failure of a distributional limit theorem}
\label{sec:noCLT}
We now turn to the proof of Theorem~\ref{thm:noCLT}. Once again, we consider a fixed length parameter $\theta\in (0,4\pi]$ and a function $f\in W^{s}(M)$ for some real $s>11/2$. This time, we suppose that the coefficients 
$D^{\pm}_{\theta,\mu}f$ vanish identically on $M$ for any Casimir eigenvalue $\mu>0$. As a result, the asymptotic expansion provided in~\eqref{eq:fullexpansion} reduces to
\begin{equation}
	\label{eq:expansionCLT}
	\begin{split}
		\frac{1}{\theta}\int_{0}^{\theta}f\circ \phi_T^{X}\circ r_s(p)\;\text{d}s=&\int_{M}f\;\text{d}\vol +e^{-T}\int_1^{T}\sum_{n\in I(0)}-G_{\theta,n}f_{0,n}(p,\xi)\;\text{d}\xi
		\\
		&+\sum_{n\in I(0)}\mathcal{R}_{\theta,0,n}f_{0,n}(p,T)+\mathcal{R}_{\theta,\text{d}}f(p,T)\;,
	\end{split}
\end{equation}
for any $p \in M$ and $T\geq 1$, where $\mathcal{R}_{\theta,\text{d}}$ is defined in~\eqref{eq:remainderneg}. 

The estimates carried out in Section~\ref{sec:proofexpandingtranslates} lead to the bound
\begin{equation}
	\label{eq:remainderCLT}
	\biggl|\sum_{n\in I(0)}\mathcal{R}_{\theta,0,n}f_{0,n}(p,T)+\mathcal{R}_{\theta,\text{d}}f(p,T)\biggr|\leq \frac{(8e\pi+\kappa_0)C_{1,s-3}C_{\text{Spec},3}\sup\{1,C_{\text{disc}}\}}{\theta}\norm{f}_{W^{s}}e^{-T}\;.
\end{equation}
On the other hand, by means of~\eqref{eq:constantterm} we expand 
\begin{equation*}
	\begin{split}
		\sum_{n\in I(0)}-G_{\theta,n}f_{0,n}(p,\xi)=&\frac{2}{\theta(1-e^{-2\xi})}\sum_{n\in I(0)}(Uf_{0,n}\circ \phi_{\xi}^{X}(p)-Uf_{0,n}\circ \phi^{X}_{\xi}\circ r_{\theta}(p))\\
		&+\frac{e^{-2\xi}}{\theta(1-e^{-2\xi})^{2}}\sum_{n\in I(0)}2in(f_{0,n}\circ \phi_{\xi}^{X}(p)-f_{0,n}\circ \phi^{X}_{\xi}\circ r_{\theta}(p))\\
		&-\frac{e^{-\xi}}{\theta(1-e^{-2\xi})^{2}}\sum_{n\in I(0)}\int_0^{\theta}f_{0,n}\circ \phi^{X}_{\xi}\circ r_{s}(p)\;\text{d}s\\
		&+\frac{2e^{-\xi}}{\theta(1-e^{-2\xi})}\sum_{n\in I(0)}\int_0^{\theta}Xf_{0,n}\circ \phi^{X}_{\xi}\circ r_s(p)\;\text{d}s\;,
	\end{split}
\end{equation*}
from which
\begin{equation*}
	\sum_{n\in I(0)}-G_{\theta,n}f_{0,n}(p,\xi)=\frac{2}{\theta(1-e^{-2\xi})}(Uf_0\circ \phi^{X}_{\xi}(p)-Uf_0\circ \phi^{X}_{\xi}\circ r_{\theta}(p))+\mathcal{R}_{G}f(p,\xi)\;,
\end{equation*}
where 
\begin{equation}
	\label{eq:remaindersecondCLT}
	\begin{split}
		|\mathcal{R}_{G}f(p,\xi)|&\leq \frac{1}{(1-e^{-2})^{2}}\biggl(\norm{f}_{W^{s}}e^{-2\xi}+\norm{f}_{\infty}e^{-\xi}+2\norm{f}_{\cC^{1}}e^{-\xi}(1-e^{-2\xi})\biggr)\\
		&\leq \frac{1+C_{0,s}+2C_{1,s}}{(1-e^{-2})^{2}}\norm{f}_{W^{s}}e^{-\xi}\;,
	\end{split}
\end{equation}
using the bound $1-e^{-2\xi}\geq 1-e^{-2}$ valid for any $\xi\geq 1$.

Let now $(B_T)_{T>0}$ be a collection of positive real numbers such that $B_T\to\infty$ as $T\to\infty$. In light of~\eqref{eq:expansionCLT},~\eqref{eq:remainderCLT} and~\eqref{eq:remaindersecondCLT}, and because of the assumption on $(B_T)_{T>0}$, the distributional limits of the random variables
\begin{equation*}
	\frac{e^{T}\bigl( \frac{1}{\theta}\int_0^{\theta}f\circ\phi^{X}_{T}\circ r_{s}(p)\;\text{d}s-\int_{M}f\;\text{d}\vol\bigr)}{B_T}\;,\quad p\sim \vol,
\end{equation*}
as $T$ tends to infinity coincide with the distributional limits of the random variables
\begin{equation}
	\label{eq:sameCLT} \frac{\frac{2}{\theta}\int_1^{T}\frac{1}{1-e^{-2\xi}}(Uf_0\circ \phi^{X}_{\xi}(p)-Uf_0\circ \phi^{X}_{\xi}\circ r_{\theta}(p))\;\text{d}\xi}{B_T}\;,\quad p\sim \vol\;.
\end{equation}
When $\theta=4\pi$, we have $r_{4\pi}(p)=p$, so that the integrand in the numerator of the above expression vanishes. Therefore, the distributional limit we are seeking after equals to zero almost surely, which proves Theorem~\ref{thm:noCLT}.

\medskip
As to Remark~\ref{rmk:geodesiccoboundary}, suppose $\theta\in (0,4\pi]$ is arbitrary, and that $Uf_0$ is a coboundary for $(\phi^{X}_t)_{t\in \R}$, namely there exists a measurable function $g\colon M\to \C$ with\footnote{More accurately, this is the notion of a measurable coboundary; by the celebrated work of Livsic on the cohomological equation for Anosov flows (cf.~\cite{Livsic}), the condition is actually equivalent to the seemingly more restrictive one of $Uf_0$ being a continuous coboundary, namely of requiring the transfer function $g$ to be continuous.} finite norm
\begin{equation*}
	\norm{g}_{L^{\infty}(M,\vol)}=\inf\{\lambda \in \R_{>0}:|g(p)|\leq \lambda \text{ for $\vol$-almost every }p \in M  \}
\end{equation*}

such that, for all $T>0$, 
\begin{equation*}
	\int_0^{T}Uf_0\circ \phi^{X}_{\xi}(p)\;\text{d}\xi=g\circ \phi^{X}_T(p)-g(p) \quad \text{for $\vol$-almost every }p \in M\;.
\end{equation*}

It follows trivially that, for every $T>0$,
\begin{equation*}
	\biggl|\int_1^{T}\frac{1}{1-e^{-2\xi}}(Uf_0\circ \phi^{X}_{\xi}(p)-Uf_0\circ \phi^{X}_{\xi}\circ r_{\theta}(p))\;\text{d}\xi\biggr|\leq \frac{4}{1-e^{-2}}\norm{g}_{L^{\infty}(M,\vol)} 
\end{equation*}
for $\vol$-almost every $p \in M$.
As a result, the distributional limit as $T\to\infty$ of the random variables in~\eqref{eq:sameCLT} vanishes almost surely, since $B_T\to\infty$.

\begin{rmk}
	Slightly more generally, when $Uf_0$ is cohomologous to a constant function, namely it differs from a constant function by a coboundary, any distributional limit of the random variables in~\eqref{eq:sameCLT} is almost surely constant.
	
	Assume now $Uf_0$ is not cohomologous to a constant function (and $\theta\neq 4\pi$). The classical central limit theorem for geodesic ergodic integrals (see~\cite{Sinai} for the constant curvature case, and~\cite{Ratner} for variable negative curvature) gives that both
	\begin{equation*}
		\frac{\int_1^{T}Uf_0\circ \phi^{X}_{\xi}(p)\;\text{d}\xi-\int_{M}Uf_0\;\text{d}\vol}{\sqrt{T}}\;,\quad p\sim \vol
	\end{equation*}
	and 
	\begin{equation*}
		\frac{\int_1^{T}Uf_0\circ \phi^{X}_{\xi}\circ r_{\theta}(p)\;\text{d}\xi-\int_{M}Uf_0\;\text{d}\vol}{\sqrt{T}}\;,\quad p\sim \vol
	\end{equation*}
	converge in distribution to a non-trivial centered Gaussian random variable as $T$ tends to infinity. \textit{A priori}, the combination of these two distributional convergences doesn't provide any information on the distributional limits of the difference, which is what appears in~\eqref{eq:sameCLT} up to the constant factor $2/\theta$; it would be desirable to reach a full understanding of this limiting distributional behaviour by carefully inspecting the dependence properties of the random variables $\int_1^{T}Uf_0\circ \phi^{X}_\xi(p)\;\text{d}\xi$ and $\int_1^{T}Uf_0\circ \phi^{X}_\xi\circ r_{\theta}(p)\;\text{d}\xi$ as $p$ is sampled according to the volume measure on $M$. 
\end{rmk}

\subsection{Some reflections on temporal distributional limit theorems}
\label{sec:temporalDLT}
An upshot of the two foregoing subsections is the following consideration: examining the statistical behaviour, for large times $T$, of the (appropriately renormalized) averages 
\begin{equation*}
	\frac{1}{\theta}\int_{0}^{\theta}f\circ \phi^{X}_{T}\circ r_s(p)\;\text{d}s
\end{equation*}
by randomly sampling the base point $p$ according to the uniform measure on $M$ leads to meaningful asymptotic results if and only if\footnote{Possibly with the exception of the case examined at the end of Section~\ref{sec:noCLT}.} at least one of the coefficients $D^{\pm}_{\theta,\mu}f$ does not vanish identically on $M$. Irrespective of whether this is the case or not, it is natural to look for different sources of randomness, which might capture oscillatory behaviours more accurately. In accordance with the perspective of temporal distributional limit theorems, pioneered by Dolgopyat and Sarig~\cite{Dolgopyat-Sarig} in the context of ergodic sums and integrals, we enquire about the existence of non-trivial distributional limits for the random variables
\begin{equation*}
	\frac{e^{t}\bigl(\frac{1}{\theta}\int_0^{\theta}f\circ \phi^{X}_t\circ r_s(p)\;\text{d}s\bigr)-A_T}{B_T}\;,
\end{equation*}
where $p$ is a fixed base point in $M$, $(A_T)_{T>0}$ and $(B_T)_{T>0}$ are collections of real numbers, possibly depending on $p$, with $B_T>0$ and $B_T\to\infty$ as $T\to\infty$, and the time $t$ is chosen uniformly at random in the interval $[0,T]$. 

\begin{rmk}
	It is informative to compare this to the quest for temporal limit theorems for ergodic integrals along the orbits of a flow: see, in particular,~\cite[Def.~1.3]{Dolgopyat-Sarig}. Observe notably that the rescaling of the circle-arc average $\frac{1}{\theta}\int_0^{\theta}f\circ \phi_t^{X}\circ r_s(p)\;\text{d}s$ by a factor of $e^{t}$ (the latter being asymptotically of the same order of the length of the expanding circle arc along which the average is taken) parallels the renormalization of ergodic averages by the linear factor $t$.
\end{rmk}

Let us denote by $\mathcal{U}_{[0,T]}$ the uniform probability measure on the compact interval $[0,T]$, for any $T>0$. If there is a non-identically vanishing coefficient $D^{\pm}_{\theta,\mu}f$ for some Casimir eigenvalue $\mu>0$, then a rather straightforward adaptation of the proof of~\cite[Cor.~5.7]{Dolgopyat-Sarig} shows that, for $\vol$-almost every $p \in M$, any limiting distribution of 
\begin{equation}
	\label{eq:temporalDLT}
	\frac{e^{t}\bigl(\frac{1}{\theta}\int_0^{\theta}f\circ \phi^{X}_t\circ r_s(p)\;\text{d}s\bigr)-A_T}{B_T}\;,\quad t\sim \mathcal{U}_{[0,T]}
\end{equation}
is necessarily constant almost surely, no matter the choice of the constants $A_T$ and $B_T$.

Suppose now that the coefficients $D^{\pm}_{\theta,\mu}f$ vanish identically on $M$ for any  positive Casimir eigenvalue $\mu$. The deduction in Section~\ref{sec:noCLT} applies almost verbatim, showing  that the distributional limits of the random variables in~\eqref{eq:temporalDLT} are the same as the limits of 
\begin{equation}
	\label{eq:sametemporal}
	\frac{\frac{2}{\theta}\int_1^{t}\frac{1}{1-e^{-2\xi}}(Uf_0\circ \phi^{X}_{\xi}(p)-Uf_0\circ \phi^{X}_{\xi}\circ r_{\theta}(p))\;\text{d}\xi-A_T}{B_T}\;,\quad t\sim \mathcal{U}_{[0,T]}
\end{equation}
as $T$ tends to infinity. In the first place, this allows tu rule out the existence of any non-trivial (namely not almost surely constant) distributional limit whenever one of the following conditions is met:
\begin{itemize}
	\item[(a)] $\theta=4\pi$;
	\item[(b)] $Uf_0$ is cohomologous to a constant function for the geodesic flow.
\end{itemize} 
On the other hand, when $Uf_0$ is not cohomologous to a constant function, then the geodesic ergodic integrals $\int_1^{t}Uf_0\circ \phi^{X}_{\xi}(p)\;\text{d}\xi$ are well-approximated by Brownian trajectories. More precisely, the Almost Sure Invariance Principle (see~\cite{Strassen,Strassen-two},~\cite[Chap.~1]{Philipp-Stout} and~\cite{Denker-Phillip}) for geodesic ergodic integrals asserts that there exist an auxiliary probability space $(\Omega,\cF,\mathbf{P})$ and two continuous-time stochastic processes $(X_t)_{t\geq 0}$ and $(B_t)_{t\geq 0}$ defined on $(\Omega,\cF,\mathbf{P})$ such that the following hold:
\begin{itemize}
	\item the law of the process $(X_{t})_{t\geq 0}$ under the probability measure $\mathbf{P}$ coincides with the law of the process $\bigl(\int_0^{t}Uf_0\circ \phi^{X}_{\xi}(p)\;\text{d}\xi\bigr)_{t\geq 0}$ when $p$ is sampled according to the probability measure $\vol$;
	\item the process $(B_{t})_{t\geq 0}$ is a standard one-dimensional Brownian motion (cf.~\cite[Chap.~2]{LeGall});
	\item there exists $\sigma \in \R^{\times}$ such that, for $\mathbf{P}$-almost every $\omega \in \Omega$,
	\begin{equation}
		\label{eq:Brownian}
		|X_t(\omega)-B_{\sigma^{2}t}(\omega)|=o(\sqrt{t}) \quad \text{as $t\to\infty$}\;.
	\end{equation}
\end{itemize}
As typical Brownian trajectories are of size $\sim\sqrt{t}$ at time $t$, the approximation in~\eqref{eq:Brownian} enables to transfer classical results about the statistical behaviour of Brownian paths to analogous properties for geodesic ergodic integrals. In particular, there is no distributional limit\footnote{Actually, when $Uf_0$ has zero average over $M$, $A_T=0$ and $B_T=\sqrt{T}$,  any random variable may appear as distributional limit along an appropriate subsequence $(T_n)_{n\in \N}$ of times: see~\cite[Thm.~3.2]{Dolgopyat-Sarig}.} for 
\begin{equation}
	\label{eq:temporalfirst}
	\frac{\int_0^{t}Uf_0\circ\phi^{X}_{\xi}(p)\;\text{d}\xi-A_T}{B_T}\;,\quad t\sim\mathcal{U}_{[0,T]}
\end{equation} 
as $T$ tends to infinity (cf.~\cite[Sec.~3.1]{Dolgopyat-Sarig}). 
Since the process $\bigl(\int_0^{t}Uf_0\circ \phi^{X}_{\xi}\circ r_{\theta}(p)\;\text{d}\xi\bigr)_{t\geq 0}$ has the same law, for $p\sim \vol$, as $(X_t(\omega))_{t\geq 0}$ for $\omega\sim \mathbf{P}$, the same applies to the random variables 
\begin{equation}
	\label{eq:temporalsecond}
	\frac{\int_0^{t}Uf_0\circ\phi^{X}_{\xi}\circ r_{\theta}(p)\;\text{d}\xi-A_T}{B_T}\;,\quad t\sim\mathcal{U}_{[0,T]}\;.
\end{equation}
As already argued in Section~\ref{sec:noCLT} in the situation where the point $p$ is selected randomly and the time $T$ is fixed, here again the absence of distributional limits for each of the summands does not rule out, in principle, the possibility of non-trivial limits for the difference, hence for~\eqref{eq:sametemporal}. Once more, a painstaking analysis of the dependence features of the two processes in~\eqref{eq:temporalfirst} and~\eqref{eq:temporalsecond} might clarify the seemingly elusive pathwise behaviour of their difference.

\section{The hyperbolic lattice point counting problem}
\label{sec:latticepoint}

This final section is consecrated to the applications of our equidistribution results to lattice-point counting problems in the hyperbolic plane; specifically, we shall first prove the precise asymptotics for the averaged counting function stated in Proposition~\ref{prop:averagedcounting} and subsequently deduce Theorem~\ref{thm:countingproblem} on the error estimate for the pointwise counting.  

Let $\Ga$ be a cocompact lattice in $\SL_2(\R)$, and denote by $d_{\Hyp}$ the hyperbolic distance function on the hyperbolic upper-half plane $\Hyp$ (cf.~Section~\ref{sec:hyperbolic}). For each real number $R>0$, let $B_R$ be the closed $d_{\Hyp}$-ball of radius $R$ centered at the point $i\in \Hyp$, and define $N(R)=|\Ga\cdot i \;\cap B_R|$, the cardinality of intersection of the $\Ga$-orbit of $i$ with $B_R$.

Recall also from Section~\ref{sec:hyperbolic} that $\SL_2(\R)$ acts on $\Hyp$ by M\"{o}bius transformations.
In what follows, we identity the quotient manifold $\SL_2(\R)/\SO_2(\R)$ with $\Hyp$ whenever convenient, by means of the diffeomorphism $g\SO_2(\R)\mapsto g\cdot i,\;g\in \SL_2(\R)$. The hyperbolic area measure $m_{\Hyp}$ (namely the volume measure arising from the hyperbolic structure on $\Hyp$) is the Radon measure on $\Hyp$ with density $\text{d}m_{\Hyp}(x,y)=y^{-2}\text{d}x\text{d}y$ with respect to the induced Lebesgue measure on $\Hyp\subset \C$. 

\begin{term}
	In order not to overburden notation in the sequel, we shall denote $\SL_2(\R)$ by $G$ and $\SO_2(\R)$ by $K$.
\end{term}

\subsection{Asymptotics for the averaged counting function}
\label{sec:averagecounting}

For any subset $A\subset G/K$, we denote by $\mathds{1}_{A}$ the indicator function of the set $A$. Define a function $F_R\colon G/\Ga\to \R_{\geq 0}$
\begin{equation}
	\label{eq:defaveragedcounting}
	F_R(g\Ga )=\frac{|g\Gamma \cdot i \cap B_R|}{m_{\Hyp}(B_R)}=\frac{1}{m_{\Hyp}(B_R)}\sum_{\gamma \Ga\cap K\in \Ga/\Ga\cap K}\mathds{1}_{B_R}(g\gamma K)\;, \quad g\in G;
\end{equation}
observe that the function $F_R$ is the subject of the averaged counting result in Proposition~\ref{prop:averagedcounting}, which we now set out to prove. 

\begin{rmk}
	\label{rmk:leftrightcosets}
	We choose to deal with spaces of left cosets in the sequel; in particular, we replace the homogeneous spaces $M=\Ga\bsl G$ we have been considering so far with $G/\Ga$, identifying them via the diffeomorphism $\Ga g\mapsto g^{-1}\Ga$.
\end{rmk}

We follow the classical argument of Eskin and McMullen \cite{Eskin-McMullen}, which relies on the well-known folding-unfolding formula for invariant measures on homogeneous spaces. For the sake of completeness, we recall it in the setting of the group $G=\SL_2(\R)$, referring the reader to~\cite[Sec.~2.6]{Folland} or to~\cite[Chap.~1]{Raghunathan} for the general statements and their proofs. 

\begin{prop}
	\label{prop:foldingunfolding}
	\begin{enumerate}
		\item Let $H<G$ be a unimodular closed subgroup. Then there exists a non-zero $G$-invariant positive Radon measure on the quotient space $G/H $, which is uniquely determined up to positive real scalars. Moreover, if $\mu_{G}$ and $\mu_H$ are Haar measures\footnote{The group $G$ is perfect, hence unimodular; thus $\mu_{G}$ is also a right Haar measure.} on $G$ and $H$, respectively, there exists a unique normalization $\mu_{G/H}$ of the $G$-invariant measure on $G(H)$ such that, for any continuous compactly supported function $\varphi\colon G\to \C$, the following folding-unfolding formula holds:
		\begin{equation}
			\label{eq:foldingunfolding}
			\int_{G}\varphi\;\emph{d}\mu_{G}=\int_{G/H}\int_{H}\varphi(gh)\;\emph{d}\mu_{H}(h)\;\emph{d}\mu_{G/H}(gH)\;.
		\end{equation}
		\item Let $H<L$ be closed subgroups of $G$, and suppose $G/L$ admits a non-zero finite $G$-invariant measure $\mu_{G/L}$ and $L/H$ admits a non-zero finite $L$-invariant measure $\mu_{L/H}$. Then $G/H$ admits a non-zero finite $G$-invariant measure $\mu_{G/H}$. Moreover, if $\mu_{G/L},\;\mu_{L/H}$ and $\mu_{G/H}$ are compatibly normalized, then for any continuous compactly supported function $\varphi\colon G/H\to \C$, it holds
		\begin{equation}
			\label{eq:foldingunfoldingchain}
			\int_{G/H}\varphi\;\emph{d}\mu_{G/H}=\int_{G/L}\int_{L/H}\varphi(glH)\;\emph{d}\mu_{L/H}(lH)\;\emph{d}\mu_{G/L}(gL)\;.
		\end{equation} 
	\end{enumerate}
\end{prop}
By means of standard approximation arguments in measure theory, formula~\eqref{eq:foldingunfolding} (resp.~ formula~\eqref{eq:foldingunfoldingchain}) holds for any Borel-measurable function $\varphi\colon G\to \C$ (resp.~$\varphi\colon G/H\to \C$) which either takes positive real values or is integrable with respect to $\mu_{G}$ (resp.~$\mu_{G/H}$). 

\smallskip
Let now $m_{K}$ be the unique probability Haar measure on the compact group $K$, and normalize the Haar measure $m_{G}$ on $G$ so that, under the identification of $G/K$ with $\Hyp$, the resulting $G$-invariant measure on $G/K$ (cf.~Proposition~\ref{prop:foldingunfolding}) corresponds to the hyperbolic area measure $m_{\Hyp}$. Let $m_{G/\Ga}$ be the unique $G$-invariant finite Borel measure on $G$ corresponding, according to Proposition~\ref{prop:foldingunfolding}, to the given choice of Haar measure on $G$ and to the counting measure on the discrete group $\Ga$. Similarly, endowing the finite discrete group $\Ga\cap K$ with the counting measure, we indicate with $m_{G/\Ga\cap K}$, $m_{K/\Ga\cap K}$ and $m_{\Ga/\Ga\cap K}$ the induced measures on the respective homogeneous spaces. Recall also that with $m_{K\cdot \Ga}$ we indicate the unique $K$-invariant probability measure supported on the compact $K$-orbit of the identity coset $\Ga$ inside $G/\Ga$ (cf.~Theorem~\ref{thm:mainarbitrarytranslates}).

The volumes of the homogeneous spaces $G/\Ga$ and $K/\Ga\cap K$ with respect to the measures $m_{G/\Ga}$ and $m_{K/\Ga\cap K}$ are indicated with $\text{covol}_{K}(\Gamma\cap K)$ and $\text{covol}_{G}(\Ga)$, respectively.

\medskip
Fix now a real parameter $s>11/2$ and a test function $\psi\in W^{s}(G/\Ga)$. We expand, for any $R>0$,

\begin{equation}
	\label{eq:foldunfold}
	\begin{split}
		&\int_{M}\psi F_R\;\text{d}m_{G/\Ga}=\int_{M}\psi(g\Ga )\biggl(\frac{1}{m_{\Hyp}(B_R)}\sum_{\gamma \Ga\cap K\in \Ga/\Ga\cap K}\mathds{1}_{B_R}(g\gamma K)\biggr) \;\text{d}m_{G/\Ga}(g\Ga)\\
		&=\frac{1}{m_{\Hyp}(B_R)}\int_{M}\int_{\Ga/\Ga\cap K}\psi(g\Ga )\mathds{1}_{B_R}(g\gamma K)\;\text{d}m_{\Ga/\Ga\cap K}(\gamma\Gamma\cap K) \;\text{d}m_{G/\Ga}(g\Ga)\\
		&=\frac{1}{m_{\Hyp}(B_R)}\int_{G/\Gamma\cap K} \psi(g\Ga)\mathds{1}_{B_R}(gK) \text{d}m_{G/\Ga\cap K}(g\Gamma\cap K)\\
		&=\frac{1}{m_{\Hyp}(B_R)}\int_{G/K}\int_{K/\Ga\cap K}\psi(gk\Ga)\mathds{1}_{B_R}(gK)\;\text{d}m_{K/\Ga\cap K}(k\Ga\cap K)\;\text{d}m_{G/K}(gK)\\
		&=\frac{1}{m_{\Hyp}(B_R)}\int_{B_R}\int_{K/\Ga\cap K}\psi\;\text{d}g_{*}m_{K/\Ga\cap K}\;\text{d}m_{\Hyp}(gK)\\
		&=\frac{\text{covol}_{ K}(\Ga\cap K )}{m_{\Hyp}(B_R)}\int_{B_R}\int_{G/\Ga}\psi\;\text{d}g_{*}m_{K\cdot \Ga}\;\text{d}m_{\Hyp}(gK)\;.
	\end{split}
\end{equation}
In the previous chain of equalities, we applied in successive order:
\begin{itemize}
	\item[(1)] the definition~\eqref{eq:defaveragedcounting} of the function $F_R$;
	\item[(2)] the fact that the invariant measure on the discrete space $\Ga/\Ga\cap K$ given by Proposition~\ref{prop:foldingunfolding} is the counting measure;
	\item[(3)] formula~\eqref{eq:foldingunfoldingchain} to the tower of subgroups $\Ga\cap K<\Ga<G$;
	\item[(4)] formula~\eqref{eq:foldingunfoldingchain} to the tower of subgroups $\Ga\cap K<K<G$;
	\item[(5)] the identification of $\Hyp$ with $G/K$ and of $m_{\Hyp}$ with $m_{G/K}$;
	\item[(6)] the relationship $m_{K/\Ga\cap K}=\text{covol}_{K}(\Ga\cap K)m_{K\cdot \Ga}$, derived from the fact that $m_{K\cdot \Ga}$ is a probability measure, the definition of $\text{covol}_{K}(\Ga\cap K)$ and uniqueness up to scalars of the $K$-invariant measure on $K/\Ga \cap K\simeq K\cdot \Ga$.  
\end{itemize}

We may now replace the inner integral in the last expression of~\eqref{eq:foldunfold} with the asymptotic expansion provided by Theorem~\ref{thm:mainarbitrarytranslates} (with the caveat of Remark~\ref{rmk:leftrightcosets}), thereby obtaining
\begin{equation}
	\label{eq:asymptoticaverage}
	\begin{split}
		\int_{G/\Ga}\psi F_R&\;\text{d}m_{G/\Ga}=\frac{\text{covol}_{K}(\Ga\cap K)}{\text{covol}_{G}(\Ga)}\int_{G/\Ga}\psi\;\text{d}m_{G/\Ga} +\frac{\text{covol}_{K}(\Ga\cap K)}{m_{\Hyp}(B_R)}\\
		&\int_{B_R}\biggl(\sum_{\mu\in \text{Spec}(\square),\;\mu>1/4}\norm{g}^{-1}_{\text{op}}\bigl(\cos{(\Im{\nu}\log{\norm{g}_{\text{op}}})}D^{+}_{\mu}\psi(\Ga,g)+\sin{(\Im{\nu}\log{\norm{g}_{\text{op}}})}D^{-}_{\mu}\psi(\Ga,g)\bigr)\\
		&+\sum_{\mu\in \text{Spec}(\square),\;0<\mu<1/4}\norm{g}_{\text{op}}^{-(1+\nu)}D^{+}_{\mu}\psi(\Ga,g)+\norm{g}_{\text{op}}^{-(1-\nu)}D^{-}_{\mu}\psi(\Ga,g)\\
		&+\varepsilon_0\bigl(\norm{g}_{\text{op}}^{-1}D^{+}_{1/4}\psi(\Ga,g)+2\norm{g}_{\text{op}}^{-1}\log{\norm{g}_{\text{op}}}D^{-}_{1/4}\psi(\Ga,g)\bigr)+\mathcal{R}\psi(\Ga,g)\biggr) \text{d}m_{\Hyp}(gK)\;.
	\end{split}
\end{equation}

We shall need the following analogue of the classical integration formula on spheres in Euclidean spaces: for any $r>0$, let $S_r=\partial{B_r}=\{z\in \Hyp:d_{\Hyp}(z,i)=r \}$ and $\sigma_r$ the induced hyperbolic length measure on the circle $S_r$.

\begin{prop}
	\label{prop:sphereintegration}
	Let $f\colon \Hyp\to \C$ be integrable with respect to $m_{\Hyp}$. Then
	\begin{equation*}
		\int_{\Hyp}f\;\emph{d}m_{\Hyp}=\int_{0}^{\infty}\int_{S_r}f(z)\;\emph{d}\sigma_{r}(z)\;\emph{d}r
	\end{equation*}
\end{prop} 

The proof does not differ from the Euclidean case, for which we refer to~\cite[Thm.~2.49]{Folland}.

\medskip
Define now, for any $\mu \in \text{Spec}(\square)\cap \R_{>0}$ and $\psi$ as above,
\begin{equation*}
	\alpha^{\pm}_{\psi,\mu}(r)=\int_{S_r}D^{\pm}_{\mu}\psi(\Ga,z)\;\text{d}\sigma_r(z)\;,\quad r>0.
\end{equation*}
From~\eqref{eq:asymptoticaverage} we get, thanks to Proposition~\ref{prop:sphereintegration}, 
\begin{equation}
	\label{eq:applyingsphereintegration}
	\begin{split}
		&\int_{G/\Ga}\psi F_R\;\text{d}m_{G/\Ga}=\frac{\text{covol}_{K}(\Ga\cap K)}{\text{covol}_{G}(\Ga)}\int_{G/\Ga}\psi\;\text{d}m_{G/\Ga}+\frac{\text{covol}_{K}(\Ga\cap K)}{m_{\Hyp}(B_R)}\\
		&\biggl(\sum_{\mu\in \text{Spec}(\square),\;\mu>1/4}\int_0^{R}e^{-\frac{r}{2}}\biggl(\cos{\biggl(\frac{\Im{\nu}}{2}r\biggr)}\alpha_{\psi,\mu}^{+}(r)+\sin{\biggl(\frac{\Im{\nu}}{2}r\biggr)}\alpha^{-}_{\psi,\mu}(r)\biggr)\text{d}r\\
		&+\sum_{\mu\in \text{Spec}(\square),\;0<\mu<1/4}\int_0^{R}e^{-\frac{1+\nu}{2}r}\alpha^{+}_{\psi,\mu}(r)+e^{-\frac{1-\nu}{2}r}\alpha^{-}_{\psi,\mu}(r)\;\text{d}r\\
		&+\varepsilon_0\biggl(\int_0^{R}e^{-\frac{r}{2}}\alpha^{+}_{\psi,1/4}(r)+re^{-\frac{r}{2}}\alpha^{-}_{\psi,1/4}(r)\;\text{d}r\biggr)+\int_0^{R}\int_{S_r}\mathcal{R}\psi(\Ga,z)\;\text{d}\sigma_r(z)\;\text{d}r\biggr)\;.
	\end{split}
\end{equation}
Let $\psi_{\mu}$ be the orthogonal projection of $\psi$ onto the closed subspace $W^{s}(H_{\mu})$. By means of~\eqref{eq:boundDthetamu}, we estimate 
\begin{equation}
	\label{eq:integralbounds}
	\sum_{\mu \in \text{Spec}(\square)\cap \R_{>0}}|\alpha_{\psi,\mu}^{\pm}(r)|\leq \sum_{\mu\in \text{Spec}(\square)\cap \R_{>0}}\norm{D^{\pm}_{\mu}\psi(\Ga,\cdot )}_{\infty}\int_{S_r}\text{d}\sigma_{r}(z)\leq 2\pi \frac{C_{1,s-3}C_{\text{Spec}}'}{4\pi}\norm{\psi}_{W^{s}}\sinh{r}
\end{equation}
for any $r>0$,
as the hyperbolic length  of $S_r$ equals\footnote{This is an elementary verification in hyperbolic geometry, for instance approximating circles with regular $n$-gons; their hyperbolic perimeter can be easily computed by means of explicit formulas for the hyperbolic distance (cf.~\cite[Thm.~1.2.6]{Katok}) and of the hyperbolic cosine law (cf.~\cite[Thm.~1.5.2]{Katok}). 
	
	Similarly, the hyperbolic area of a ball is easily computed by approximation via the Gauss-Bonnet formula for the area of hyperbolic triangles (cf.~\cite[Thm.~1.4.2]{Katok}).} $2\pi\sinh{r}$. Recalling that $m_{\Hyp}(B_R)=2\pi(\cosh{R}-1)$ for any $R>0$, we deduce from~\eqref{eq:applyingsphereintegration} that
\begin{equation*}
	\begin{split}
		\int_{G/\Ga}\psi F_R\;\text{d}m_{G/\Ga}=&\frac{\text{covol}_{K}(\Ga\cap K)}{\text{covol}_{G}(\Ga)}\int_{G/\Ga}\psi\;\text{d}m_{G/\Ga}\\
		&+\text{covol}_{K}(\Ga\cap K)
		\biggl(e^{-\frac{R}{2}} \sum_{\mu\in \text{Spec}(\square),\;\mu>1/4}\beta^{+}_{\psi,\mu}(R)+\beta^{-}_{\psi,\mu}(R)\\
		&
		+\sum_{\mu\in \text{Spec}(\square),\;0<\mu<1/4}e^{-\frac{1+\nu}{2}R}\beta^{+}_{\psi,\mu}(R)+e^{-\frac{1-\nu}{2}R}\beta^{-}_{\psi,\mu}(R)\\
		&+\varepsilon_0\bigl(e^{-\frac{R}{2}}\beta^{+}_{\psi,1/4}(R)+Re^{-\frac{R}{2}}\beta^{-}_{\psi,1/4}(R)\bigr)+\gamma_{\psi}(R)\biggr)
	\end{split}
\end{equation*}
for any $R\geq 1$, where we have set
\begin{align*}
	&\beta^+_{\psi,\mu}(R)=\frac{e^{-\frac{R}{2}}}{\pi(1-2e^{-R}+e^{-2R})}\int_0^{R}e^{-\frac{r}{2}}\cos{\biggl(\frac{\Im{\nu}}{2}r\biggr)}\alpha_{\psi,\mu}^{+}(r)\;\text{d}r\;,\quad \mu>1/4\;,\\
	&\beta^-_{\psi,\mu}(R)=\frac{e^{-\frac{R}{2}}}{\pi(1-2e^{-R}+e^{-2R})}\int_0^{R}e^{-\frac{r}{2}}\sin{\biggl(\frac{\Im{\nu}}{2}r\biggr)}\alpha^{-}_{\psi,\mu}(r)\text{d}r\;,\quad\mu>1/4\;,\\
	&\beta^{\pm}_{\psi,\mu}(R)=\frac{e^{-\frac{1\mp\nu}{2}R}}{\pi(1-2e^{-R}+e^{-2R})}\int_0^{R}e^{-\frac{1\pm\nu}{2}r}\alpha^{\pm}_{\psi,\mu}(r)\;\text{d}r\;,\quad 0<\mu<1/4\;,\\
	&\beta^{+}_{\psi,1/4}(R)=\frac{e^{-\frac{R}{2}}}{\pi(1-2e^{-R}+e^{-2R})}\int_0^{R}e^{-\frac{r}{2}}\alpha^{+}_{\psi,1/4}(r)\;\text{d}r\;,\\
	&\beta^{-}_{\psi,1/4}(R)=\frac{R^{-1}e^{-\frac{R}{2}}}{\pi(1-2e^{-R}+e^{-2R})}\int_0^{R}e^{-\frac{r}{2}}\alpha^{-}_{\psi,1/4}(r)\;\text{d}r\;,\\
	&\gamma_{\psi}(R)=\frac{e^{-R}}{\pi(1-2e^{-R}+e^{-2R})}\int_0^{R}\int_{S_r}\mathcal{R}\psi(\Ga,z)\;\text{d}\sigma_r(z)\;\text{d}r
\end{align*}
. Because of~\eqref{eq:integralbounds}, we have the following estimates on the previous coefficients: for any $R\geq 1$,
\begin{align*}
	&\sum_{\mu\in \text{Spec}(\square)\cap \R_{>0}}|\beta^{\pm}_{\psi,\mu}(R)|\leq \frac{5C_{1,s-3}C'_{\text{Spec}}}{2\pi}\norm{\psi}_{W^{s}}\;,\\
	& |\gamma_{\psi}(R)|\leq \frac{5C_{1,s-3} C_{\text{Spec}}}{4\pi} \norm{\psi}_{W^{s}}(R+1)e^{-R}\;,
\end{align*}
using  the (crude) bound $(1-2e^{-R}+e^{-2R})^{-1}\leq 5$ in each of the previous inequalities.

This establishes Proposition~\eqref{prop:averagedcounting} in its entirety.

\subsection{Error estimate for the pointwise counting problem}
\label{sec:countingproblem}

This subsection is devoted to the deduction of the estimate on the error for the counting problem stated in Theorem~\ref{thm:countingproblem}, starting from the asymptotic expansion in~\eqref{eq:averagedcountingfunction} for the averaged counting function. 

Recall from~\eqref{eq:defaveragedcounting} that, for any real number $R>0$, the ratio $N(R)/m_{\mathbb{H}}(B_R)$ equals the value of the function $F_R$ at the identity coset $\Ga\in G/\Ga$. In order to find a convenient approximation for the latter, we shall compare it with the averages 
\begin{equation*}
	\int_{G/\Ga}\psi F_R\;\text{d}m_{G/\Ga}
\end{equation*}  
where the function $\psi$ ranges over a suitably defined approximate identity\footnote{The terminology is common in the context of locally compact groups; see, for instance, \cite[Sec.~2.5]{Folland}.} in $G/\Ga$. 

We now expose the details. Let us fix a parameter $\delta\in \R_{>0}$, on which we shall subsequently impose conditions according to the needs of the argument; choose
\begin{itemize}
	\item[(a)] an open symmetric\footnote{Namely, $U_{\delta}$ coincides with the set of inverses of its elements.} neighborhood $U_{\delta}$ of the identity in $G$ such that, for any $R>0$, 
	\begin{equation}
		\label{eq:neighborhoodcontainment}
		B_{R-\delta}\subset\bigcap_{g\in U_{\delta}}g\cdot B_R\subset \bigcup_{g\in U_{\delta}}g\cdot B_R \subset B_{R+\delta}
	\end{equation}
	\item[(b)] and a smooth function $\psi_{\delta}\colon G/\Ga\to \R_{\geq 0}$ with compact support contained in the open set  $U_{\delta}\Ga=\{g\Ga:g\in U_{\delta} \}$ and satisfying
	\begin{equation}
		\label{eq:integralone}
		\int_{G/\Ga}\psi_{\delta}\;\text{d}m_{G/\Ga}=1\;.
	\end{equation}
\end{itemize}

\begin{rmk}
	The existence, for any $\delta>0$, of a neighborhood $U_{\delta}$ with the properties claimed above is routinely referred to in the literature (see, for instance,~\cite{Eskin-McMullen}) as the \emph{well-roundedness} property of the collection of balls $(B_R)_{R>0}$. A geometric condition of this sort affords to leverage equidistribution results to study lattice point counting problems.
\end{rmk}

Observe that we may harmlessly replace $U_{\delta}$ with
\begin{equation*} KU_{\delta}=\bigcup_{k\in K}kU_{\delta}\;,
\end{equation*}
and thus assume that $U_{\delta}$ is saturated with respect to left translations by elements of $K$. Property~\eqref{eq:neighborhoodcontainment} is unaffected: for any $k\in K$ and $z\in \Hyp$, we have 
\begin{equation*}
	d_{\Hyp}(k\cdot z,i)=d_{\Hyp}(k\cdot z,k\cdot i)=d_{\Hyp}(z,i)\;,
\end{equation*}
as the subgroup $K$ fixes $i$ and acts by hyperbolic isometries; therefore  $k\cdot B_{r}=B_{r}$ for any $k\in K$ and any $r>0$. As a consequence of this, we might and shall assume that $\psi_{\delta}$ is $K$-invariant.

We now express, for any $R>0$, the ratio $N(R)/m_{\Hyp}(B_R)$ as
\begin{equation*}
	\begin{split}
		F_R(\Ga)&=F_R(\Ga)-\int_{G/\Ga}\psi_{\delta}F_R\;\text{d}m_{G/\Ga}+\int_{G/\Ga}\psi_{\delta}F_R\;\text{d}m_{G/\Ga}\\
		&=\int_{G/\Ga}\psi_{\delta}(g\Ga)(F_R(\Ga)-F_R(g\Ga))\;\text{d}m_{G/\Ga}(g\Ga)+\int_{G/\Ga}\psi_{\delta}F_R\;\text{d}m_{G/\Ga}\;,
	\end{split}
\end{equation*}
where the second inequality follows from the property in~\eqref{eq:integralone}. Let us call $\mathscr{E}_{\delta}(R)$, for notational simplicity, the quantity
\begin{equation*}
	\int_{G/\Ga}\psi_{\delta}(g\Ga)(F_R(\Ga)-F_R(g\Ga))\;\text{d}m_{G/\Ga}(g\Ga)\;;
\end{equation*}
in view of~\eqref{eq:averagedcountingfunction} applied to $\int_{G/\Ga}\psi_{\delta}F_R\;\text{d}m_{G/\Ga}$, we may write 
\begin{equation}
	\label{eq:pointwiseaverage}
	\begin{split}
		F_R(\Ga)=&\frac{\text{covol}_{K}(\Ga\cap K)}{\text{covol}_{G}(\Ga)}+\mathscr{E}_{\delta}(R)\\
		&+\text{covol}_{K}(\Ga\cap K)
		\biggl( e^{-\frac{R}{2}} \sum_{\mu\in \text{Spec}(\square),\;\mu>1/4}\beta^+_{\psi_{\delta},\mu}(R)+\beta^-_{\psi_{\delta},\mu}(R)\\
		&+\sum_{\mu\in \text{Spec}(\square),\;0<\mu<1/4}e^{-\frac{1+\nu}{2}R}\beta^{+}_{\psi_{\delta},\mu}(R)+e^{-\frac{1-\nu}{2}R}\beta^{-}_{\psi_{\delta},\mu}(R)\\
		& +\varepsilon_0\biggl(e^{-\frac{R}{2}}\beta^{+}_{\psi_{\delta},1/4}(R)+Re^{-\frac{R}{2}}\beta^{-}_{\psi_{\delta},1/4}(R)\biggr)+\gamma_{\psi_{\delta}}(R)\biggr)\;.
	\end{split}
\end{equation}
We estimate, for any $R>0$,
\begin{equation}
	\label{eq:Edelta}
	|\mathscr{E}_{\delta}(R)|\leq \int_{G/\Ga}\psi_{\delta}(g\Ga)|F_R(\Ga)-F_R(g\Ga)|\;\text{d}m_{G/\Ga}(g\Ga)
	\leq \sup_{g\in U_{\delta}}|F_R(\Ga)-F_R(g\Ga)|\;,
\end{equation}
the last inequality being a consequence of~\eqref{eq:integralone} and the fact that $\text{supp}\;\psi_{\delta}\subset U_{\delta}\Ga$.

Now, for any $g\in U_{\delta}$, we have
\begin{equation}
	\label{eq:FR}
	\begin{split}
		|F_R(\Ga)-F_R(g\Ga)|&=\frac{\bigl||\Ga\cdot H\cap B_R|-|g\Ga\cdot H\cap B_R|\bigr|}{m_{\mathbb{H}}(B_R)}=\frac{\bigl||\Ga\cdot H\cap B_R|-|\Ga\cdot H\cap g^{-1}\cdot B_R|\bigr|}{m_{\mathbb{H}}(B_R)}\\
		&\leq \frac{\bigl|\Ga\cdot H\cap \bigl(\bigl(\bigcup_{g\in U_{\delta}}g\cdot B_R\bigr)\setminus \bigl(\bigcap_{g\in U_{\delta}}g\cdot B_{R}\bigr)\bigr)\bigr|}{m_{\mathbb{H}}(B_R)}\leq\frac{ N(R+\delta)-N(R-\delta)}{m_{\mathbb{H}}(B_R)}\\
		&=F_{R+\delta}(\Ga)\frac{m_{\mathbb{H}}(B_{R+\delta})}{m_{\mathbb{H}}(B_R)}-F_{R-\delta}(\Ga)\frac{m_{\mathbb{H}}(B_{R-\delta})}{m_{\mathbb{H}}(B_R)}\;,
	\end{split}
\end{equation}
where the second-to-last inequality follows from~\eqref{eq:neighborhoodcontainment}. Choose $R_0=R_0(\Ga)>0$ such that the quantities
\begin{equation*}
	M=\sup_{r\geq R_0}F_{r}(\Ga) \quad \text{ and }\quad m=\inf_{r\geq R_0}F_r(\Ga)
\end{equation*}
are non-zero and finite\footnote{A straightforward modification of the effective argument we are running leads to the well-known non-effective convergence 
	\begin{equation*}
		F_{R}(\Ga)\overset{R\to\infty}{\longrightarrow}\frac{\text{covol}_{K}(\Ga\cap K)}{\text{covol}_{G}(\Ga)}\in \R_{>0}\;.
	\end{equation*}
}. Plugging~\eqref{eq:FR} into~\eqref{eq:Edelta}, we get that, for any $R\geq 2R_0$ and $\delta<R_0$, 
\begin{equation}
	\label{eq:neighboringvalues}
	\begin{split}
		&|\mathscr{E}_{\delta}(R)|\leq (M-m)\biggl(\frac{m_{\mathbb{H}}(B_{R+\delta})}{m_{\mathbb{H}}(B_R)}-\frac{m_{\mathbb{H}}(B_{R-\delta})}{m_{\mathbb{H}}(B_R)}\biggr)\\
		&=\frac{(M-m)}{1-2e^{-R}+e^{-2R}}\biggl(e^{\delta}(1-2e^{-(R+\delta)}+e^{-2(R+\delta)})-e^{-\delta}(1-2e^{-(R-\delta)}+e^{-2(R-\delta)})\biggr)\\
		&\leq(M-m)(ce^{\delta}-e^{-\delta})\;,
	\end{split}
\end{equation}
where we might for example take $c=c_{\Ga}=\frac{1+e^{-2R_0}}{1-2e^{-R_0}}$.

We now let the parameter $\delta$ be a function of the radius $R$; for reasons which we will shortly elucidate (see Remark~\ref{rmk:exponentialdecay}), we let $\delta=\delta(R)=e^{-\eta R}$ in~\eqref{eq:pointwiseaverage}, for a certain $\eta>0$ to be determined later on. In this way, we obtain an expression of the form
\begin{equation}
	\label{eq:Rfunctiondelta}
	\begin{split}
		F_R(\Ga)=&\frac{\text{covol}_{K}(\Ga\cap K)}{\text{covol}_{G}(\Ga)}+\mathscr{E}_{e^{-\eta R}}(R)\\
		&+\text{covol}_{K}(\Ga\cap K)
		\biggl(e^{-\frac{R}{2}} \sum_{\mu\in \text{Spec}(\square),\;\mu>1/4}\beta^+_{\psi_{e^{-\eta R}},\mu}(R)+\beta^-_{\psi_{e^{-\eta R}},\mu}(R)\\
		&+\sum_{\mu\in \text{Spec}(\square),\;0<\mu<1/4}e^{-\frac{1+\nu}{2}R}\beta^{+}_{\psi_{e^{-\eta R}},\mu}(R)+e^{-\frac{1-\nu}{2}R}\beta^{-}_{\psi_{e^{-\eta R}},\mu}(R)\\
		& +\varepsilon_0\biggl(e^{-\frac{R}{2}}\beta^{+}_{\psi_{e^{-\eta R}},1/4}(R)+Re^{-\frac{R}{2}}\beta^{-}_{\psi_{e^{-\eta R}},1/4}(R)\biggr)+\gamma_{\psi_{e^{-\eta R}}}(R)\biggr)\;,
	\end{split}
\end{equation}
which does not depend on the parameter $\delta$ any longer.

Multiplying by $m_{\mathbb{H}}(B_R)$ on both sides of~\eqref{eq:Rfunctiondelta} yields
\begin{equation}
	\label{eq:pointwisecounting}
	\begin{split}
		N(R)=&\frac{\text{covol}_{K}(\Ga\cap K)}{\text{covol}_{G}(\Ga)}m_{\mathbb{H}}(B_R)+\pi(1-2e^{-R}+e^{-2R})e^{R}\mathscr{E}_{e^{-\eta R}}(R)\\
		&+\pi(1-2e^{-R}+e^{-2R})\;\text{covol}_{K}(\Ga\cap K)
		\biggl(e^{\frac{R}{2}} \sum_{\mu\in \text{Spec}(\square),\;\mu>1/4}\beta^+_{\psi_{e^{-\eta R}},\mu}(R)+ \beta^-_{\psi_{e^{-\eta R}},\mu}(R)\\
		&
		+\sum_{\mu\in \text{Spec}(\square),\;0<\mu<1/4}e^{\frac{1-\nu}{2}R}\beta^{+}_{\psi_{e^{-\eta R}},\mu}(R)+e^{\frac{1+\nu}{2}R}\beta^{-}_{\psi_{e^{-\eta R}},\mu}(R)\\
		& +\varepsilon_0\biggl(e^{\frac{R}{2}}\beta^{+}_{\psi_{e^{-\eta R}},1/4}(R)+Re^{\frac{R}{2}}\beta^{-}_{\psi_{e^{-\eta R}},1/4}(R)\biggr)+e^{R}\gamma_{\psi_{e^{-\eta R}}}(R)\biggr)\;.
	\end{split}
\end{equation}

In order to reach an accurate upper bound for the error
\begin{equation*}
	E(R)=\biggl|N(R)-\frac{\text{covol}_{K}(\Ga\cap K)}{\text{covol}_{G}(\Ga)}m_{\mathbb{H}}(B_R)\biggr|\;,
\end{equation*}
in our counting problem, it remains to determine which are the highest-order terms in the expansion~\eqref{eq:pointwisecounting}. To this end, it is relevant to estimate the Sobolev norms of the functions $\psi_{e^{-\eta R}}$ for $R>0$, because of the bounds in~\eqref{eq:betabound} and~\eqref{eq:gammabound}.

\begin{lem}
	\label{lem:growthSobolev}
	For any $0<\delta<1$, the function $\psi_{\delta}$ can be chosen to satisfy 
	\begin{equation}
		\label{eq:decaySobolev}
		\norm{\psi_{\delta}}_{W^{s}}\leq \delta^{-(1+s)}\norm{\psi_1}_{W^{s}}
	\end{equation}
	for any $s>0$.
\end{lem}
\begin{proof}
	Recall that $\psi_\delta$ is assumed to be $K$-invariant or, in other words, a smooth compactly supported function on the two-dimensional manifold $K\bsl G/\Ga$. Since any Riemannian metric on $K\bsl G/\Ga$ is equivalent, on a fixed compact coordinate ball containing the identity coset $Ke\Ga$, to the Euclidean metric on a compact neighborhood of the origin in $\R^{2}$ (cf.~\cite[Lem.~13.28]{Lee}), the problem of constructing $\psi_{\delta}$ so to meet our requirement can be transferred to the Euclidean plane. Specifically, we would like to construct a collection $(\psi_{\delta})_{0<\delta\leq1}$ of mollifiers (cf.~\cite[Sec.~4.4]{Brezis}) so that~\eqref{eq:decaySobolev} is satisfied, where $\norm{\cdot}_{W^{s}}$ are now the standard fractional Sobolev norms on $\R^{2}$. A straightforward computation allows to ascertain that the customary choice 
	\begin{equation*}
		\psi_{\delta}(x)=\frac{1}{\delta^2}\psi_1\biggl(\frac{x}{\delta}\biggr)\;,\quad x\in \R^{2},
	\end{equation*}
	where $\psi_1$ is a fixed compactly supported smooth nonnegative function with unit average over $\R^{2}$, fulfills~\eqref{eq:decaySobolev}. 
\end{proof}

Henceforth, we assume that the collection $(\psi_{\delta})_{0<\delta<1}$ satisfies the condition in Lemma~\ref{lem:growthSobolev}.

We remind the reader that we indicate with $\mu_*$ the spectral gap of the hyperbolic surface $S=\Ga\bsl \Hyp$, that is, the infimum of the set $\text{Spec}(\square)\cap \R_{>0}$. Also, we denote by $\nu_*$ the complex number defined by the properties $\nu_*\in \R_{\geq 0}\cup i\R_{>0}$ and $1-\nu_*^2=4\mu_*$.

Since $e^{\delta}-e^{-\delta}\sim 2\delta$ for $\delta\sim 0$, we deduce from~\eqref{eq:neighboringvalues} that the term $e^{R}\mathscr{E}_{e^{-\eta R}}(R)$ is at most of order $e^{(1-\eta)R}$. On account of Lemma~\ref{lem:growthSobolev}, the highest-order term in the expression 
\begin{equation*}
	\begin{split}
		&e^{\frac{R}{2}} \sum_{\mu\in \text{Spec}(\square),\;\mu>1/4}\beta^+_{\psi_{e^{-\eta R}},\mu}(R)+ \beta^-_{\psi_{e^{-\eta R}},\mu}(R)
		+\sum_{\mu\in \text{Spec}(\square),\;0<\mu<1/4}e^{\frac{1-\nu}{2}R}\beta^{+}_{\psi_{e^{-\eta R}},\mu}(R)\\
		& +e^{\frac{1+\nu}{2}R}\beta^{-}_{\psi_{e^{-\eta R}},\mu}(R)+\varepsilon_0\biggl(e^{\frac{R}{2}}\beta^{+}_{\psi_{e^{-\eta R}},1/4}(R)+Re^{\frac{R}{2}}\beta^{-}_{\psi_{e^{-\eta R}},1/4}(R)\biggr)+e^{R}\gamma_{\psi_{e^{-\eta R}}}(R)
	\end{split}
\end{equation*}
is $e^{\frac{1+\Re{\nu_*}}{2}R}\beta_{\psi_{e^{-\eta R}},\mu_*}(R)$; because of Lemma~\ref{lem:growthSobolev} and~\eqref{eq:betabound}, the latter is at most of order
\begin{equation*}
	e^{\frac{1+\Re{\nu_*}}{2}R}e^{(1+s)\eta R}=e^{\frac{1+\Re{\nu_*}+2(1+s)\eta}{2}R}\;.
\end{equation*}

\begin{rmk}
	\label{rmk:exponentialdecay}
	The reason for choosing $\delta$ to decay exponentially fast with $R$ becomes now apparent: it is the only way to get a sensible comparison between the orders of the two terms considered above. 	
\end{rmk}

Ostensibly the optimal choice of the parameter $\eta$ for our purposes is
\begin{equation*} \eta=\frac{1-\Re{\nu_*}}{2(2+s)}\;,
\end{equation*}
which realizes the equality of exponents
\begin{equation*}
	1-\eta=\frac{1+\Re{\nu_*}+2(1+s)\eta}{2}\;.
\end{equation*}
Bearing in mind that the $K$-invariance of $\psi_{e^{-\eta R}}$ allows to choose $s$ can arbitrarily close to $9/2$ (cf.~Theorem~\ref{thm:expandingonsurface}), it is straightforward to deduce that, setting $\eta_*=\frac{1}{13}(1-\Re{\nu_{*}})$,  we have
\begin{equation*}
	\lim\limits_{R\to\infty}\frac{E(R)}{e^{(1-\eta_{*}+\eps)R}}=0
\end{equation*} 
for any $\eps>0$, which establishes Theorem~\ref{thm:countingproblem}.


\footnotesize

\begin{thebibliography}{99}
	
	\bibitem{Adams}
	R.A.~Adams.
	\newblock {\em Sobolev spaces}.
	\newblock Pure and Applied Mathematics, 65. Academic Press [Harcourt Brace Jovanovich, Publishers], New York-London, 1975.
	
	
	\bibitem{Aubin}
	T.~Aubin.
	\newblock {\em Nonlinear analysis on manifolds. Monge-Amp\`{e}re equations}.
	\newblock Grundlehren der mathematischen Wissenschaften, 252. Springer-Verlag, New York, 1982.
	
	\bibitem{Bekka-Mayer}
	M.B.~Bekka, M.~Mayer.
	\newblock {\em Ergodic theory and topological dynamics of group actions on homogeneous spaces}.
	\newblock London Mathematical Society Lecture Note Series, 269. Cambridge University Press, Cambridge, 2000.
	
	
	\bibitem{Benoist-Oh}
	Y.~Benoist, H.~Oh, \emph{Effective equidistribution of $S$-integral points on symmetric varieties}, Ann.~Inst.~Fourier \textbf{62} (2012), 1889--1942. 
	
	\bibitem{Bergeron}
	N.~Bergeron.
	\newblock {\em The spectrum of hyperbolic surfaces}.
	\newblock Universitext. Springer, Cham; EDP Sciences, Les Ulis, 2016.
	
	\bibitem{Billingsley}
	P.~Billingsley.
	\newblock {\em Convergence of probability measures}. Second edition.
	\newblock Wiley Series in Probability and Statistics: Probability and Statistics. A Wiley-Interscience Publication. John Wiley \& Sons, Inc., New York, 1999. 
	
	\bibitem{Borel}
	A.~Borel.
	\newblock {\em Automorphic forms on $\SL_2(\R)$}.
	\newblock Cambridge Tracts in Mathematics, 130. Cambridge University Press, Cambridge, 1997.
	
	\bibitem{Brezis}
	H.~Brezis.
	\newblock {\em Functional analysis, Sobolev spaces and partial differential equations}.
	\newblock Universitext. Springer New York, NY, 2011.
	
	\bibitem{Bridson-Haefliger}
	M.R.~Bridson, A.~Haefliger.
	\newblock {\em Metric Spaces of Non-Positive Curvature}.
	\newblock Grundlehren der mathematischen Wissenschaften, Springer, Berlin, 1999.
	
	\bibitem{BuFo} 
	A.~Bufetov, G.~Forni, 
	\newblock {\em Limit theorems for horocycle flows}, 
	\newblock Ann. Sci. {\' E}c. Norm. Sup{\' e}r., \textbf{47} (2014), 851--903.
	
	\bibitem{Bur} 
	M.~Burger. 
	\newblock {\em Horocycle flow on geometrically finite surfaces},
	\newblock Duke Math.~J. \textbf{61} (1990), 779--803.
	
	\bibitem{Buser}
	P.~Buser.
	\newblock {\em Geometry and spectra of compact Riemann surfaces}.
	\newblock Progress in Mathematics, 106. Birkh\"{a}user Boston, Inc., Boston, MA, 1992.
	
	\bibitem{Colognese-Pollicott}
	P.~Colognese, M.~Pollicott, \emph{The growth and distribution of large circles on translation surfaces}, {\texttt arXiv:2107.14058} (2021). 
	
	\bibitem{Delsarte}
	J.~Delsarte, \emph{Sur le Gitter Fuchsien}, C.~R.~Acad.~Sci.~Paris \textbf{214} (1942), 147-149.
	
	\bibitem{Denker-Phillip}
	M.~Denker, W.~Philipp \emph{Approximation  by Brownian motion for Gibbs measures and flows under a function}, Ergodic Theory Dynam.~Systems \textbf{4} (1984), 541-552.
	
	\bibitem{Dolgopyat-Sarig}
	D.~Dolgopyat, O.~Sarig,
	\newblock {\em Temporal distributional limit theorems for dynamical systems},
	\newblock J.~Stat.~Phys. \textbf{166} (2017), 680--713, 
	
	\bibitem{Duke-Rudnick-Sarnak}
	W.~Duke, Z.~Rudnick and P.~Sarnak, \emph{Density of integer points on affine homogeneous spaces}, Duke Math.~J. \textbf{71} (1993), 143--179.
	
	\bibitem{Dyatlov-Faure-Guillarmou}
	S.~Dyatlov, F.~Faure and C.~Guillarmou, \emph{Power spectrum of the geodesic flow on hyperbolic manifolds}, Anal.~PDE \textbf{8} (2015), 923--1000.
	
	\bibitem{Edwards-unpublished}
	S.~Edwards.,
	\newblock {\em On the rate of equidistribution of expanding translates of horospheres in $\Gamma\backslash G$},
	\newblock  Comment. Math. Helv., \textbf{96} (2021), 275--337.
	
	\bibitem{Edw}
	S.~Edwards, 
	\newblock {\em On the equidistribution of translates of orbits of symmetric subgroups in $\Ga\bsl G$}, available at \href{https://sites.google.com/view/samedwards}{https://sites.google.com/view/samedwards},
	\newblock  preprint (2018).
	
	\bibitem{Einsiedler-Margulis-Venkatesh}
	M.~Einsiedler, G.~Margulis and A.~Venkatesh, \emph{Effective equidistribution of closed orbits of semisimple groups on homogeneous spaces}, Invent.~Math. \textbf{177} (2009), 137--212.
	
	\bibitem{Einsiedler-Ward}
	M.~Einsiedler, T.~Ward. \emph{Ergodic Theory with a view towards Number Theory}. Graduate Texts in Mathematics, Springer-Verlag, London, 2011. 
	
	
	\bibitem{Eskin-McMullen}
	A.~Eskin,~C.~McMullen, \emph{Mixing, counting and equidistribution in Lie groups}, Duke Math.~J. \textbf{71} (1993), 181--209.
	
	\bibitem{Eskin-Mozes-Shah}
	A.~Eskin, S.~Mozes and N.~Shah, \emph{Unipotent flows and counting lattice points on homogeneous varieties}, Ann.~of Math. \textbf{143} (1996), 253--299.
	
	\bibitem{Farkas-Kra}
	H.~Farkas, I.~Kra.
	\newblock {\em Riemann Surfaces}. Second Edition.
	\newblock Springer-Verlag, Graduate Texts in Math. Vol. 71,
	1980.
	
	\bibitem{Flaminio-Forni} L.~Flaminio, G.~Forni, {\it Invariant distributions and time averages for horocycle flows}, Duke Math.~J. \textbf{119} (2003), 465--536.
	
	\bibitem{Folland}
	G.B.~Folland.
	\newblock {\em A Course in Abstract Harmonic Analysis}. Second Edition.
	\newblock Textbooks in Mathematics, CRC Press, Portland, OR, 2015.
	
	\bibitem{Folland-real}
	G.B.~Folland.
	\newblock {\em Real Analysis}. Second Edition.
	\newblock Pure and Applied Mathematics, John Wiley \& Sons, New York, NY, 1999.
	
	\bibitem{Forni}
	G.~Forni, \emph{Ruelle resonances from cohomological equations}, {\texttt arXiv:2007.03116}  (2020).
	
	
	\bibitem{Gorodnik-Nevo}
	A.~Gorodnik, A.~Nevo.
	\newblock {\em The ergodic theory of lattice subgroups}.
	\newblock Annals of Mathematics Studies, 172. Princeton University Press, Princeton, NJ, 2010.
	
	\bibitem{Hebey}
	E.~Hebey.
	\newblock {\em Nonlinear analysis on manifolds: Sobolev spaces and inequalities}.
	\newblock Courant Lecture Notes in Mathematics, 5. New York University, Courant Institute of Mathematical Sciences, New York. American Mathematical Society, Providence, RI, 1999.
	
	\bibitem{Ivic} 
	A.~Ivi\'{c}, E.~Kr\"{a}tzel, M.~K\"{u}hleitner and W.G.~Nowak. {\it Lattice points in large regions and related arithmetic functions: recent developments in a very classical topic}. Elementare und analytische Zahlentheorie, 89--128, Franz Steiner Verlag Stuttgart, Stuttgart, 2006.
	
	\bibitem{Iwaniec}
	H.~Iwaniec.
	\newblock {\em Spectral methods of automorphic forms}. Second edition.
	\newblock Graduate Studies in Mathematics, 53. American Mathematical Society, Providence, RI; Revista Matem\'{a}tica Iberoamericana, Madrid, 2002.
	
	\bibitem{Hasselblatt-Katok}
	A.~Katok, B.~Hasselblatt, 
	\newblock {\em Introduction to the modern theory of dynamical systems}.
	\newblock Encyclopedia of Mathematics and its Applications, 54. Cambridge University Press, Cambridge, 1995.
	
	\bibitem{Katok}
	S.~Katok.
	\newblock {\em Fuchsian Groups }.
	Chicago Lecture Notes in Mathematics, University of Chicago Press, Chicago, IL, 1992.
	
	\bibitem{Kleinbock-Margulis}
	D.~Kleinbock, G.A.~Margulis, {\it Bounded orbits of nonquasiunipotent flows on homogeneous spaces}, Amer.~Math.~Soc.~Transl. \textbf{171} (1996), 141--172.
	
	\bibitem{Knapp}
	A.W.~Knapp.
	\newblock {\em Lie groups beyond an introduction}. Second Edition. 
	\newblock Progress in Mathematics, Birkh\"{a}user Boston, New York, NY, 2002.
	
	\bibitem{Kra-Shah-Sun}
	B.~Kra, N.A.~Shah and W.~Sun, \emph{Equidistribution of dilated curves on nilmanifolds}, J.~Lond.~Math.~Soc. \textbf{98} (2018), 708--732.
	
	\bibitem{Lang}
	S.~Lang.
	\newblock {\em $\SL_2(\mathbf{R})$}.
	\newblock Graduate Texts in Mathematics, 105. Springer-Verlag, New York, 1985.
	
	\bibitem{Lax-Phillips}
	P.~Lax,~R.~Phillips, \emph{The asymptotic distribution of lattice points in Euclidean and non-Euclidean spaces}, J.~Funct.~Anal. \textbf{46} (1982), 280-350.
	
	
	\bibitem{Lee}
	J.M.~Lee. \textit{Introduction to Smooth Manifolds.} Second Edition. Graduate Texts in Mathematics, 218. Springer, New York, 2013.
	
	\bibitem{Lee-Riemannian}
	J.M.~Lee. \textit{Introduction to Riemannian Manifolds.} Second Edition. Graduate Texts in Mathematics, 176. Springer, Cham, 2018.
	
	\bibitem{LeGall}
	J.-F.~Le Gall. \textit{Brownian motion, martingales, and stochastic calculus.} Graduate Texts in Mathematics, 274. Springer, Cham, 2016.
	
	\bibitem{Livsic}
	A.~Livsic, \emph{Certain properties of the homology of $Y$-systems},
	Mat.~Zametki \textbf{10} (1971), 555--564.	
	
	\bibitem{Mackey}
	G.W.~Mackey.
	\newblock {\em The theory of unitary group representations}.
	\newblock Chicago Lectures in Mathematics. University of Chicago Press, Chicago, Ill.-London, 1976.
	
	\bibitem{Margulis}
	G.A.~Margulis, \emph{Applications of ergodic theory to the investigation of manifolds of negative curvature}, Funct.~Anal.~Appl. \textbf{4} (1969), 335.
	
	\bibitem{Margulis-thesis}
	G.A.~Margulis. \textit{On some aspects of the theory of Anosov systems.} Springer Monographs in Mathematics. Springer-Verlag, Berlin, 2004.
	
	\bibitem{Marklof}
	J.~Marklof. \emph{Selberg's trace formula: an introduction}. Hyperbolic geometry and applications in quantum chaos and cosmology, 83-119. London Math.~Soc.~Lecture Note Ser., 397. Cambridge Univ.~Press, Cambridge, 2012.
	
	\bibitem{Munkres}
	J.R.~Munkres. \emph{Topology}. Second Edition. Prentice Hall, Upper Saddle River, NJ, 2000.
	
	
	\bibitem{Philipp-Stout}
	W.~Philipp, W.~Stout \emph{Almost sure invariance principles for partial sums of weakly dependent random variables}, Mem.~Amer.~Math.~Soc. \textbf{2} no.~161 (1975), iv+140pp.
	
	\bibitem{Phillips-Rudnick}
	R.~Phillips, Z.~Rudnick, \emph{The circle problem in the hyperbolic plane}, J.~Funct.~Anal. \textbf{121} (1994), 78-116.
	
	\bibitem{Raghunathan}
	M.S.~Raghunathan.
	\newblock {\em Discrete subgroups of Lie groups}.
	\newblock
	Ergebnisse der Mathematik und ihrer Grenzgebiete, Band 68.
	Springer-Verlag, New York-Heidelberg, 1972.
	
	\bibitem{Ratcliffe}
	J.G.~Ratcliffe.
	\newblock {\em Foundations of hyperbolic manifolds}. Third Edition.
	\newblock Graduate Texts in Mathematics, 149. Springer, Cham, 2019. 
	
	
	
	\bibitem{Randol}
	B.~Randol, \emph{The behaviour under projection of dilating sets in a covering space}, Trans.~Amer.~Math.~Soc. \textbf{285} (1984), 855--859.
	
	\bibitem{Ratner-CLT} 
	M.~Ratner, 
	\newblock {\em The central limit theorem for geodesic flows on $n$-dimensional manifolds of negative curvature},
	\newblock Israel J.~Math. \textbf{16} (1973), 181--197.
	
	\bibitem{Ratner} 
	M.~Ratner, 
	\newblock {\em The rate of mixing for geodesic and horocycle flows},
	\newblock Ergodic Theory Dynam. Systems \textbf{7} (1987), 267--288.
	
	\bibitem{Rav-arcs} 
	D.~Ravotti, 
	\newblock {\em Quantitative equidistribution of horocycle push-forwards of transverse arcs},
	\newblock Enseign.~Math. \textbf{66} (2020), 135--150.
	
	\bibitem{Rav}
	D.~Ravotti.
	\newblock {\em Asymptotics and limit theorems for horocycle ergodic integrals \`a la Ratner},
	\newblock {\texttt arXiv:2107.02090} (2021).
	
	
	\bibitem{Sarnak}
	P.~Sarnak,
	\emph{Spectra of hyperbolic surfaces}, Bull.~Amer.~Math.~Soc. \textbf{40} (2003), 441--478.
	
	\bibitem{Sasaki}
	S.~Sasaki,
	\emph{On the differential geometry of tangent bundles of Riemannian manifolds},
	Tohoku Math.~J. \textbf{10} (1958), 338--354.
	
	\bibitem{Schwartz}
	L.~Schwartz. \textit{Analyse. Topologie g\'{e}n\'{e}rale et analyse fonctionnelle}. Collection Enseignement des Sciences, No.~11. Hermann, Paris, 1970. 
	
	\bibitem{Selberg-trace}
	A.~Selberg, \emph{Harmonic analysis and discontinuous groups in weakly symmetric Riemannian spaces with applications to Dirichlet series}, J.~Indian Math.~Soc. \textbf{20} (1956), 47--87.
	
	\bibitem{Selberg}
	A.~Selberg, \emph{Equidistribution in discrete groups and the spectral theory of automorphic forms}, \href{https://publications.ias.edu/selberg/section/2491}{https://publications.ias.edu/selberg/section/2491}.
	
	
	\bibitem{Shah}
	N.A.~Shah, \emph{Limit distributions of expanding translates of certain orbits on homogeneous spaces}, Proc.~Indian Acad.~Sci.~Math.~Sci. \textbf{106} (1996), 105--125.
	
	\bibitem{Shah-second}
	N.A.~Shah, \emph{Limiting distributions of curves under geodesic flow on hyperbolic manifolds}, Duke Math.~J. \textbf{148} (2009), 251--279.
	
	
	\bibitem{Shah-third}
	N.A.~Shah, \emph{Equidistribution of expanding translates of curves and Dirichlet's theorem on Diophantine approximation}, Invent.~Math. \textbf{177} (2009), 509--532.
	
	\bibitem{Shah-fourth}
	N.A.~Shah, \emph{Expanding translates of curves and Dirichlet-Minkowski theorem on linear forms}, J.~Amer.~Math.~Soc, \textbf{23} (2010), 563--589.
	
	\bibitem{Sinai}
	Ya.G.~Sinai, \emph{The central limit theorem for geodesic flows on manifolds of constant negative curvature}, Soviet Math.~Dokl., \textbf{1} (1960), 938--987.
	
	\bibitem{Strassen}
	V.~Strassen, \emph{An invariance principle for the law of the iterated logarithm}, Z.~Wahrscheinlichkeitstheorie und Verw.~Gebiete \textbf{3} (1964), 211--226.
	
	
	\bibitem{Strassen-two}
	V.~Strassen, \emph{Almost sure behaviour of sums of independent random variables and martingales}. Proc.~Fifth Berkeley Sympos.~Math.~Statist.~and Probability (Berkeley, Calif., 1965/66), Vol.~II: Contributions to Probability Theory, Part 1, pp.~315-343. Univ.~California Press, Berkeley, Calif., 1967.
	
	\bibitem{Strichartz-Sobolev}
	R.S.~Strichartz,
	\newblock {\em Analysis of the Laplacian on the complete Riemannian manifold},
	\newblock J.~Funct.~Anal. \textbf{52} (1983), 48--79, 
	
	
	\bibitem{Strichartz}
	R.S.~Strichartz,
	\newblock {\em Magnified curves on a flat torus, determination of almost periodic functions, and the Riemann-Lebesgue lemma},
	\newblock Proc.~Amer.~Math.~Soc. \textbf{107} (1989), 755--759, 
	
	
	\bibitem{Strombergsson-closedhorocycles}
	A.~Str\"{o}mbergsson,
	\newblock {\em On the uniform equidistribution of long closed horocycles},
	\newblock Duke Math.~J. \textbf{123} (2004), 507--547, 
	
	\bibitem{Str}
	A.~Str\"{o}mbergsson,
	\newblock {\em On the deviation of ergodic averages for horocycle flows},
	\newblock J.~Mod.~Dyn. \textbf{7} (2013), 291--328.
	
	\bibitem{Terras}
	A.~Terras.
	\newblock {\em Harmonic analysis on symmetric spaces -- Euclidean space, the sphere, and the Poincar\'{e} upper half plane}. Second edition.
	\newblock Springer, New York, 2013.
	
	\bibitem{Thurston}
	W.P.~Thurston. \emph{The Geometry and Topology of Three-Manifolds}. Electronic version. \href{https://www.msri.org/publications/books/gt3m}{https://www.msri.org/publications/books/gt3m}, 2002.
	
	\bibitem{Triebel}
	H.~Triebel.
	\newblock {\em Theory of function spaces II}. 
	\newblock Monographs in Mathematics, 84. Birkh\"{a}user Verlag, Basel, 1992.
	
	\bibitem{Venkatesh}
	A.~Venkatesh,
	\newblock {\em Sparse equidistribution problems, period bounds and subconvexity},
	\newblock Ann.~of.~Math. \textbf{172} (2010), 989--1094.
	
	\bibitem{PYang}
	P.~Yang, \emph{Equidistribution of expanding translates of curves and Diophantine approximation on matrices}, Invent.~Math. \textbf{220} (2020), 909--948.
	
\end{thebibliography}
\end{document}